\documentclass[11pt]{article}
\usepackage[T1]{fontenc}
\usepackage{latexsym,amssymb,amsmath,amsfonts,amsthm}
\usepackage{graphics}
\usepackage{graphicx}
\usepackage{mathrsfs}
\usepackage{subfigure}
\usepackage{color}
\usepackage{tikz}
\usepackage{caption}
\newtheorem{theorem}{Theorem}[section]
\newtheorem{lemma}[theorem]{Lemma}

\newtheorem{definition}[theorem]{Definition}
\newtheorem{remark}[theorem]{Remark}

\numberwithin{equation}{section} % This is needed!
\numberwithin{figure}{section}   % This is optional.
\numberwithin{table}{section}    % This is optional.

\definecolor{hw}{rgb}{1,0,0}

\usepackage[colorlinks,linkcolor=red,anchorcolor=green,citecolor=red,CJKbookmarks=True,bookmarks=true]{hyperref}
\usepackage[titletoc,title]{appendix}
\begin{document}
\renewcommand{\theequation}{\arabic{section}.\arabic{equation}}
\title{\bf A Nystr\"{o}m Method for Scattering by a Two-layered Medium with a Rough Boundary}
\author{
Haiyang Liu\thanks{LMAM, School of Mathematical Sciences, Peking University, Beijing 100871, China ({\tt haiyang.l@pku.edu.cn})} \quad
Long Li\thanks{Academy of Mathematics and Systems Science, Chinese Academy of Sciences, Beijing 100190, China ({\tt longli@amss.ac.cn})}
\quad
Jiansheng Yang\thanks{LMAM, School of Mathematical Sciences, Peking University, Beijing 100871, China ({\tt jsyang@pku.edu.cn})} \quad
Bo Zhang\thanks{SKLMS and Academy of Mathematics and Systems Science, Chinese Academy of Sciences, Beijing, 100190, China and School of Mathematical Sciences, University of Chinese Academy of Sciences, Beijing 100049, China ({\tt b.zhang@amt.ac.cn})}
\quad
Haiwen Zhang\thanks{Corresponding author. SKLMS and Academy of Mathematics and Systems Science, Chinese Academy of Sciences,
Beijing 100190, China ({\tt zhanghaiwen@amss.ac.cn})}
}

\date{}

\maketitle
\begin{abstract}
This paper is concerned with problems of scattering of time-harmonic acoustic waves by a two-layered medium with
a non-locally perturbed boundary (called a rough boundary in this paper) in two dimensions, where a Dirichlet or
impedance boundary condition is imposed on the boundary. The two-layered medium is composed of two unbounded media
with different physical properties and the interface between the two media is considered to be a planar surface.
We formulate the scattering problems considered as boundary value problems and prove the result of the
well-posedness of each boundary value problem by utilizing the integral equation method associated with
the two-layered Green function.
Moreover, we develop a Nystr\"{o}m method for numerically solving the boundary value problems considered,
based on the proposed integral equation formulations.
We establish the convergence results of the Nystr\"{o}m method with the convergence rates depending on the smoothness
of the rough boundary. It is worth noting that in establishing the well-posedness
of the boundary value problems as well as the convergence results of the Nystr\"{o}m method,
an essential role is played by the investigation of the asymptotic properties of the two-layered Green function
for small and large arguments.
Finally, numerical experiments are carried out to show the effectiveness of the Nystr\"{o}m method.

\vspace{.2in}
\end{abstract}

\section{Introduction}
This paper is concerned with the well-posedness and the numerical method for the problems of
scattering of time-harmonic acoustic waves in a two-layered medium in two
dimensions.
The two-layered medium is composed of two unbounded media with
different physical properties
and the interface between the two media
is considered to be a planar surface.
The boundary of the two-layered medium is assumed to be a \textit{rough surface}, which is a non-local perturbation with a finite height from a planar surface.
Such scattering problems occur in various scientific and engineering applications, such as ground-penetrating radar, seismic exploration, ocean exploration, photonic crystal, and diffraction by gratings. For an introduction and historical remarks, we refer to \cite{C1999, V1999, D2002, P1980, WC2001, SS2001}.

There are many works concerning
the well-posedness of
the rough surface scattering problems for acoustic waves.
The rough surface scattering problems with Dirichlet or impedance boundary conditions
have been studied in
\cite{CZ1998, ZC2003, CR1996, CRZ1999} by using the integral equation methods.
In each of these works, the layer potential technique was applied to transform the scattering problem into an equivalent boundary integral equation.
\cite{ZC1998, CZ1998a, CZ1999}
considered the rough surface scattering problems by penetrable interfaces and inhomogeneous layers,
using the
integral equation methods.
In \cite{CM2005, CE2010},
the authors studied
the rough surface scattering problem with a sound-soft boundary by employing the variational approach in the classical Sobolev space or the weighted Sobolev space.
Moreover,
the method in \cite{CM2005}
was extended in \cite{HLQZ2015}
to study the scattering problem by an inhomogeneous layer of a finite height, where the Neumann or generalized impedance boundary condition
was imposed on the lower boundary of the inhomogeneous layer.
For more works on the well-posedness of the rough surface scattering problems for electromagnetic or elastic waves, we refer to \cite{EH2012, EH2015, HL2011, LWZ2011, HLZ2020}.

Some numerical methods have also been developed for the rough surface scattering problems.
In \cite{MCK2000}, the authors introduced
the Nystr\"{o}m method for
the second-kind integral equation defined on the real line.
Based on this, numerical algorithms were proposed for the rough surface scattering problems; see \cite{MCK2000} for the sound-soft case and \cite{LSZ2016} for the penetrable case.
An adaptive finite element method with a perfectly matched layer (PML) was proposed in \cite{CW2003} for the wave scattering by periodic structures.
In \cite{ZZ2018b}, the authors proposed the Nystr\"{o}m method for the scattering problem by penetrable diffraction gratings. In this method, a fast FFT-based algorithm developed in \cite{ZZ2018} was utilized for efficient computation
of the quasi-periodic Green's functions.
In \cite{CM2009}, the authors
investigated the use of the PML to truncate the rough surface scattering problem and proved the linear rate of convergence for the proposed PML-based method.

In this paper, we consider the scattering problems in a two-layered medium, where the Dirichlet or impedance boundary condition is imposed on the rough boundary.
First, we formulate the considered scattering problems as
the boundary value problems and
prove the uniqueness and existence results of each boundary value problem by utilizing the integral equation method associated with the two-layered Green function.
Our proofs follow the ideas in \cite{CRZ1999,CZ1998,ZC2003}, which are based on an integral equation theory on unbounded domains given in \cite{CZ1997}.
We note that different from \cite{CRZ1999,CZ1998,ZC2003}, in this paper we use the two-layered Green function (rather than the half-space Dirichlet Green function or the half-space impedance Green function) in the proposed integral equation formulations, which is due to the presence of the two-layered medium with the planar interface.
It is also worth noting that in the proofs of the uniqueness and existence results of this paper, an essential role
is played by the investigation of the asymptotic properties of the two-layered Green function for small and large arguments.
Second, based on the proposed integral equation formulations, we develop the
Nystr\"{o}m method for numerically solving the considered boundary value problems, where the relevant integral equations are discretized by using the method given in \cite{MCK2000}.
With the aid of the convergence theory of the Nystr\"{o}m method given in \cite{MCK2000}, we establish the convergence results of our method with the convergence rates depending on the smoothness of the rough boundary.
It should be noted that the asymptotic properties of the two-layered Green function obtained in this paper provide a theoretical foundation for our convergence results.
Finally, numerical experiments are carried out to show the effectiveness of our Nystr\"{o}m method.

The rest of the paper is organized as follows. In Section \ref{section2}, we introduce the considered scattering problems and formulate them as the boundary value problems.
In Section \ref{section8}, we present the properties of the two-layered Green function.
Based on these properties,
we establish the well-posedness of the considered boundary value problems in
Section \ref{section3}.
Section \ref{section6} is devoted to the Nystr\"{o}m method for the considered boundary value problems. The convergence results
and the numerical experiments
of the Nystr\"{o}m method are also
given in Section \ref{section6}.
Some concluding remarks are given in Section \ref{section13}.
We prove Lemma \ref{thm34} and Theorem \ref{thm28} in Appendix \ref{section14}.
In Appendixes \ref{section7} and \ref{section9}, we present the potential theory and the solvability of integral operators on the real line, respectively,
associated with the two-layered Green function.

\section{Mathematical Models of the Scattering Problems}
\label{section2}
In this section, we introduce the mathematical models of the scattering problems considered in this paper.
To this end, we give some notations, which will be used throughout the paper.
Let $V\subset \mathbb{R}^{m}$ ($m = 1, 2$). We denote by $BC(V)$ the set of functions bounded and continuous on $V$, a Banach space under the norm $\|\phi\|_{\infty,V} : = \sup_{x\in V}  |\phi(x)|$, and by $BUC(V)$ the closed subspace of $BC(V)$ containing functions that are bounded and uniformly continuous on $V$. We abbreviate $\|\cdot\|_{\infty,\mathbb{R}}$ by $\|\cdot\|_{\infty}$. For $0<\alpha\leq 1$, we denote by $C^{0,\alpha}(V)$ the Banach space of functions $\phi \in BC(V)$, which are uniformly H\"{o}lder continuous with exponent $\alpha$ and with norm $\|\cdot\|_{C^{0,\alpha}(V)}$ defined by
$\|\phi\|_{C^{0,\alpha}(V)}:=\|\phi\|_{\infty,V}+\sup_{x,y\in V,x\neq y} [|\phi(x)-\phi(y)|/|x-y|^{\alpha}]$.
We let $C^{1,\alpha}(\mathbb{R}):=\{\phi \in BC(\mathbb{R})\cap C^{1}(
    \mathbb{R}) : \phi'\in C^{0,\alpha}(\mathbb{R}) \}$ be a Banach space under the norm $\|\phi\|_{C^{1,\alpha}(\mathbb{R})}:=\|\phi\|_{\infty}+\|\phi'\|_{C^{0,\alpha}(\mathbb{R})}$.
For any $a\in \mathbb{R}$, define $\Gamma_a := \{(x_1,a) \,:\, x_1\in \mathbb{R}\}$ and $U_a:=\left\{(x_1,x_2) \,:\, x_1\in \mathbb{R}, x_2>a\right\}$. In particular, the notation $\Gamma_{0}$ denotes the plane $x_2=0$. Let $\mathbb{R}_{\pm}^2:=\{(x_1,x_2)\in \mathbb{R}^2\,:\,x_2 \gtrless 0\}$ be the upper and lower half-spaces, respectively. For any $x,y\in \mathbb{R}^2$, let $x=(x_1,x_2)$ and $y=(y_1,y_2)$. For any $x\in \mathbb{R}^2$ with $x\neq 0$, let $\hat{x}:=x/|x|$ denote the direction of $x$.
Define $\mathbb{S}^1_\pm:=\{x=(x_1,x_2)\in\mathbb{R}^2:|x|=1,x_2\gtrless 0\}$.
Let $C(V)$ represent the space of continuous functions on $V$ and let $C^i(V)$ represent the space of $C^i$-continuous functions on $V$ for $i=1,2$.
Throughout this paper, the constants may be different at different places.

The geometry of the scattering problems we consider is shown in Figure \ref{fig1}.
Let $\mathbb{R}_{+}^2$ and $\mathbb{R}_{-}^2$ denote the homogeneous media above and below
$\Gamma_0$, respectively.
The wave numbers of the media in the upper and lower half-spaces are $k_{+}$ and $k_{-}$, respectively, with $k_{+},k_{-} > 0 $ and $k_+\neq k_-$. Define $n:=k_{-}/k_{+}$.
Assume that a rough surface $\Gamma:=\left\{(x_1,x_2)\,:\,x_2=f(x_1),x_1 \in \mathbb{R}\right\}$ is fully embedded in the lower half-space $\mathbb{R}^2_{-}$,
where $f\in C^{1,1}(\mathbb{R})$ with $f_{+}:=\sup_{x\in \mathbb{R}}f(x) < 0$.
Let $f_{-}:=\inf_{x\in \mathbb{R}}f(x)$ and the Lipschitz constant $L := \|f'\|_{\infty}$.
Define the domain $D:=\{ (x_1,x_2) \,:\, x_2>f(x_1)\}$.

\begin{figure}[ht]
    \centering
    \includegraphics[width=.5\textwidth]{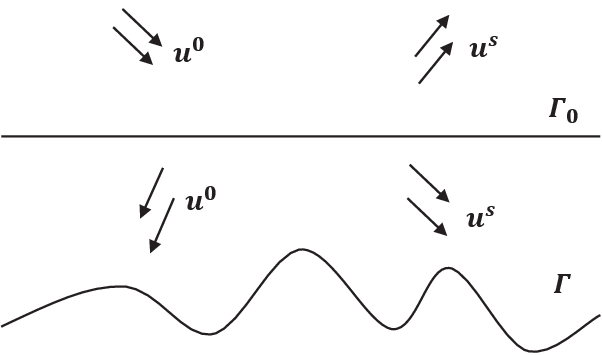}
    \caption{Geometry of the scattering problems}
    \label{fig1}
\end{figure}

Consider the scattering problems with time-harmonic incident waves in the domain $D$.
In this paper, we assume that the incident wave $u^{i}$ is either a plane wave or a point-source wave.
The reference wave $u^{0}$ is generated by the incident wave $u^{i}$ and the two-layered medium. The explicit expressions of the incident wave and its corresponding reference wave will be described later.
Then the total field $u^{tot}=u^0+u^s$ is the sum of the reference wave $u^0$ and the scattered wave $u^s$, where $u^s$ satisfies the following Helmholtz equations
\begin{align}\label{eq2}
    \left\{ \begin{aligned}
    &\Delta u^{s}+k_{+}^2u^{s}=0 &&\text{ in }\mathbb{R}^2_{+}, \\
    &\Delta u^{s}+k_{-}^2u^{s}=0 &&\text{ in }\mathbb{R}^2_{-}\cap D.
    \end{aligned} \right.
\end{align}
Moreover, we assume the total field $u^{tot}$ satisfies the following boundary conditions on the interface $\Gamma_{0}$, i.e.,
\begin{equation}
u^{tot}\big|_{+}=u^{tot}\big|_{-},\quad \left.\frac{\partial u^{tot}} {\partial x_2}\right|_{+} = \left.\frac{\partial u^{tot} }{\partial x_2}\right|_{-} \text{ on }\Gamma_{0},
\label{eq172}
\end{equation}
where '+/-' denote the limits from $\mathbb{R}^2_+$ and
$\mathbb{R}^2_-$, respectively.
Furthermore, the boundary condition imposed on $\Gamma$ is given by $\mathscr{B}(u^{tot})=0$ on $\Gamma$. Here, $\mathscr{B}$ denotes one of the following two boundary conditions:
\begin{equation*}
\begin{cases}\mathscr{B}(u^{tot}):=u^{tot} \quad &\text{ on }\Gamma,\quad \text{ if }\Gamma \text{ is a sound-soft boundary,}\\
\mathscr{B} (u^{tot}):=\partial u^{tot} / \partial \nu - ik_{-}\beta u^{tot} \quad & \text{ on }\Gamma, \quad \text{ if }\Gamma \text{ is an impedance boundary,}
\end{cases}
\end{equation*}
where $\beta\in BC(\Gamma)$, $\nu(x)$ denotes the unit normal at $x\in \Gamma$ pointing out of $D$ and $\partial u^{tot}/\partial \nu$ denotes the normal derivative of $u^{tot}$ with respect to $\nu$.

To guarantee the uniqueness of the considered scattering problems, the scattered wave $u^s$ is required to satisfy a radiation condition. In contrast to the bounded obstacle scattering problems, which utilize the Sommerfeld radiation condition, $u^s$ needs to satisfy the so-called upward propagating radiation condition in $U_{0}$ with respect to $k_{+}$, that is,
\begin{equation}
    u^s(x) =2\int_{\Gamma_{h}}\frac{\partial \Phi_{k_{+}}(x,y)}{\partial y_2}\phi(y)ds(y),\quad x\in U_{h},
    \label{eq1}
\end{equation}
for some $h>0$ and $\phi\in L^{\infty}(\Gamma_{h})$, where $\Phi_{k_{+}}(x,y):=\frac{i}{4}H_{0}^{(1)}(k_{+}|x-y|)$ with $x,y\in \mathbb{R}^2$ and $x\neq y$ is the free-space Green function for the Helmholtz equation $\Delta w+k^2_+ w=0$ with the wave number $k_{+}$ and $H_{0}^{(1)}(t)$ with $t\in \mathbb{R}$ denotes the Hankel function of the first kind of order zero.
We also need $u^s$ to satisfy the following boundedness condition
\begin{equation}
    \sup_{x\in D}\big|(x_2+|f_{-}|+1)^{\alpha}u^s(x)\big|<\infty
    \label{eq3}
\end{equation}
for some $\alpha\in \mathbb{R}$.

Furthermore, if $\Gamma$ is an impedance boundary, the scattered wave $u^s$ needs to satisfy that for some $\theta\in (0,1)$ and some constant $C_{\theta}>0$,
\begin{equation}\label{eq4}
\left|\nabla u^s(x)\right|\leq C_{\theta}[x_2-f(x_1)]^{\theta-1}
\end{equation}
for $x\in D\backslash \overline{U_{b}}$, where $b=f_{+}/2$.

Now we describe the reference wave $u^0$ more specifically. The reference wave is the total field of the scattering problem in the two-layered medium without the rough surface $\Gamma$ and is generated by the incident wave $u^i$.
In this paper, we consider two types of incident waves, that is, the plane wave and the point-source wave.

First, we describe the reference wave in the case when the incident wave $u^i$ is the plane wave $u_{pl}^i(x):= e^{ik_{+}x\cdot d}$, where
$d := (\cos(\theta_d),\sin(\theta_d))$, $\theta_d\in (\pi ,2\pi)$.
In this case, the reference wave $u^0=u^{0}_{pl}$ is given by (see, e.g., (2.13a) and (2.13b) in \cite{P2016} or Section 4 in \cite{LYZZ2022})
\begin{equation}\label{hw-eq17}
    u^0_{pl}(x) = \left\{
\begin{array}{ll}
u^i_{pl}(x)+u^r_{pl}(x) & \text{ in }\mathbb{R}^2_{+}, \\
u^t_{pl}(x) & \text{ in }\mathbb{R}^2_{-},
\end{array}
\right.
% \label{eq57}
\end{equation}
with
\begin{equation*}
    u^r_{pl}(x) := \mathcal{R}(\pi+\theta_{d}) e^{ik_{+}x\cdot d^r},\quad
u^t_{pl}(x) :=\mathcal{T}(\pi+\theta_{d}) e^{ik_{-}x\cdot d^t},
\end{equation*}
where $d^r = (\cos(\theta_d),-\sin(\theta_d))$ is the reflected direction, $d^t=n^{-1}(\cos (\theta_d),-i\mathcal{S}(\cos \theta_d,n))$
and where
$\mathcal{R}\left(\pi+\theta_{d}\right)$ and $\mathcal{T}\left(\pi+\theta_{d}\right)$ are called the reflection and transmission coefficients, respectively, with $\mathcal{R}$ and $\mathcal{T}$ defined by
\begin{equation*}
    \mathcal{R}(\theta):=\frac{i \sin \theta+\mathcal{S}(\cos \theta, n)}{i \sin \theta-\mathcal{S}(\cos \theta, n)}, \quad \mathcal{T}(\theta):=\mathcal{R}(\theta)+1 \quad \text { for } \theta \in \mathbb{R}.
\end{equation*}
Here, $S(z,a)$ with $z\in \mathbb{R}$ and $a>0$ is defined by
\begin{equation*}
    \mathcal{S}(z,a):=\left\{\begin{array}{ll}
        -i\sqrt{a^2-z^2} & \text{ if } a^{-1}|z|\leq 1,\\
        \sqrt{z^2-a^2} & \text{ if } a^{-1}|z| > 1.
    \end{array}\right.
    % \label{eq71}
\end{equation*}
The definition of $\mathcal{S}$ gives that
\begin{equation*}
    d^t= \begin{cases}\left(n^{-1} \cos \theta_d,-\sqrt{1-\left(n^{-1} \cos \theta_d\right)^2}\right) & \text { if } n^{-1}\left|\cos \theta_d\right| \leq 1, \\ \left(n^{-1} \cos \theta_d,-i \sqrt{\left(n^{-1} \cos \theta_d\right)^2-1}\right) & \text { if } n^{-1}\left|\cos \theta_d\right|>1.\end{cases}
    \end{equation*}
In particular, if $|n^{-1}\cos (\theta_d)|\leq 1$, then $d^{t}=(\cos(\theta_d^{t}),\sin(\theta_d^{t}))$ is the transmitted direction with $\theta_d^t\in [\pi,2\pi]$ satisfying $\cos(\theta_d^t)=n^{-1}\cos(\theta_d)$.
It is easy to see that such reference wave $u^0$ satisfies the following conditions
\begin{align}\label{hw-eq16}
    \left\{ \begin{aligned}
    &\Delta u^0+k_{+}^2u^{0}=0 &&\text{ in }\mathbb{R}^2_{+}, \\
    &\Delta u^0+k_{-}^2u^{0}=0 &&\text{ in }\mathbb{R}^2_{-},\\
    &u^0_{+}=u^{0}_{-},\quad\left.\frac{\partial u^{0}}{\partial x_2}\right|_{+} =\left.\frac{\partial u^{0} }{\partial x_2}\right|_{-} &&\text{ on }\Gamma_{0}. \end{aligned} \right.
\end{align}

Second, we describe the reference wave in the case when the incident wave $u^i$ is the point source wave $u_{ps}^{i}(x)$, where $u_{ps}^{i}(x)=\frac{i}{4}H_{0}^{(1)}(k_+ |x-y|)$ if the source point $y\in \mathbb{R}^2_+$ and $u_{ps}^{i}(x)=\frac{i}{4}H_{0}^{(1)}(k_- |x-y|)$ if the source point $y\in D\cap\mathbb{R}^2_-$. In this case, the reference wave $u^0(x) = G(x,y)$, where $G(x,y)$ denotes the so-called two-layered Green function.
Precisely,
for any $y\in\mathbb{R}^2_{+}\cup \mathbb{R}^2_{-}$, the two-layered Green function $G(x,y)$ is the solution of the following scattering problem (see page 17 in \cite{P2016})
\begin{align}
    &\Delta_{x} G(x,y)+(k(x))^2G(x,y)=-\delta(x-y) &&\text{in }\mathbb{R}^2,\label{eq166} \\
   & [G(x,y)]=0,\; [\partial G(x,y)/\partial \nu(x)]=0 &&\text{on }\Gamma_{0},\\
&\lim\limits_{|x|\to\infty}\sqrt{|x|}\left(\frac{\partial G(x,y)}{\partial |x|}-ik(x)G(x,y)\right)=0 &&\text{uniformly for all } \hat{x} \in \mathbb{S}^1_+\cup\mathbb{S}^1_-, \label{eq22}
\end{align}
where $k(x)=k_{\pm}$ for $x\in \mathbb{R}^2_{\pm}$, $\delta$ denotes the Dirac delta distribution, $\nu$ denotes the unit normal on $\Gamma_0$ pointing into $\mathbb{R}^2_{+}$ and $[\cdot]$ denotes the jump across the interface $\Gamma_0$.
Here, (\ref{eq22}) is called the Sommerfeld radiation condition.
The explicit expression of $G(x,y)$ is given by (see, e.g., \cite[formula (2.27)]{P2016})
\begin{equation}
    G(x,y)= \begin{cases}
       \frac{i}{4}H_{0}^{(1)}(k_{+}|x-y|)+G_{\mathcal{R}}(x,y),\quad x,y\in \mathbb{R}^2_{+},\\
       G_{\mathcal{Q}}(x,y),\quad x\in \mathbb{R}^{2}_{+}, y\in \mathbb{R}^2_{-} \text{ or }x\in \mathbb{R}^{2}_{-}, y\in \mathbb{R}^2_{+},\\
       \frac{i}{4}H_{0}^{(1)}(k_{-}|x-y|)+G_{\mathcal{R}}(x,y),\quad x,y\in \mathbb{R}^2_{-},
    \end{cases}
    \label{eq63}
\end{equation}
where $G_{\mathcal{R}}(x,y)$ and $G_{\mathcal{Q}}(x,y)$ are given by
\begin{equation}\label{hw-eq9}
    G_{\mathcal{R}}(x,y)=\begin{cases}
        \frac{1}{4 \pi} \int_{-\infty}^{\infty} \frac{\mathcal{S}\left(\xi, k_{+}\right)-\mathcal{S}\left(\xi, k_{-}\right)}{\mathcal{S}\left(\xi, k_{+}\right)+\mathcal{S}\left(\xi, k_{-}\right)} \frac{e^{-\mathcal{S}\left(\xi, k_{+}\right)\left|x_2+y_2\right|}}{\mathcal{S}\left(\xi, k_{+}\right)} e^{i \xi\left(x_1-y_1\right)} d \xi,\quad x,y\in \mathbb{R}^2_{+}, \\
        \frac{1}{4 \pi} \int_{-\infty}^{\infty} \frac{\mathcal{S}\left(\xi, k_{-}\right)-\mathcal{S}\left(\xi, k_{+}\right)}{\mathcal{S}\left(\xi, k_{+}\right)+\mathcal{S}\left(\xi, k_{-}\right)} \frac{e^{-\mathcal{S}(\xi, k_{-})\left|x_2+y_2\right|}}{\mathcal{S}\left(\xi, k_{-}\right)} e^{i \xi\left(x_1-y_1\right)} d \xi,\quad x,y\in \mathbb{R}^2_{-},
    \end{cases}
\end{equation}
\begin{equation}\label{hw-eq11}
    G_{\mathcal{Q}}(x,y)=\frac{1}{2 \pi} \int_{-\infty}^{\infty} \frac{e^{-\mathcal{S}\left(\xi, k_{-}\right) |y_2|-\mathcal{S}\left(\xi, k_{+}\right) |x_2|}}{\mathcal{S}\left(\xi, k_{+}\right)+\mathcal{S}\left(\xi, k_{-}\right)} e^{i \xi\left(x_1-y_1\right)} d \xi,\quad x\in \mathbb{R}^{2}_{+}, y\in \mathbb{R}^2_{-} \text{ or }x\in \mathbb{R}^{2}_{-}, y\in \mathbb{R}^2_{+}.
\end{equation}

Now the above scattering problems can be formulated as the following two boundary value problems (DBVP) and (IBVP) for the scattered wave $u^s$.

\begin{definition}[$\mathscr{R}_{d}(D)$]
    Let $\mathscr{R}_{d}(D)$ denote the set of functions $v\in C^{2}(D\backslash \Gamma_{0})\cap C(\overline{D})$ such that $v|_{\overline{U}_{0}}\in C^{1}(\overline{U}_{0})$ and $v|_{D\backslash U_{0}}\in C^{1}(D\backslash U_{0})$.
\end{definition}

\textbf{Dirichlet Boundary Value Problem (DBVP).}
    Given $g\in BC(\Gamma)$, determine  $u^s \in \mathscr{R}_{d}(D)$ such that:

(\romannumeral1) $u^s$ is a solution of the Helmholtz equations in \eqref{eq2};

(\romannumeral2) $u^s|_{+}=u^s|_{-},\partial u^s/ \partial x_2\left|_{+}=\partial u^s/ \partial x_2\right|_{-}$ on $\Gamma_{0}$;

(\romannumeral3) $u^s = g$ on $\Gamma$;

(\romannumeral4) $u^s$ satisfies (\ref{eq3}) for some $\alpha\in\mathbb{R}$;

(\romannumeral5) $u^s$ satisfies the upward propagating radiation condition (\ref{eq1}) in $U_{0}$ with the wave number $k_{+}$.

\begin{definition}[$\mathscr{R}_{i}(D)$]
Let $\mathscr{R}_{i}(D)$ denote the set of functions $v \in C^2(D\backslash\Gamma_{0})\cap C(\overline{D})$ satisfying $v|_{\overline{U}_{0}}\in C^{1}(\overline{U}_{0})$, $v|_{D\backslash U_{0}}\in C^{1}(D\backslash U_{0})$ and
satisfying that the normal derivative of $v$ defined by $\partial v / \partial \nu (x): = \lim_{h\to 0+}\nu(x)\cdot \nabla(x-h\nu(x))$
exists uniformly for $x$
on any compact subset of $\Gamma$.
\end{definition}

\textbf{Impedance Boundary Value Problem (IBVP).} Given $g \in BC(\Gamma),\beta \in BC(\Gamma)$, determine  $u^s \in \mathscr{R}_{i}(D)$ such that:

(\romannumeral1) $u^s$ is a solution of the Helmholtz equations in \eqref{eq2};

(\romannumeral2) $u^s|_{+}=u^s|_{-},\partial u^s/ \partial x_2\left|_{+}=\partial u^s/ \partial x_2\right|_{-}$ on $\Gamma_{0}$;

(\romannumeral3) $\partial u^s/\partial \nu - ik_{-}\beta u^s=g$ on $\Gamma$;

(\romannumeral4) $u^s$ satisfies (\ref{eq3}) for some $\alpha\in\mathbb{R}$;

(\romannumeral5) For some $\theta\in (0,1)$ and some constant $C_{\theta}>0$,
$u^s$ satisfies (\ref{eq4})
for $x\in D\backslash \overline{U_{b}}$, where $b=f_{+}/2$;

(\romannumeral6) $u^s$ satisfies the upward propagating radiation condition (\ref{eq1}) in $U_0$ with the wave number $k_{+}$.

\begin{remark}
We note that if $u^s$ is the scattered wave of the scattering problem (\ref{eq2})--(\ref{eq3})
associated with the sound-soft boundary $\Gamma$ (resp. the scattering problem (\ref{eq2})--(\ref{eq4}) associated with the impedance boundary $\Gamma$),
then $u^s$ satisfies the problem (DBVP) with
$g=-u^0|_\Gamma\in BC(\Gamma)$ (resp. the problem (IBVP) with $g = -\partial u^{0}/\partial\nu|_\Gamma + ik_{-}\beta u^{0}|_\Gamma\in BC(\Gamma)$), where $u^0$ is given as above.
% \label{rk1}
\end{remark}

\section{Properties of the Two-layered Green Function}
\label{section8}
In this section, we present some properties of the two-layered Green function $G(x,y)$, which are useful for the investigation of the well-posedness of the considered boundary value problems and the convergence of the Nystr\"{o}m method in the following two sections.

Let the two-layered Green function $G(x,y)$ with $y\in\mathbb{R}^2_+\cup\mathbb{R}^2_-$ be given as in Section \ref{section2}. For any source point $y$ lying on the interface $\Gamma_0$, due to the well-posedness of
the scattering problem in a two-layered medium (see \cite{BHY2018}), we can define the two-layered Green function $G(x,y)$ as the unique solution that satisfies $G(\cdot,y)-G_0(\cdot,y)\in H^{1}_{loc}(\mathbb{R}^2)$, $\Delta_x G(x,y)+k^2(x)G(x,y)=-\delta(x,y)$ in $\mathbb{R}^2$ (in the distributional sense) and the Sommerfeld radiation condition \eqref{eq22}, where $G_{0}(\cdot,y):=-1/(2\pi)\ln|\cdot -y|$ denotes the fundamental solution of the Laplace equation $\Delta w=0$ in $\mathbb{R}^2$.
Here, $H^{1}_{loc}(\mathbb{R}^2)$ denotes the space of all functions $\phi:\mathbb{R}^2\to\mathbb{C}$ such that $\phi \in H^1(B)$ for all open balls $B\subset \mathbb{R}^2$. Moreover, by the expression of the Hankel function $H^{(1)}_{0}(t)$ given in \cite[Section 3.5]{CK2019} and the expression of $G(x,y)$ given in \eqref{eq63}, it is clear that for any $y\in \mathbb{R}^2_{+}\cup \mathbb{R}^2_{-}$, $G(x,y)$ also satisfies $G(\cdot,y)-G_{0}(\cdot,y)\in H^{1}_{loc}(\mathbb{R}^2)$.

Let $x,y\in \mathbb{R}^2$ with $x=(x_1,x_2),y=(y_1,y_2)$. For any $y=(y_1,y_2)$, let $y':=(y_1,-y_2)$.
Using the following integral representation of Hankel function (see \cite[formula (2.2.11)]{C1999})
    \begin{equation}
        \frac{i}{4}H_{0}^{(1)}(\kappa|x-y|) = \frac{1}{4\pi}\int_{-\infty}^{+\infty}\frac{e^{-\mathcal S(\xi,\kappa)|x_{2}-y_{2}|}}{\mathcal{S}(\xi,\kappa)}e^{i\xi(x_1-y_1)}d\xi \label{eq156}
    \end{equation}
    for $\kappa>0$, $x,y\in \mathbb{R}^2$ with $x\neq y$, the formula (\ref{eq63}) for $G(x,y)$ can be written as
    \begin{equation}\label{hw-eq12}
        G(x,y)= \begin{cases}
           G_{\mathcal{D},k_{+}}(x,y)+G_{\mathcal{P}}(x,y),\quad x,y\in \mathbb{R}_{+}^{2}, \\
           G_{\mathcal{Q}}(x,y),\quad x\in \mathbb{R}^{2}_{+}, y\in \mathbb{R}^2_{-} \text{ or }x\in \mathbb{R}^{2}_{-}, y\in \mathbb{R}^2_{+}, \\
           G_{\mathcal{D},k_{-}}(x,y)+G_{\mathcal{P}}(x,y),\quad x,y\in \mathbb{R}_{-}^{2},
        \end{cases}
        % \label{eq721}
    \end{equation}
    where $G_{\mathcal{D},\kappa}(x,y)$ is the half-space Dirichlet Green function for $\kappa>0$ (see \cite{CR1996, ZC2003}) and is defined as
    \begin{equation*}
        G_{\mathcal{D},\kappa}(x,y) := \frac{i}{4}H_{0}^{(1)}(\kappa|x-y|)-\frac{i}{4}H_{0}^{(1)}(\kappa|x-y'|)
    \end{equation*}
    for $x,y\in \mathbb{R}^2$ with $x\notin \{y,y'\}$
    and where $G_{\mathcal{P}}$ is given by
        \begin{equation}
    G_{\mathcal{P}}(x,y) :=\begin{cases}
        \frac{1}{2\pi}\int_{-\infty}^{+\infty}\frac{1}{\mathcal{S}(\xi,k_{+})+\mathcal{S}(\xi,k_{-})}e^{-\mathcal{S}(\xi,k_{+})|x_{2}+y_{2}|}e^{i\xi(x_1-y_1)}d\xi, \quad x,y\in \mathbb{R}^{2}_{+},\\
        \frac{1}{2\pi}\int_{-\infty}^{+\infty}\frac{1}{\mathcal{S}(\xi,k_{+})+\mathcal{S}(\xi,k_{-})}e^{-\mathcal{S}(\xi,k_{-})|x_{2}+y_{2}|}e^{i\xi(x_1-y_1)}d\xi,\quad x,y\in \mathbb{R}^{2}_{-}.
    \end{cases}  \label{eq126}
        \end{equation}
Further, with the help of (\ref{eq156}), we write $G_{\mathcal{Q}}(x,y)$ as
\begin{equation*}
    G_{\mathcal{Q}}(x,y) =  \frac{i}{4}H_{0}^{(1)}(k_{+}|x-y|)+G_{\mathcal{S}}(x,y) \quad \text{for } x\in \mathbb{R}^{2}_{-},y\in \mathbb{R}^{2}_{+} \text{ or }x\in  \mathbb{R}^{2}_{+},y\in \mathbb{R}^{2}_{-},
\end{equation*}
where $G_{\mathcal{S}}(x,y)$ is defined by
\begin{equation}\label{hw-eq13}
    G_{\mathcal{S}}(x,y) := \frac{1}{4\pi}\int_{-\infty}^{+\infty} \bigg(\frac{2e^{-\mathcal{S}\left(\xi, k_{-}\right) |y_2|-\mathcal{S}\left(\xi, k_{+}\right) |x_2|}}{\mathcal{S}\left(\xi, k_{+}\right)+\mathcal{S}\left(\xi, k_{-}\right)} e^{i \xi\left(x_1-y_1\right)}-\frac{e^{-\mathcal{S}(\xi,k_{+})|x_2-y_2|}}{\mathcal{S}(\xi,k_{+})}e^{i\xi (x_1-y_1)}\bigg)d\xi
\end{equation}
for $x\in \mathbb{R}^{2}_{-},y\in \mathbb{R}^{2}_{+} \text{ or }x\in  \mathbb{R}^{2}_{+},y\in \mathbb{R}^{2}_{-}.$

The following lemma presents the continuity properties of $G(x,y)$.
The proof of this lemma is given in Appendix \ref{section14}.

\begin{lemma}For any $k_{+},k_{-}>0$ with $k_{+}\neq k_{-}$, we have $R(x,y):=G(x,y)-G_{0}(x,y)\in C^{1}(\mathbb{R}^2\times\mathbb{R}^2)$.
\label{thm34}
\end{lemma}

\begin{remark}By Lemma \ref{thm34}, $ G_{\mathcal{R}}(x,y) $ can be extended as a function in $ C^1(\overline{\mathbb{R}^2_{+}} \times \overline{\mathbb{R}^2_{+}}) \cup C^1(\overline{\mathbb{R}^2_{-}} \times \overline{\mathbb{R}^2_{-}}) $ and $ G_{\mathcal{S}}(x,y) $ can be extended as a function in $ C^1(\overline{\mathbb{R}^2_{+}} \times \overline{\mathbb{R}^2_{-}}) \cup C^1(\overline{\mathbb{R}^2_{-}} \times \overline{\mathbb{R}^2_{+}}) $.\label{thm35}
\end{remark}

With the aid of Lemma \ref{thm34}, Remark \ref{thm35} and some far-field asymptotic properties of the two-layered Green function obtained in \cite{LYZZ2022}, we have the following theorem on the estimates of $G_{\mathcal{P}}(x,y)$ and $G_{\mathcal{Q}}(x,y)$.
The proof of this theorem is also given in Appendix \ref{section14}.

\begin{theorem}
    Assume that $k_{+},k_{-}>0$ with $k_{+}\neq k_{-}$. Let $x = (x_1, x_2)\in\mathbb{R}^2$ and $y = (y_1,y_2)\in\mathbb{R}^2$. Define $y':=(y_1,-y_2)$ and $\widetilde{x-y}:=(x_1-y_1,x_2)$. Then we have the following statements.

    {\rm (i)} If $x,y$ satisfy $x_2\cdot y_2 > 0$, then $G_{\mathcal{P}}(x,y)$ satisfies the inequalities
    \begin{equation*}
        \left|G_{\mathcal{P}}(x,y)\right|,~\left|\nabla_{y} G_{\mathcal{P}}(x,y)\right| \leq C (1+|x_2|+|y_2|) |x-y'|^{-\frac{3}{2}},
         % \label{eq142}
    \end{equation*}
where the constant $C$ depends only on $k_{\pm}$.

    {\rm (ii)} If $x,y$ satisfy $x_2\cdot y_2<0$ and $|y_2|\leq h$ for some $h>0$, then $G_{\mathcal{Q}}(x,y)$ satisfies the inequalities
    \begin{equation*}
        \left|G_{\mathcal{Q}}(x,y)\right|,~\left|\nabla_{y} G_{\mathcal{Q}}(x,y)\right| \leq C (1+|x_2|)\big|\widetilde{x-y}\big|^{-\frac{3}{2}},
        % \label{eq143}
    \end{equation*}
    where the constant $C$ depends only on $k_{\pm}$ and $h$.
    \label{thm28}
\end{theorem}

Similar properties as in Theorem \ref{thm28} have been established for the half-space Dirichlet Green function and the half-space impedance Green function (see \cite[inequalities (8) and (24)]{ZC2003}). Especially, we mention that $G_{\mathcal{D},\kappa}$ satisfies the estimates (see \cite[Formula (8)]{ZC2003})
\begin{equation}
    |\nabla_{y} G_{\mathcal{D},\kappa}(x, y)|,~|G_{\mathcal{D},\kappa}(x, y)| \leq C\left(1+\left|x_2\right|\right)\left(1+\left|y_2\right|\right)\left\{|x-y|^{-3 / 2}+\left|x-y^{\prime}\right|^{-3 / 2}\right\} \label{eq161}
\end{equation}
for $x,y\in \mathbb{R}^2$ with $x\notin\{y,y'\}$, where the constant $C$ depends only on $\kappa>0$.

Finally, as a direct consequence of \eqref{hw-eq12}, \eqref{eq161}, Lemma \ref{thm34} and Theorem \ref{thm28}, we can obtain the following theorem on the estimates of $G(x,y)$
(especially the asymptotic estimates of $G(x,y)$ for small and large arguments), which is crucial for this paper.

\begin{theorem}\label{thm38}
Assume that $k_{+},k_{-}>0$ with $k_{+}\neq k_{-}$. Let $x = (x_1, x_2)\in\mathbb{R}^2$ and $y = (y_1,y_2)\in\mathbb{R}^2$. Define $y':=(y_1,-y_2)$ and $\widetilde{x-y}:=(x_1-y_1,x_2)$. Then we have the following statements.

{\rm (i)} If $x,y$ satisfy $x_2\cdot y_2 \geq 0$, then $G(x,y)$ satisfies the inequalities
    \begin{equation}
\left|G(x,y)\right|,~\left|\nabla_{y} G(x,y)\right| \leq C (1+|x_2|)(1+|y_2|) \left\{|x-y|^{-\frac{3}{2}}+|x-y'|^{-\frac{3}{2}}\right\}
         \quad \text{ for } x\neq y,y',\label{eq191}
    \end{equation}
where the constant $C$ depends only on $k_{\pm}$.

{\rm (ii)} If $x,y$ satisfy $x_2\cdot y_2<0$ and $|y_2|\leq h$ for some $h>0$, then $G(x,y)$ satisfies the inequalities
    \begin{equation}
\left|G(x,y)\right|,~\left|\nabla_{y} G(x,y)\right| \leq C (1+|x_2|)\big|\widetilde{x-y}\big|^{-\frac{3}{2}},\label{eq192}
    \end{equation}
    where the constant $C$ depends only on $k_{\pm}$ and $h$.
\end{theorem}

\section{The Well-posedness of the Problems (DBVP) and (IBVP)}
\label{section3}
In this section, we consider the well-posedness of the problems (DBVP) and (IBVP). In Section \ref{section3a}, we provide some a priori estimates of the first derivatives of relevant solutions. Then following the ideas in \cite{CRZ1999,CZ1998,ZC2003}, we prove the uniqueness results for the problems (DBVP) and (IBVP) in Sections \ref{section3b} and \ref{section3c}, respectively. Furthermore, the existence results for the problems (DBVP) and (IBVP) are given in Sections \ref{section10} and \ref{section11}, respectively.

\subsection{The Derivative Estimates}
\label{section3a}
If $u^s \in \mathscr{R}_{d}(D)$ satisfies the conditions (i)--(iv) of the problem (DBVP) with $g =0$, we can apply the standard elliptic regularity estimate \cite[Theorem 8.34]{GT1983} to deduce that $u^s \in C^{1}(\overline{D})$.
Let $L^{\infty}(G)$ denote the space of essentially bounded functions defined on a set $G\subset \mathbb{R}^2$.
Then the following lemma presents the local regularity estimate of solutions to the Laplace equation.

\begin{lemma}[Lemma 2.7 in \cite{CZ1998a}]
   If $G \subset \mathbb{R}^2$ is open and bounded, $v \in L^{\infty}(G)$, and $\Delta v=f \in L^{\infty}(G)$ (in a distributional sense), then $v \in C^1(G)$ and
\begin{equation*}
|\nabla v(x)| \leq C(d(x))^{-1}\left(\|v\|_{\infty,G}+\left\|d^2 f\right\|_{\infty,G}\right), \quad x \in G,
\end{equation*}
where $C$ is an absolute constant and $d(x)=\operatorname{dist}(x, \partial G)$.
\label{thm1}
\end{lemma}

Using the formula (\ref{eq3}) and Lemma \ref{thm1} with the domain $G$ to be a sufficiently small ball centered at $x$, we can obtain the following theorem. See \cite[formula (3.1)]{CZ1998} for a similar result.

\begin{theorem}
        If $u^s \in \mathscr{R}_{d}(D)$ satisfies the conditions {\rm (\romannumeral1)--(\romannumeral4)} of the problem {\rm (DBVP)} or $u^s \in \mathscr{R}_{i}(D)$ satisfies the conditions {\rm (\romannumeral1)--(\romannumeral4)} of the problem {\rm (IBVP)}, then there exists some $\alpha \in \mathbb{R}$ such that
    \begin{equation*}
       \sup_{x_1\in \mathbb{R},x_2> f(x_1)+\epsilon } \big|(x_2+|f_{-}|+1)^{\alpha}\nabla u^s(x)\big| < \infty
    \end{equation*}
    for all $\epsilon>0$.
    \label{thm2}
\end{theorem}

Moreover, by similar arguments as in the proof of \cite[Theorem 3.1]{CZ1998}, we have the following estimates on the solution satisfying the conditions (i)--(iii) of the problem (DBVP) with $g=0$.

\begin{theorem}
    Let $\epsilon:=|f_{+}|/2$ and $1/2<\alpha<\pi/(2\pi-\theta)<1$ with $\theta:=\pi-2\arctan(L)$. If $u^s \in \mathscr{R}_{d}(D)$ satisfies the conditions (i)--(iii) of the problem (DBVP) with $g=0$, then we have that for some positive constant C,
    \begin{align}
    &|u^s(x)|\leq C[x_2-f(x_1)]^{\alpha}, \label{p1-eq8} \\
    &|\nabla u^s(x)|\leq C[x_2-f(x_1)]^{\alpha-1}, \label{p1-eq9}
    \end{align}
    for $x\in \{x=(x_1,x_2)\in \mathbb{R}^2\,:\,x_1 \in \mathbb{R},f(x_1)<x_2<f(x_1)+\epsilon\}$.
    \label{thm3}
\end{theorem}

\begin{proof}
Let $E = 2|f_{+}|/3$ and let $\Omega$ be given by
\[
\Omega := \left\{ (x_1, x_2) : |x_1| < E/L,\, -L|x_1| < x_2 < E \right\}.
\]
Define $w \in C^2(\Omega) \cap C(\overline{\Omega})$ such that $\Delta w = k_{-}^2$ in $\Omega$ and $w = h$ on $\partial \Omega$, where $h \in C(\partial \Omega)$ is chosen such that $-1 \leq h \leq 0$, $h = -1$ on
\[
 \mathcal{I}_1 := \left\{ x : |x_1| = E/L,\, - E \leq x_2 \leq E \right\} \cup \left\{ x : |x_1| \leq E/L,\, x_2 = E \right\}
\]
and $h = 0$ on $ \mathcal{I}_2 := \left\{ x : |x_1| \leq E/L,\, x_2 = -L|x_1| \right\}$. By the elliptic singularity theory (see \cite{Grisvard1975}), there exists some $K > 0$ such that
\begin{equation}
\label{p1-eq7}
|w(x)| \leq K|x|^\alpha, \quad x \in \overline{\Omega}.
\end{equation}
Furthermore, by the maximum principle, we have $w \leq 0$ in $\overline{\Omega}$.

Let $b = f_+ + E$ and $C = \sup_{x \in D\backslash {U_E}  } |u^s(x)|$. For $x \in \Gamma$, let $\Omega_x = \Omega + x := \{ y + x : y \in \Omega \}$ and define $w_x \in C(\overline{\Omega_x}) \cap C^2(\Omega_x)$ by $w_x(y) = Cw(y - x)$ for $y \in \overline{\Omega_x}$. Let $v$ denote either the real or the imaginary part of $u^s$ and let $V = D \cap \Omega_x$. Then $v \in C(\overline{V}) \cap C^2(V)$, $|\Delta v| \leq Ck_{-}^2$, $|v| \leq C$ in $V$, and $v = 0$ on $\Gamma \cap \partial V$. Moreover, $\Delta w_x = Ck_{-}^2$ in $V$, $w_x = -C$ on $\partial V \cap D$, and $w_x \leq 0$ on $\Gamma \cap \partial V$. Define $v_\pm = \pm v + w_x$. Then $v_\pm \leq 0$ on $\partial V$ and $\Delta v_\pm \geq 0$ in $V$, and so, by the maximum principle, $v_\pm \leq 0$ in $\overline{V}$. Consequently, $|v| \leq -w_x$ in $\overline{V}$. Thus it follows from the equation \eqref{p1-eq7} that for $0 \leq h \leq E$ and $x\in\Gamma$,
\begin{align}
|v(x + h e_2)| \leq -w_x(x + h e_2) = -Cw(0, h) \leq CKh^\alpha, \label{p1-eq11}
\end{align}
where $e_2:=(0,1)$.
For $r > 0$, define the set
\[
D_{r} := \{x = (x_1, x_2) \in \mathbb{R}^2: x_1 \in \mathbb{R},\, f(x_1) < x_2 < f(x_1) + r\}.
\]
Then we can apply \eqref{p1-eq11} and Lemma \ref{thm1} to obtain that there exists $C_1 > 0$ such that
\begin{equation}
|u^s(x)| \leq C_1 |x_2 - f(x_1)|^\alpha, \quad x \in D_{2|f_{+}|/3}.\label{p1-eq10}
\end{equation}
This implies that \eqref{p1-eq8} holds true.

On the other hand, by Lemma \ref{thm1}, we have that for $x \in D_{|f_{+}|/2}$,
\begin{equation*}
|\nabla u^s(x)| \leq \tilde{C} [\eta(x)]^{-1} (1 + k_{-}^2) \sup_{y \in B_{\eta(x)}(x)} |u^{s}(y)|,
\end{equation*}
where $\eta(x) := \min(|f_{+}|/6, d(x)/2)$ with $d(x) := \operatorname{dist}(x, \Gamma)$. Using the equation \eqref{p1-eq10} and noting that
\begin{equation*}
(1 + L^2)^{-1/2} \leq \frac{d(x)}{x_2 - f(x_1)} \leq 1, \quad x \in D,
\end{equation*}
we obtain that for $x \in D_{|f_{+}|/2}$,  
\begin{align*}
&\sup_{y \in B_{\eta(x)}(x)} |u^s(y)| \leq \sup_{y \in B_{\eta(x)}(x)} C_{1} |y_2-f(y_1)|^{\alpha}
\\
&\leq \sup_{y \in B_{\eta(x)}(x)} C_{1} \left[x_2 - f(x_1) + |y_2-x_2| + L|x_1-y_1|\right]^{\alpha}\\
& \leq C_{1} \left(x_2 - f(x_1) +(1+L^2)^{1/2} \eta(x)\right)^{\alpha} \leq C_{1} (1+L^2)^{\alpha/2} \left(\frac{3d(x)}{2}\right)^{\alpha}.
\end{align*}
Hence, it follows that for $x \in D_{|f_{+}|/2}$, 
\begin{align*}
& |\nabla u^s(x)| \leq \tilde{C} C_1 [\eta(x)]^{-1} (1 + k_{-}^2) (1+L^2)^{\alpha/2} \left( \frac{3d(x)}{2} \right)^\alpha \\
&\leq \tilde{C} C_1 \max [ (|f_{+}|/6)^{-1}, (d(x)/2)^{-1}]  (1 + k_{-}^2) (1+L^2)^{\alpha/2} \left( \frac{3d(x)}{2} \right)^\alpha \\
& \leq \tilde{C} C_1 \max [ ((x_2-f(x_1))/3)^{-1}, ((1+L^2)^{-1/2}(x_2-f(x_1))/2)^{-1}]  (1 + k_{-}^2) (1+L^2)^{\alpha/2} \left( \frac{3d(x)}{2} \right)^\alpha \\ 
& \leq 2 \tilde{C} C_1 \max [ 3, (1+L^2)^{1/2}]  (1 + k_{-}^2)  (1+L^2)^{\alpha/2} \left(\frac{3}{2}\right)^{\alpha} (x_2-f(x_1))^{\alpha-1}  \\
&\leq  C (x_2 - f(x_1))^{\alpha - 1},
\end{align*}
where $C = 2 \tilde{C} C_1 \max [ 3, (1+L^2)^{1/2}]  (1 + k_{-}^2) (1+L^2)^{\alpha/2}   \left({3}/{2}\right)^{\alpha} $.
Therefore, the proof is complete.
\end{proof}

\subsection{The Uniqueness Result of the Problem (DBVP)}
\label{section3b}
In this subsection, we prove the uniqueness of the problem (DBVP) with the help of the a priori estimates given in Section \ref{section3a}.
We introduce some notations, which will be used in the rest of this paper.
For $a\in \mathbb{R}$ and $B,A\in \mathbb{R}$ with $B<A$, define $\Gamma_{a}(B, A) := \{x=(x_1,a)\in \mathbb{R}^2\,:\,B<x_1<A\}$ and $\Gamma(B,A) := \{x=(x_1,x_2)\in \mathbb{R}^2 \,:\,B<x_1<A,x_2=f(x_1)\}$.
For $t,a\in \mathbb{R}$ with $a>f_{+}$, define $\gamma_{a}(t):=\left\{x=(x_1,x_2)\in \mathbb{R}^2\,:\,x_1=t, f(x_1) <x_2 < a\right\}$.
Given an open set $V\subset\mathbb{R}^2$ and
$v\in L^\infty(V)$, let $\partial_j v$ ($j=1, 2$) denote the (distributional) derivative $\partial v(x)/\partial x_j$ and
we abbreviate $\partial v/\partial\nu$ (that is, the normal
derivative of $v$) as $\partial_{\nu}v$.

The following theorem presents an inequality for the solution of the problem (DBVP) with $g = 0$, which plays a crucial role in the proof of the uniqueness result.

\begin{theorem}
Assume $k_{+}> k_{-}$. Let
$u^s\in \mathscr{R}_{d}(D)$ be the solution of the problem (DBVP) with $g = 0$. Let $a>0$ and $B,A\in \mathbb{R}$ with $B<A$.  Then we have
    \begin{equation}
\int_{\Gamma(B,A)}\left|\partial_{\nu} u^s \right|^2ds \leq  C\left(I_{1}(B,A)+R_1(B,A)\right),\label{eq7}
\end{equation}
where $\nu$ denotes the unit normal on $\Gamma$ pointing out of  $D$ and where
$I_{1}(B,A)$ and $R_{1}(B,A)$
are given by
\begin{align*}
I_{1}(B,A):=&\int_{\Gamma_{a}(B,A)}\left(|\partial_{2}u^s|^2-|\partial_1u^s|^2+k_{+}^2|u^s|^2\right)ds,\\
    R_{1}(B,A) :=&\,  2\mathrm{Re}\left(\int_{\gamma_{a}(A)}-\int_{\gamma_{a}(B)}\right)\overline{\partial_{2}u^s}\partial_1 u^sds.
\end{align*}
Here, $C$ is a constant depending only on $\Gamma$.
\end{theorem}

\begin{proof}
Define $T(B,A) := \left\{x\in  D\backslash \overline{U}_{0}\,:\,B<x_1<A\right\}$ for $B<A$ and let $\partial T(B,A)$ be the boundary of $T(B,A)$. Let $\nu = (\nu_1,\nu_2)$ denote the outward unit normal on $\partial T(B,A)$.
Noting that Rellich's type identity $2\mathrm{Re}[\overline{\partial_2 u^s}(\Delta u^s+k_{-}^2u^s)] =2\mathrm{Re}[\nabla\cdot (\overline{\partial_2 u^s}\nabla u^s)]-\partial_2(|\nabla u^s|^2)+k_{-}^2\partial_2(|u^s|^2)$, we find, by applying the divergence theorem in $T(B,A)$, that
\begin{equation}
    0=\int_{\partial T(B,A)}\left(2 \mathrm{Re}(
        \overline{\partial_2u^s}\nabla u^s)\cdot  \nu - |\nabla u^s|^2\nu_2+k_{-}^2|u^s|^2\nu_2\right)ds=L_1+L_2+L_3,\label{eq201}
\end{equation}
where
$L_1$, $L_2$ and $L_3$ are given by
\begin{align*}
    L_1:&= \int_{\Gamma_{0}(B,A)}\left(|\partial_2 u^s|^2-|\partial_{1}u^s|^2+k_{-}^2|u^s|^2\right)ds,\\
    L_2:&= 2\mathrm{Re}\left(\int_{\gamma_{0}(A)}-\int_{\gamma_{0}(B)}\right)\overline{\partial_2 u^s}\partial_1 u^sds, \\
    L_3:&=\int_{\Gamma(B,A)}\left(2 \mathrm{Re}(
        \overline{\partial_2u^s}\nabla u^s)\cdot  \nu - |\nabla u^s|^2\nu_2+k_{-}^2|u^s|^2\nu_2\right)ds,
\end{align*}
Furthermore, by using the identity $2\mathrm{Re}[\overline{\partial_2 u^s}(\Delta u^s+k_{+}^2u^s)] =2\mathrm{Re}[\nabla\cdot (\overline{\partial_2 u^s}\nabla u^s)]-\partial_2(|\nabla u^s|^2)+k_{+}^2\partial_2(|u^s|^2)$ in the domain $(B,A)\times(0,a)$ and the fact that $k_{+}>k_{-}$, we obtain that
\begin{equation}
L_1\leq \int_{\Gamma_{0}(B,A)}\left(|\partial_2 u^s|^2-|\partial_{1}u^s|^2+k_{+}^2|u^s|^2\right)ds=L_4+I_{1}(B,A),\label{eq202}
\end{equation}
where
\begin{equation*}
L_4:= 2\mathrm{Re}\left(\int_{\gamma_{a}(A)\backslash \gamma_{0}(A)}-\int_{\gamma_{a}(B)\backslash \gamma_{0}(B)}\right)\overline{\partial_2 u^s}\partial_1 u^sds.
\end{equation*}
Thus, combining \eqref{eq201} and \eqref{eq202}, we have that  $0\leq I_1(B,A)+R_{1}(B,A)+L_3$.

It follows from Theorems \ref{thm2} and \ref{thm3} that the integral $R_1(B,A)$ is well-defined and thus we have
\begin{equation}
    \sup_{t\in \mathbb{R}}\left|2\mathrm{Re}\int_{\gamma_{a}(t)}\partial_2\overline{u^s}\partial_1u^sds\right|<\infty.
    \label{eq8}
\end{equation}
By the boundary condition of $u^s$, we have $\nabla u^s(x) = \partial_{\nu}u^s\cdot \nu(x)$ on $\Gamma$. Thus we can deduce that $|\nabla u^s|^2 = |\partial_{\nu}u^s|^2$ and $\partial_2 u^s(x) = \partial_{\nu}u^s \nu_2$ on $\Gamma$. This, together with the fact that $\nu_2 = -1/\sqrt{1+|f'|^2} \leq -1/\sqrt{1+L^2}$ on $\Gamma$, implies that
\begin{equation*}
   L_3 = \int_{\Gamma(B,A)}|\partial_{\nu}u^s|^2\nu_2ds \leq -1/\sqrt{1+L^2} \int_{\Gamma(B,A)}|\partial_{\nu}u^s|^2 ds.
\end{equation*}

 Therefore, from the above discussions, it follows that \eqref{eq7} holds. This completes the proof.
\end{proof}

\begin{remark}
Taking $A=j$ and $B = j-1$ with $j\in \mathbb{Z}$ in \eqref{eq7}, we can apply Theorem \ref{thm2} as well as the formulas (\ref{eq3}) and (\ref{eq8}) to obtain that $\partial_\nu u^s$ in $L_{loc}^2(\Gamma)$, where $L_{loc}^2(\Gamma)$ denotes the space of all functions $g:\Gamma\rightarrow\mathbb{C}$ such that $g\in L^{2}(\Gamma(B,A))$ for all $B,A\in\mathbb{R}$ with $B<A$. Moreover, we have
\begin{equation}
    \sup_{j\in \mathbb{Z}}\int_{\Gamma(j-1,j)}|\partial_\nu u^s|^2ds<\infty.
    \label{eq9}
\end{equation}
\label{thm37}
\end{remark}

Next, we show that the solution of the problem (DBVP) with $g=0$ can be written as an integral relevant to its normal derivative on $\Gamma$.
For this purpose, we define $\Gamma(A):=\Gamma(-A,A)=\left\{x=(x_1,x_2):x\in \Gamma, |x_1|<A\right\}$ for $A>0$ and introduce the following definition.

\begin{definition}Given a domain $G \subset \mathbb{R}^2$ and $\kappa>0$, call $v \in C^2(G) \cap L^{\infty}(G)$  a radiating solution of the Helmholtz equation in $G$ if $\Delta v+\kappa^2 v=0$ in $G$ and
\begin{align*}
v(x)=O\left(r^{-1 / 2}\right)&, \\
\frac{\partial v(x)}{\partial r}-i \kappa v(x)=o\left(r^{-1 / 2}\right)&
\end{align*}
as $r=|x| \rightarrow \infty$, uniformly in $x /|x|$.
\end{definition}

\begin{theorem}
    Let $u^s\in \mathscr{R}_{d}(D)$ be the solution of the problem (DBVP) with $g =0$. Then
    \begin{equation}\label{eq143}
        u^s(x) =  \int_{\Gamma}\partial_{\nu}u^s(y)G(x,y)\,ds(y),\quad x\in D,
        \end{equation}
    where $\nu$ denotes the unit normal on $\Gamma$ pointing out of $D$.
    \label{thm6}
\end{theorem}

\begin{proof}
First, we consider the case when $x=(x_1,x_2) \in D\backslash \overline{U}_{0}$. Let $A>0$ and define the domain
    \begin{equation}
        T_{A}^{\epsilon} := \left\{x\,:\,|x_1|<A, x\in D\backslash \overline{U}_{0} \right\}\backslash \overline{B_{\epsilon}(x)},\label{eq206}
    \end{equation}
    where $B_{\epsilon}(x)$ denotes the ball centered at $x$ with radius $\epsilon$ small enough such that $\overline{B_{\epsilon}(x)}\subset D\backslash \overline{U_0}$.
    Since $u^s\in C^{1}(\overline{D}\backslash U_{|f_{+}|/2})$, it follows from Green's theorem that
    \begin{align*}
      0 &= \int_{\partial T_{A}^{\epsilon}}\left(u^s(y)\frac{\partial G(x,y)}{\partial \nu(y)}-\frac{\partial u^s}{\partial \nu}(y)G(x,y) \right)ds(y) \\ &=\left( \int_{S_{\epsilon}(x)}+\int_{\gamma_{0}(-A)}+\int_{\gamma_{0}(A)} +\int_{\Gamma_{0}(A)}+\int_{\Gamma(A)}\right) \left(u^s(y)\frac{\partial G(x,y)}{\partial \nu(y)}-\frac{\partial u^s}{\partial \nu}(y)G(x,y) \right)ds(y) \\
      &=:L_1+L_2+L_3+L_4+L_5,
    \end{align*}
    where $\nu$ denotes the outward unit normal on $\partial T_{A}^{\epsilon}$.
    By the mean value theorem and the formula (\ref{eq63}), we obtain that
    \begin{equation*}
        \lim_{\epsilon\to 0+}L_1=\lim_{\epsilon\to 0+}\int_{\partial B_{\epsilon}(x)}\left(u^s(y)\frac{\partial G(x,y)}{\partial \nu(y)}-\frac{\partial u^s}{\partial \nu}(y)G(x,y)\right)ds(y)= u^s(x).
    \end{equation*}
    By the estimates in (\ref{eq191}) as well as Theorems \ref{thm2} and \ref{thm3}, it follows that
    \begin{equation}
        \lim_{A\to +\infty}(L_2+L_3)=\lim_{A\to \infty}\left(\int_{\gamma_{0}(-A)}+\int_{\gamma_{0}(A)}\right)\left( u^s(y)\frac{\partial G(x,y)}{\partial \nu(y)}-\frac{\partial u^s}{\partial \nu}(y)G(x,y) \right)ds(y) = 0.\label{eq203}
    \end{equation}
    Using the transmission boundary conditions of $u^s(x)$ and $G(x,y)$ on the interface $\Gamma_{0}(A)$, we obtain
    \begin{align*}
        L_4&=\int_{\Gamma_{0}(A)}\left( u^s|_{-}(y)\frac{\partial G(x,y)}{\partial y_2}-\frac{\partial u^s}{\partial y_2}\big|_{-}(y) G(x,y) \right)ds(y)  \\&=
        \int_{\Gamma_{0}(A)}\left( u^s|_{+}(y)\frac{\partial G(x,y)}{\partial y_2}-\frac{\partial u^s}{\partial y_2}\big|_{+} (y) G(x,y) \right)ds(y),
    \end{align*}
    where '+/-' are the limits given as in \eqref{eq172}.
    With the help of Theorems \ref{thm38} and \ref{thm2}, we can apply Green's theorem in the domain $\left\{x\,:\,|x_1|<A,0<x_{2}<d\right\}$ with $d>0$ to obtain that
    \begin{align*}
         \lim_{A\to +\infty} \int_{\Gamma_{0}(A)}\left( u^s|_{+}(y)\frac{\partial G(x,y)}{\partial y_2}-\frac{\partial u^s}{\partial y_2}\big|_{+} (y) G(x,y) \right)ds(y)
       \notag \\
        =\int_{\Gamma_{d}}\left( u^s(y)\frac{\partial G(x,y)}{\partial y_2}-\frac{\partial u^s}{\partial y_{2}}(y)G(x,y) \right)ds(y).
        % \label{eq204}
    \end{align*}
    From the definition of the two-layered Green function given in \eqref{eq166}--\eqref{eq22} and the estimates in \eqref{eq192}, together with the symmetry property  $G(x_0,y_0)=G(y_0,x_0)$ for $x_0,y_0\in \mathbb{R}^2\backslash\Gamma_0$ with $x_0\neq y_0$ (see \cite[(2.28)]{P2016}), we have that $G(x,\cdot)$ is a radiating solution of $\Delta G(x,\cdot)+k_{+}^2 G(x,\cdot)=0$ in $U_{0}$ and that $G(x,\cdot)|_{\Gamma_{d}}$ and $\partial_{y_2} G(x,\cdot)|_{\Gamma_{d}}$ belong to $L^1(\Gamma_{d})$. Note that $u^s$ satisfies the upward propagating radiation condition (\ref{eq1}) in $U_{0}$. Hence we can employ \cite[Lemma 2.1]{CZ1998} to obtain
    \begin{equation}  \lim_{A\to \infty} L_4 =\int_{\Gamma_{d}}\left( u^s(y)\frac{\partial G(x,y)}{\partial \nu(y)}-\frac{\partial u^s}{\partial \nu}(y)G(x,y) \right)ds(y) = 0.\label{eq205}
    \end{equation}
    From the facts that $u^s=0$ on $\Gamma$ and $u^s\in C^{1}(\overline{D}\backslash U_{|f_{+}|/2})$, together with (\ref{eq9}) and the estimates in (\ref{eq191}), we can deduce that
    \begin{equation*}
      \lim_{A\to \infty} L_5 =  -\lim_{A\to \infty}\int_{\Gamma(A)}\frac{\partial u^s}{\partial \nu}(y)G(x,y)ds(y)=-\int_{\Gamma}\frac{\partial u^s}{\partial \nu}(y)G(x,y) ds(y).
        \end{equation*}
    By using the above discussions, we obtain that the formula (\ref{eq143}) holds for $x\in D\backslash \overline{U}_{0}$.

    Second, by the dominated convergence theorem and the above discussions, it easily follows that the formula (\ref{eq143}) holds for $x\in \Gamma_{0}$.
    Moreover, using similar arguments as above, we can deduce that \eqref{eq143} also holds for $x \in U_{0}$.
\end{proof}

To proceed further, we need the following three lemmas. Lemma \ref{thm5} can be found in \cite{CZ1998}. In what follows, let $L_{loc}^2(\mathbb{R})$ denote the space of all functions $g:\mathbb{R}\rightarrow\mathbb{C}$ such that $g\in L^{2}(B,A)$ for all $B,A\in\mathbb{R}$ with $B<A$. 

\begin{lemma}[Lemma A in \cite{CZ1998}] Suppose that $F \in L^{2}_{loc}(\mathbb{R})$ and that, for some nonnegative constants $M$, $C$, $\epsilon$, and $A_{0}$,
\begin{equation*}
\int_{j-1}^{j}\left|F(t)\right|^{2} d t \leq M^{2}, \quad j \in \mathbb{Z},
\end{equation*}
and
\begin{equation*}
\int_{-A}^{A}|F(t)|^{2} dt\leq C \int_{\mathbb{R} \backslash[-A, A]} G_{A}^{2}(t) d t+C \int_{-A}^{A}\left(G_{\infty}(t)-G_{A}(t)\right) G_{\infty}(t) d t+\epsilon, \quad A>A_{0},
\end{equation*}
where, for $0<A \leq+\infty$,
\begin{equation*}
G_{A}(s):=\int_{-A}^{A}(1+|s-t|)^{-3 / 2}|F(t)| d t, \quad s \in \mathbb{R} .
\end{equation*}
Then $F \in L^{2}(\mathbb{R})$ and
\begin{equation*}
\int_{-\infty}^{+\infty}|F(t)|^{2} d t \leq \epsilon.
\end{equation*}
\label{thm5}
\end{lemma}

The following lemma gives some properties of the two-layered Green function, which will be used in this subsection, in Section \ref{section3c} and in Appendix \ref{section7}.

\begin{lemma}\label{thm36} Assume  $k_{+},k_{-}>0$ with $k_{+}\neq k_{-}$. Define $x=(x_1,x_2)\in \mathbb{R}^2$ and $y=(y_1,y_2)\in \mathbb{R}^2$. Then we have  the following statements.

{\rm (i)} For $y\in \mathbb{R}^{2}_{-}$, there hold $G(\cdot,y)|_{\overline{U_{0}}}\in C^{1}(\overline{U_{0}})$,  $\nabla_{y}G(\cdot,y)|_{\overline{U_{0}}}\in C^{1}(\overline{U_{0}})$, $G(\cdot,y)|_{\mathbb{R}^2\backslash (U_{0}\cup \{y\})}\in C^{1}(\mathbb{R}^2\backslash (U_{0}\cup \{y\}))$ and  $\nabla_{y}G(\cdot,y)|_{\mathbb{R}^2\backslash (U_{0}\cup \{y\})}\in C^{1}(\mathbb{R}^2\backslash (U_{0}\cup \{y\}))$.

{\rm (ii)} Let $h_{0},h_{1},\delta>0$.
There hold
\begin{equation}
|\nabla_{x}G(x,y)|,~|\nabla_{x}\nabla_{y}G(x,y)|\leq C|x_1-y_1|^{-\frac{3}{2}}\label{eq189}
\end{equation}
for all $x,y\in \mathbb{R}^2$ satisfying $|x_2|\leq h_{0}$, $0<|y_2|<h_{1}$ and $|x_1-y_1|>\delta$, where the constant $C$ depends only on $h_{0},h_{1},k_{\pm},\delta$.

{\rm (iii)} Let $K$ be a bounded domain such that $\overline{K}\subset \mathbb{R}^2_-$. Then we have that $G(x,y)$ and $\partial_{y_i}G(x,y)$ ($i=1,2$) satisfy the Sommerfeld radiation condition \eqref{eq22} uniformly for all $\hat{x}\in \mathbb{S}^1_+$ and $y\in K$.
\end{lemma}

\begin{proof}
The statement (i) can be directly deduced by the expression \eqref{eq63} of G(x,y) and Lebesgue's dominated convergence theorem.

For the statement (ii),  we only derive the estimate for $\nabla_{x}\nabla_{y}G(x,y)$, since the estimate for $\nabla_{x}G(x,y)$ can be deduced in a similar manner.
We choose $x_{0}=(x_{0}^{(1)},x_{0}^{(2)})$ and $y_{0}=(y_{0}^{(1)},y_{0}^{(2)})$ in $\mathbb{R}^2$ such that $|x_{0}^{(2)}|\leq h_{0}$, $0<|y_{0}^{(2)}|<h_{1}$ and $|x_{0}^{(1)}-y_{0}^{(1)}|>\delta$. Then by the expression \eqref{eq63} of $G(x,y)$, together with the integral representation of Hankel function given in \eqref{eq156}, it can be verified that $\partial_{y_{i}}G(x,y_{0})$, $i=1,2$, satisfies the Helmholtz equations in $\mathbb{R}^{2}_{\pm}$ with the wave numbers $k_{\pm}$, respectively, and satisfies the transmission boundary condition on $\Gamma_{0}$, i.e.,
\begin{align*}
&\partial_{y_{i}}G(x,y_{0})|_{x_2\to 0+}=\partial_{y_{i}}G(x,y_{0})|_{x_2\to 0-},\\
&\partial_{x_2}\partial_{y_{i}}G(x,y_{0})|_{x_2\to 0+}=\partial_{x_2}\partial_{y_{i}}G(x,y_{0})|_{x_2\to 0-}
\end{align*}
for $i=1,2$.
Thus, taking $\epsilon$ such that $0<\epsilon<\delta/2$, we obtain that $\partial_{y_{i}}G(x,y_0),i=1,2$, satisfies $\Delta v(x)=f(x)$ in $B_{\epsilon}(x_0)$ in the distributional sense with $v(x):=\partial_{y_{i}}G(x,y_0)$ and $f(x):=-k^2(x)\partial_{y_i}G(x,y)$, where $B_{\epsilon}(x_0)$ is a ball with center $x_0$ and radius $\epsilon$. Hence, using Lemma \ref{thm1} for $\partial_{y_{i}}G(x,y)$ $(i=1,2)$ in $B_{\epsilon}(x_0)$ and applying the statements (i) and (ii) in Theorem \ref{thm38}, we obtain
\begin{align*}
&|\nabla_{x}\nabla_{y}G(x_{0},y_{0})|\leq \epsilon^{-1}\sup_{x\in B_{\epsilon}(x_0)}(1+\max (k_{+},k_{-})\epsilon^2)|\nabla_{y}G(x,y)|\\
&\leq \epsilon^{-1}(1+\max (k_{+},k_{-})\epsilon^2)\sup_{x\in B_{\epsilon}(x_0)}C(1+|x_2|)(1+|y_2|)(|x-y_0|^{-3/2}+|x-y_0'|^{-3/2})\\
&\leq \tilde{C} |x_{0}^{(1)}-y_{0}^{(1)}|^{-3/2},
\end{align*}
where $y_{0}'=(y_{0}^{(1)},-y_{0}^{(2)})$ and the constant $\tilde{C}$ depends only on $h_{0},h_{1},k_{\pm},\delta$. This completes the proof.

Finally,
by employing similar arguments as in the proofs of Theorems 2.1 and 2.14 in \cite{LYZZ2022}, we can use patient calculations to obtain that the statement (iii) holds true.
\end{proof}

The following lemma has been proved in \cite{CZ1998a}.

\begin{lemma}[Lemma 6.1 in \cite{CZ1998a}]
\label{thm39}Let $h>0$. If $\phi \in L^2(\Gamma_{h})\cap L^{\infty}(\Gamma_{h})$ and $v$ is defined by \eqref{eq1}, then $v|_{\Gamma_{a}}$, $\partial_{1}v|_{\Gamma_{a}}$ and $\partial_{2}v|_{\Gamma_{a}}$ are in $L^{2}(\Gamma_{a})\cap BC(\Gamma_{a})$ for all $a>h$ and
\begin{align}
&\int_{\Gamma_{a}}[|\partial_{2}v|^2-|\partial_{1}v|^2+k_{+}^2|v|^2]ds\leq 2k_{+}\mathrm{Im}\int_{\Gamma_{a}}\overline{v}\partial_{2}v ds,\label{eq198} \\
&\mathrm{Im}\int_{\Gamma_{a}}\overline{v}\partial_{2}v ds\geq 0.\label{eq199}
\end{align}
\end{lemma}

Now, we assume that  $k_+$ > $k_-$ and $u^s\in \mathscr{R}_{d}(D)$ is the solution of the problem (DBVP) with $g =0$. We proceed to show that $\partial_\nu u^s$ vanishes on $\Gamma$.
Let $A>0$ and $a>0$. Then we can set $B=-A$ in the formula (\ref{eq7}) to obtain that
\begin{equation*}
K(A):=\int_{\Gamma(A)}|\partial_\nu u^s|^2ds\leq C\left(I_{1}(A)+R_1(A)\right),
% \label{eq193}
\end{equation*}
where $I_{1}(A) := I_{1}(-A,A)$ and $R_1(A):=R_1(-A,A)$.
Let $v_{dir}$ be defined by
\begin{equation}
    v_{dir}(x): = \int_{\Gamma(A)}\partial_\nu u^s(y)G(x,y)ds(y),\quad x \in D.
    \label{eq11}
\end{equation}
 By employing Lemma \ref{thm34}, \eqref{eq192} and the property of $\partial_{\nu}u^{s}$ given in Remark \ref{thm37}, it can be derived that $v_{dir}|_{\Gamma_{b}}\in BC(\Gamma_{b})\cap L^2(\Gamma_{b})$ for all $b>0$.  On the other hand, it is easy to see from \eqref{eq166}, \eqref{eq192} and the statement (iii) of Lemma \ref{thm36} that $v_{dir}$ is a radiating solution of $\Delta v_{dir}+k_{+}^2 v_{dir}=0$ in $U_{b}$ for all $b>0$. Thus, in view of the equivalence of the statements (ii) and (iv) in \cite[Theorem 2.9]{CZ1998a}, $v_{dir}$ satisfies \eqref{eq1} with $h=b$ and $\phi=v_{dir}|_{\Gamma_{b}}$ for every $b>0$. Hence, by employing \eqref{eq198}, we have $I_{1}''(A)\leq 2k_{+}J''(A)$, where
\begin{equation*}
    I_{1}''(A) := \int_{\Gamma_{a}}\left(|\partial_{2}v_{dir}|^2-|\partial_{1}v_{dir}|^2+k_{+}^2|v_{dir}|^2\right)ds, \quad J''(A):= \mathrm{Im}\int_{\Gamma_{a}}\overline{v_{dir}}\partial_{2}v_{dir}ds.
\end{equation*}
By applying Green's theorem in the domain $\left\{x=(x_1,x_2)\,:\, x\in D\backslash \overline{U}_{0},|x_1|<A \right\}$ and in the domain $\{x=(x_1,x_2)\,:\,| x_{1}|< A, 0< x_{2}< a\}$, we can use the conditions (ii) and (iii) in the problem (DBVP) to find that $J(A)=R_{2}(A)$, where
\begin{equation*}
    J(A):=\mathrm{Im} \int_{\Gamma_{a}}\overline{u^{s}}\partial_{2}u^{s}ds,\quad R_{2}(A):=\mathrm{Im}\left(\int_{\gamma_{a}(-A)}-\int_{\gamma_{a}(A)}\right)\overline{u^{s}}\partial_1 u^{s} ds.
\end{equation*}
Thus, from the above discussions, we can derive
\begin{align*}
    K(A)\leq C\left(I_{1}(A)-I_{1}''(A)+2k_{+}\left(J''(A)-J(A)\right)+R_{1}(A)+2k_{+}R_{2}(A)\right).
\end{align*}
Set
\begin{align*}
    I_{1}'(A) := \int_{\Gamma_{a}(A)}\left(|\partial_{2}v_{dir}|^2-|\partial_{1}v_{dir}|^2+k_{+}^2|v_{dir}|^2\right)ds, \quad J'(A):= \mathrm{Im}\int_{\Gamma_{a}(A)}\overline{v_{dir}}\partial_{2}v_{dir}ds
\end{align*}
and $w\left(x_{1}\right):=\partial_{\nu} u^s\left(x_{1}, f\left(x_{1}\right)\right), x_{1} \in \mathbb{R}$. Then for all $A>0$,
\begin{equation*}\int_{-A}^{A}\left|w\left(x_{1}\right)\right|^{2} d x_{1} \leq  \int_{\Gamma(A)}|\partial_{\nu}u^s|^2ds \leq\left(1+L^{2}\right)^{1 / 2} \int_{-A}^{A}\left|w\left(x_{1}\right)\right|^{2} d x_{1}.\end{equation*}
By the formulas (\ref{eq143}) and (\ref{eq11}) and the estimates  \eqref{eq192} and \eqref{eq189}, we obtain that
\begin{align*}
    |v_{dir}(x)|,~|\nabla v_{dir}(x)| & \leq C_{a}\left(1+L^{2}\right)^{1 / 2} W_{A}\left(x_{1}\right), \quad x \in \Gamma_{a}, \\
    |u^s(x)-v_{dir}(x)|,~|\nabla u^s(x)-\nabla v_{dir}(x)| & \leq C_{a}\left(1+L^{2}\right)^{1 / 2}\left(W_{\infty}\left(x_{1}\right)-W_{A}\left(x_{1}\right)\right), \quad x \in \Gamma_{a},
\end{align*}
where the constant $C_a$ is independent of $x_1$ but dependent on $a$ and where $W_{A}(x_1)$ and $W_{\infty}(x_1)$ are defined by
\begin{align*}
&W_{A}\left(x_{1}\right):=\int_{-A}^{A}\left(1+\left|x_{1}-y_{1}\right|\right)^{-3 / 2}\left|w\left(y_{1}\right)\right| d y_{1}, \quad x_{1} \in \mathbb{R},\\
&W_{\infty}\left(x_{1}\right):=\int_{-\infty}^{+\infty}\left(1+\left|x_{1}-y_{1}\right|\right)^{-3 / 2}\left|w\left(y_{1}\right)\right| d y_{1}, \quad x_{1} \in \mathbb{R}.
\end{align*}
These lead to
\begin{align*}
&\left|I_{1}'(A)-I_{1}''(A)\right|,~
\left|J'(A)-J''(A)\right| \leq C \int_{\mathbb{R} \backslash[-A, A]}\left(W_{A}\left(x_{1}\right)\right)^{2} d x_{1}, \\
&\left|I_{1}(A)-I_{1}'(A)\right|,~
\left|J(A)-J'(A)\right| \leq 2 C \int_{-A}^{A}\left(W_{\infty}\left(x_{1}\right)-W_{A}\left(x_{1}\right)\right) W_{\infty}\left(x_{1}\right) d x_{1},
\end{align*}
where $C=C_{a}^{2}\left(1+L^{2}\right)\left(2+k_{-}^{2}\right)$. Hence, there exists a constant $C>0$ such that for all $A>0$,
\begin{align*}
    \int_{-A}^{A}\left|w\left(x_{1}\right)\right|^{2} d x_{1} & \leq C \left( \int_{\mathbb{R} \backslash[-A, A]}\left(W_{A}\left(x_{1}\right)\right)^{2} d x_{1}\right.\\
      & \quad \left.+  \int_{-A}^{A}\left(W_{\infty}\left(x_{1}\right)-W_{A}\left(x_{1}\right)\right) W_{\infty}\left(x_{1}\right) d x_{1}+|R_1(A)|+2k_{+}|R_2(A)| \right).
\end{align*}
Combining this with \eqref{eq9} and the fact that $\partial_\nu u^s \in L_{loc}^2(\Gamma)$ (see Remark \ref{thm37}), we can apply Lemma \ref{thm5} to conclude that  $w\in L^2(\mathbb{R})$ (which is equivalent to $\partial_{\nu}u^s\in L^2(\Gamma)$), and that for all $A_{0}>0$,
\begin{equation}
    \left(1+L^{2}\right)^{-1 / 2} \int_{\Gamma}\left|\partial_{\nu} u^s\right|^{2} d s \leq \int_{-\infty}^{\infty}\left|w\left(x_{1}\right)\right|^{2} d x_{1} \leq C \sup _{A>A_{0}}\big(\left|R_1(A)\right|+2k_{+}|R_2(A)|\big).
    \label{eq12}
\end{equation}
For $x\in D\backslash \overline{U}_{a}$ with $|x_{1}|\geq 1$, we deduce by \eqref{eq143} and \eqref{eq192} that
    \begin{align*}
        \left|u^s(x)\right|^{2} &\leq 2 \left(\int_{\Gamma \backslash \Gamma\left(\left|x_{1}\right| / 2\right)}\left|\partial_{\nu} u^s(y)  G(x, y)\right| d s(y)\right)^{2}+2\left(\int_{\Gamma\left(\left|x_{1}\right| / 2\right)}\left|\partial_{\nu} u^s(y)G(x, y)\right| d s(y)\right)^{2} \\
        &\leq C_{1}  \int_{\Gamma \backslash \Gamma\left(\left|x_{1}\right| / 2\right)}\left|\partial_{\nu} u^{s}\right|^{2} d s+C_{2}\left(\frac{\left|x_{1}\right|}{2}\right)^{-3}\int_{\Gamma}\left|\partial_{\nu} u^{s}\right|^{2} d s,
        \end{align*}
where
\begin{equation*}
    C_{1}=2 \sup _{x \in D\backslash \overline{U}_{a}} \int_{\Gamma}\left|G(x, y)\right|^{2} d s(y)<\infty.
\end{equation*}
Thus, $u^s(x)\to 0$ as $x_{1}\to \infty$ with $x \in D\backslash \overline{U}_{a}$, uniformly in $x_{2}$.
Hence by Theorems \ref{thm2} and \ref{thm3} as well as Lemma \ref{thm1}, we have $R_j(A)\to 0$  as $A \to \infty$, $j=1,2$. Therefore, it follows from (\ref{eq12}) that $\partial_{\nu}u^s = 0 $ on $\Gamma$.

In conclusion, based on the above discussions and Theorem \ref{thm6}, we establish the following theorem on the uniqueness of the problem (DBVP).

\begin{theorem}
    For every $g\in BC(\Gamma)$, there exists at most one solution  $u^s \in \mathscr{R}_{d}(D)$ that satisfies the boundary value problem (DBVP) under the assumption $k_{+} > k_{-}>0$.
    \label{thm9}
\end{theorem}

\begin{remark}
In the case $k_+<k_-$, it can be seen from {\rm\cite[Example 2.3]{KL2018}} that when the rough boundary $\Gamma$ is a planar surface, there exist some wave numbers $k_{+}$ and $k_{-}$ such that the problem (DBVP) with $g=0$ has a nontrivial solution.
\end{remark}

\subsection{The Uniqueness Result of the Problem (IBVP)}
\label{section3c}

\begin{theorem}
    Let $u^s\in \mathscr{R}_{i}(D)$ be the solution of the problem (IBVP). Then
    \begin{equation}
        u^s(x) = -\int_{\Gamma}\left( \frac{G(x,y)}{\partial \nu(y)}-ik_{-}\beta(x)G(x,y)\right)u^s(y)ds(y)+\int_{\Gamma}G(x,y)g(y)ds(y),\,\, x\in D,
        \label{eq14}
    \end{equation}
    where $\nu$ denotes the unit normal on $\Gamma$ pointing out of $D$.
    \label{green_ip}
\end{theorem}

\begin{proof}
First, we consider the case when $x=(x_1,x_2)\in D\backslash \overline{U}_{0}$. Let $A>0$ and let the domain $T_{A}^{\epsilon}$ be given as in \eqref{eq206}, where $B_{\epsilon}(x)$ denotes a ball centered at $x$ with radius $\epsilon$ small enough such that $\overline{B_{\epsilon}(x)}\subset  \left\{x\,:\,|x_1|<A, x\in D\backslash \overline{U}_{0} \right\}$.
    By applying Green's theorem in the domain $T_{A}^{\epsilon}$ and letting $\epsilon\to 0$, it follows that
    \begin{align}
        u^s(x) & = -\int_{\Gamma(A)}\left( \frac{\partial G(x,y)}{\partial \nu(y)}-ik_{-}\beta(x)G(x,y)\right)u^s(y)ds(y)+\int_{\Gamma(A)}G(x,y)g(y)ds(y)\notag\\
        &+\int_{\Gamma_{0}(A)}\left(G(x,y)\frac{\partial u^s}{\partial y_2}\big|_{-}(y)-u^s|_{-}(y)\frac{\partial G(x,y)}{\partial y_2}\right)ds(y)\notag \\
        &+ \left(\int_{\gamma_{0}(-A)}-\int_{\gamma_{0}(A)}\right)\left(u^s(y)\frac{\partial G(x,y)}{\partial y_1}-G(x,y)\frac{\partial u^s(y)}{\partial y_1}\right)ds(y),\label{eq13}
    \end{align}
    where '-' in the third integral
of the above formula is the limit given as in \eqref{eq172}.
    With the help of Theorems \ref{thm38} and \ref{thm2}, we can apply the similar argument as in the derivations of \eqref{eq203}--\eqref{eq205} to obtain that
    \begin{equation*}
        \lim_{A\to\infty} \left(\int_{\gamma_{0}(-A)}-\int_{\gamma_{0}(A)}\right)\left(u^s(y)\frac{\partial G(x,y)}{\partial y_1}-G(x,y)\frac{\partial u^s(y)}{\partial y_1}\right)ds(y) = 0
    \end{equation*}
    and that for $a>0$,
    \begin{align*}
        &\lim_{A\to \infty} \int_{\Gamma_{0}(A)}\left(G(x,\cdot)\frac{\partial u^s}{\partial y_2}\big|_{-}-u^s|_{-}\frac{\partial G(x,\cdot)}{\partial y_2}\right)ds \\
       &= \int_{\Gamma_{a}}\left(G(x,y)\frac{\partial u^s(y)}{\partial y_2}-u^s(y)\frac{\partial G(x,y)}{\partial y_2}\right)ds(y)=0.
    \end{align*}
    Thus we can obtain the formula (\ref{eq14}) by letting $A\to +\infty$ in the formula (\ref{eq13}).

    Second, by the dominated convergence theorem and the above discussions, the formula (\ref{eq14}) holds for $x\in \Gamma_{0}$. Moreover, by using similar arguments as above, we have that \eqref{eq14} also holds for $x \in U_{0}$.
\end{proof}

In the rest of this subsection, with a slight abuse of notations, we will redefine $J(A)$, $J'(A)$, $J''(A)$, $R_1(A)$, $K(A)$, $\omega(\cdot)$, $W_A(\cdot)$ and $W_\infty(\cdot)$.
Applying Green's theorem in the domains $\{x=(x_1,x_2)\,:\,x\in D\backslash \overline{U}_{0},|x_1|< A\}$ and $\{x=(x_1,x_2)\,:\,|x_1|< A, 0< x_2< a \}$ with $A,a>0$ and using the conditions (ii) and (iii) in the problem (IBVP), we can immediately obtain the following lemma.

\begin{lemma}
Set $A,a>0$. Let $u^s\in \mathscr{R}_{i}(D)$ satisfy the problem (IBVP) with $g=0$. Then
\begin{equation*}
    k_{-}\int_{\Gamma(A)}\mathrm{Re}(\beta)|u^s|^2ds+J(A)=R_{1}(A),
\end{equation*}
where
\begin{equation}
    J(A):=\mathrm{Im}\int_{\Gamma_{a}(A)}\overline{u^s}\partial_2 u^sds,\quad R_{1}(A):=\mathrm{Im}\left(\int_{\gamma_{a}(-A)}-\int_{\gamma_{a}(A)}\right)\overline{u^s}\partial_1 u^sds.\label{eq200}
\end{equation}
\end{lemma}

Now we give the uniqueness of the problem (IBVP).

\begin{theorem}
    Suppose that $k_{\pm}>0$ and $d>0$. If $\beta\in BC(\Gamma)$ with $\mathrm{Re}(\beta(x))\geq d$ on $x\in \Gamma$, then the problem (IBVP) has at most one solution for every $g\in BC(\Gamma)$.
    \label{thm26}
\end{theorem}

\begin{proof}
Let $u^s \in \mathscr{R}_{i}(D)$ satisfy the problem (IBVP) with $g=0$. We need to show that $u^s \equiv 0$ in $D$.
Let $A>0$ and define $v_{imp}$ by
  \begin{equation}
    v_{imp}(x): = -\int_{\Gamma(A)}\left( \frac{\partial G(x,y)}{\partial \nu(y)}-ik_{-}\beta(x)G(x,y)\right)u^s(y)ds(y),\quad x \in D.\label{eq194}
\end{equation}
By utilizing the estimates of $G$ in \eqref{eq192} and the fact that $u^s\in \mathscr{R}_{i}(D)$,  it can be derived that $v_{imp}|_{\Gamma_{b}}\in BC(\Gamma_{b})\cap L^2(\Gamma_{b})$ for all $b>0$. On the other hand, it follows from
\eqref{eq166}, \eqref{eq192} and the statement (iii) of Lemma \ref{thm36} that $v_{imp}$ is a radiating solution of $\Delta v_{imp}+k_{+}^2v_{imp}=0$ in $U_{b}$ for all $b>0$. Thus, in view of the equivalence of the statements (ii) and (iv) in Theorem 2.9 in \cite{CZ1998a}, $v_{imp}$ satisfies \eqref{eq1} with $h=b$ and $\phi=v_{imp}|_{\Gamma_{b}}$ for every $b>0$.

Let $a>0$ and set
\begin{equation*}
    J'(A):= \mathrm{Im}\int_{\Gamma_{a}(A)}\overline{v_{imp}}\partial_2 v_{imp}ds,\quad  J''(A) := \int_{\Gamma_{a}}\overline{v_{imp}}\partial_2 v_{imp}ds.
\end{equation*}
Then by \eqref{eq199} in Lemma \ref{thm39}, $J''(A)\geq 0$, so that, by \eqref{eq200} and the fact that $\mathrm{Re}(\beta(x))\geq d>0$ for $x\in \Gamma$, we have
\begin{equation*}
K(A):=\int_{\Gamma(A)}|u^s|^2ds\leq (k_{-}d)^{-1}\left(-J(A)+R_1(A)\right)\leq (k_{-}d)^{-1}\left(J''(A)-J(A)+R_1(A)\right).
\end{equation*}
Let $w(x_1)=u^s(x_1,f(x_1))$. Then
\begin{equation*}
    \int_{-A}^{A}|w(x_1)|^2dx_1\leq K(A) \leq \sqrt{1+L^2}\int_{-A}^{A}|w(x_1)|^2dx_1.
\end{equation*}
Set
\begin{align*}
    &W_A(x_1):=\int_{-A}^{A}(1+|x_1-y_1|)^{-\frac{3}{2}}|w(y_1)|dy_1,\quad x_1\in \mathbb{R},\\
    &W_\infty(x_1):=\int_{-\infty}^{+\infty}(1+|x_1-y_1|)^{-\frac{3}{2}}|w(y_1)|dy_1,\quad x_1\in \mathbb{R}.
\end{align*}
It follows from the formulas \eqref{eq14} and \eqref{eq194} and the estimates \eqref{eq192} and \eqref{eq189} that
\begin{align*}
    &|v_{imp}(x)|,~|\nabla v_{imp}(x)|\leq C W_{A}(x_1),\quad  x\in \Gamma_{a},\\
    &|u^s(x)-v_{imp}(x)|,~|\nabla u^s(x)-\nabla v_{imp}(x)|\leq C (W_{\infty}(x_1)-W_{A}(x_1)),\quad x\in \Gamma_{a}.
\end{align*}
This leads to
\begin{align*}
    &|J'(A)-J''(A)|  \leq C\int_{\mathbb{R}\backslash[-A,A]}(W_{A}(x_1))^2dx_1,\\
   & |J(A)-J'(A)|  \leq 2C \int_{-A}^{A}(W_{\infty}(x_1)-W_A(x_1))W_{\infty}(x_1)dx_1.
\end{align*}
Hence, the above analysis gives that
\begin{equation*}
    \int_{-A}^{A}|w(x_1)|^2dx_1\leq C\left(\int_{\mathbb{R}\backslash[-A,A]}(W_{A}(x_1))^2dx_1+\int_{-A}^{A}(W_{\infty}(x_1)-W_{A}(x_1))W_{\infty}(x_1)dx_1 + |R_{1}(A)| \right).
\end{equation*}
By employing  Lemma \ref{thm5}, we obtain that for all $A_0>0$,
\begin{equation}
\left(1+L^2 \right)^{-1/2}\int_{\Gamma}|u^s|^2ds\leq \int_{-\infty}^{\infty}(W_A(x_1))^2dx_1\leq C\sup_{A>A_{0}}|R_{1}(A)|.
    \label{eq62}
\end{equation}

From \eqref{eq3} and the fact that $u^{s}\in C(\overline{D})$, we obtain that $u^{s}\in BC(\Gamma)$. This, together with Theorems \ref{green_ip} and \ref{thm10}, implies that $u^s\in C^{0,\lambda}(\Gamma)$ for every $\lambda \in (0,1)$. Thus $u^s\in BUC(\Gamma)\cap L^2(\Gamma)$, which yields that  $u^s(x)\to 0$ as $|x|\to \infty$ for $x\in \Gamma$.
Choose a cutoff function $\psi_{A}\in BC(\Gamma)$ such that $\|\psi_{A}\|_{\infty,\Gamma}=1$ with  $\psi_{A}(x)=1$ for $|x_1|\leq A/3$ and $\psi_{A}(x)=0$ for $|x_1|\geq 2A/3$.
Let $u^s_1(x)$ and $u^s_2(x)$ be given by \eqref{eq14} with $g=0$, where the density $u^{s}$ is replaced by $u^{s}(1-\psi_{A})$ and $u^{s}\psi_{A}$, respectively. Thus
$u^s(x)=u^s_1(x)+u^s_2(x)$ for $x\in D$.
From Theorems \ref{thm7} (iii) and \ref{thm27} (iii), we have that there exists some constant $C>0$ such that for all $x\in \gamma_{a}(-A)\cup \gamma_{a}(A)$, $|u^s_1(x)|\leq C\|u^s(1-\psi_A)\|_{\infty,\Gamma}\to 0$ as $A\to \infty$.
Moreover, it follows from the definition of $u_2^{s}(x)$ and Theorem \ref{thm38} that there exists some constant $C>0$ such that
\begin{equation*}
\sup_{x\in\gamma_{a}(-A) \cup \gamma_{a}(A)} |u^s_2(x)|\leq C\|u^s\|_{\infty,\Gamma} \int_{\frac{A}{3}}^{\frac{5A}{3}}t^{-\frac{3}{2}}dt\to 0 \quad \text{ as } A \to \infty.
\end{equation*}
Hence, by
\eqref{eq4} we obtain that $R_{1}(A)\to 0$ as $A\to \infty$. Consequently, from  \eqref{eq62} we have $u^s=0$ on $\Gamma$. This, together with \eqref{eq14}, implies that $u^s=0$ in $D$.
Therefore, the proof is complete.
\end{proof}

\subsection{The Existence Result of the Problem (DBVP)}
\label{section10}
For $\psi \in BC(\Gamma)$, the integrals
\begin{align}
    W(x)&:=\int_{\Gamma}\frac{\partial G(x,y)}{\partial \nu(y)}\psi(y)ds(y),\quad x \in \mathbb{R}^2\backslash \Gamma, \label{eq67}\\
    V(x)&:=\int_{\Gamma}G(x,y)\psi(y)ds(y), \quad x \in \mathbb{R}^2\backslash \Gamma,\label{eq68}
\end{align}
are called the double- and single-layer potentials, respectively.  Here, $\nu$ denotes the unit normal on $\Gamma$ pointing out of $D$.
The properties of the double-layer potential (\ref{eq67}) and the single-layer potential (\ref{eq68}) are summarized in Appendix \ref{section7}.

We introduce a function in the form of a combined double- and single-layer potential, i.e.,
\begin{equation}
    u^s(x):=\int_{\Gamma}\left(\frac{\partial G(x,y)}{\partial \nu(y)}+i\eta G(x,y)\right)\psi(y)ds(y),\quad x\in \mathbb{R}^2\backslash \Gamma,
\label{eq16}
\end{equation}
where $\psi \in BC(\Gamma)$, $\eta \neq 0$ is a constant.
From the statements (\romannumeral1),  (\romannumeral3) and (\romannumeral5) in Theorem \ref{thm7} and the statements (\romannumeral1),  (\romannumeral3) and (\romannumeral4) in Theorem \ref{thm27}, the potential $u^s$ satiefies the conditions (\romannumeral1), (\romannumeral2), (\romannumeral4) and (\romannumeral5) of the problem (DBVP) with $\alpha=-1/2$. Furthermore, by the statement (\romannumeral2) in Theorem \ref{thm7} and the statement (\romannumeral2) in Theorem \ref{thm27}, $u^s$ satisfies the condition (\romannumeral3) of the problem (DBVP) provided $\psi \in BC(\Gamma)$ is the solution of the following boundary integral equation
\begin{equation}
    \psi(x)=2\int_{\Gamma}\left(\frac{\partial G(x,y)}{\partial \nu(y)}+i\eta G(x,y)\right)\psi(y)ds(y)-2g(x) \text{ on } \Gamma.
    \label{eq17}
\end{equation}
Thus we get the following result.

\begin{theorem}
    The combined double- and single-layer potential (\ref{eq16}) satisfies the problem {\rm (DBVP)} with $\alpha=-1/2$, provided $\psi\in BC(\Gamma)$ satisfies the boundary integral equation (\ref{eq17}).
    \label{thm8}
\end{theorem}

Define $\tilde{\psi},\tilde{g}\in BC(\mathbb{R})$ by
\begin{equation}
     \tilde{\psi}(s):=\psi(s,f(s)),\quad \tilde{g}(s) := g(s, f(s)), \quad s\in \mathbb{R}.\label{eq168}
\end{equation}
By parameterizing the equation (\ref{eq17}), we obtain the following integral equation problem: find $\tilde{\psi}\in BC(\mathbb{R})$ such that
\begin{equation}
    \tilde{\psi}(s)-2\int_{\mathbb{R}}\left(\frac{\partial G(x,y)}{\partial \nu(y)}+i\eta G(x,y)\right)\sqrt{1+|f'(t)|^2}\tilde{\psi}(t)dt=-2\tilde{g}(s),\quad s\in \mathbb{R},
    \label{eq35}
\end{equation}
where $x = (s,f(s)),y=(t,f(t))$. Define the kernel $\kappa_{f}$ by
\begin{equation}
    \kappa_{f}(s,t)=2\left(\frac{\partial G(x,y)}{\partial \nu(y)}+i\eta G(x,y)\right)\sqrt{1+|f'(t)|^2},\quad s,t\in \mathbb{R},\quad s\neq t,\label{eq162}
\end{equation}
with $x = (s,f(s)),y=(t,f(t))$.
Using this kernel, define the integral operator $K_{f}$ by
\begin{equation*}
    (K_{f}\phi)(s):=\int_{\mathbb{R}}\kappa_{f}(s,t)\phi(t)dt,\quad s\in \mathbb{R},
\end{equation*}
for $\phi\in BC(\mathbb{R})$.
Then the equation (\ref{eq17}) can be written as
\begin{equation*}
    (I-K_{f})\tilde{\psi}=-2\tilde{g},
\end{equation*}
where $I$ denotes the identity operator on $BC(\mathbb{R})$. Here, we use the subscript to indicate the dependence of the kernel $\kappa_{f}$ and the operator $K_{f}$ on the function $f$.

Since $K_{f}\tilde{\psi}$ with $\tilde{\psi}\in BC(\mathbb{R})$ is an integral over the unbounded interval $\mathbb{R}$, $K_{f}$ is not a compact operator on $BC(\mathbb{R})$. Thus it is impossible to use the Riesz-Fredholm theorem to establish the solvability of the integral equation (\ref{eq35}). To overcome this difficulty, we follow the approach in \cite{ZC2003}. The following theorem presents the uniqueness of the integral equation (\ref{eq17}).

\begin{theorem}\label{integral_equation_uniqueness}
    If $\eta > 0$ and $k_{+} > k_{-} > 0$, then the integral equation (\ref{eq17}) has at most one solution in $BC(\mathbb{R})$.
\end{theorem}

\begin{proof}
Suppose that $\tilde{\psi}\in BC(\mathbb{R})$ satisfies
\begin{equation}
(I-K_{f})\tilde{\psi}=0.
\label{eq19}
\end{equation}
It suffices to prove that $\tilde{\psi}=0$.
Define $\psi\in BC(\Gamma)$ by $\psi(t,f(t)) :=\tilde{\psi}(t)$ with $t\in \mathbb{R}$. Let $v^s(x)$ with $x\in\mathbb{R}^2\backslash \Gamma$ be the combined double- and single-layer potential with the density function $\psi$, that is,
    \begin{equation*}
        v^s(x):=\int_{\Gamma}\left(\frac{\partial G(x,y)}{\partial \nu(y)}+i\eta G(x,y)\right)\psi(y)ds(y),\quad x\in \mathbb{R}^2\backslash \Gamma.
    \end{equation*}
    Then $v^s$ satisfies (\ref{eq17}) with $g=0$. Hence, it follows from Theorem \ref{thm8} that $v^s$ satisfies the problem (DBVP) with $g=0$, so that, by Theorem \ref{thm9}, $v^s\equiv0$ in $D$. Furthermore, let $\partial v^s_{\pm}/\partial \nu$ and $v^s_{\pm}$ be defined as in the equations (\ref{eq20}) and (\ref{eq21}), respectively. Then it follows that $v^s_{+}=\partial v^s_{+}/\partial\nu=0$ on $\Gamma$. By the statements (\romannumeral2) and (\romannumeral4) in Theorem \ref{thm7} and the statement (\romannumeral2) in Theorem \ref{thm27}, we have the following jump relations
    \begin{equation}
    v^s_{-}-v^s_{+}=\psi,\quad \partial v^s_{-}/\partial \nu-\partial v^s_{+}/\partial \nu=-i\eta \psi\quad\textrm{on}~\Gamma,
    \label{eq26}
    \end{equation}
    which implies that $\psi=v^s_{-}$ and $\partial v^s_{-}/\partial \nu=-i\eta \psi$. Hence \begin{equation}
        \partial v_{-}^s/\partial \nu+i\eta v_{-}^s=0\quad\textrm{on}~\Gamma. \label{eq61}
    \end{equation}
    For $x=(x_1,x_2)\in \mathbb{R}^2$, we define $\tilde{x}:=(x_1,-x_2)$. Define
    \begin{align}
    &\tilde{\Gamma}:=\{x=(s,-f(s))\,:\,s\in \mathbb{R}\},\label{eq195}\\
    &\tilde{D}: = \{x=(x_1,x_2)\,:\,x_1\in \mathbb{R},x_2 > -f(x_1)
    \}.\label{eq196}
    \end{align}
     It can be observed that $\tilde{x}\in \tilde{\Gamma}$ (resp. $\tilde{x}\in \tilde{D}$) if and only if $x\in \Gamma$ (resp. $x\in \mathbb{R}^2\backslash \overline{D}$).

Let $\tilde{v}^s$ be defined as
\begin{equation}
\tilde{v}^s(x) := v^s(\tilde{x})
\label{eq197}
\end{equation}
for $x\in \tilde{D}$.
Let $\nu$ be the unit  normal to $\Gamma$ pointing out of $D$ and let $\tilde{\nu}$ be the unit normal to $\tilde{\Gamma}$ pointing out of $\tilde{D}$. Define
\begin{equation*}
    \frac{\partial \tilde{v}^s}{\partial \tilde{\nu}}(x):=\lim_{h\to 0+}\nabla \tilde{v}^{s}(x-h\tilde{\nu}(x))\cdot \tilde{\nu}(x),\quad \tilde{v}^s(x):=\lim_{h\to 0+}\tilde{v}^{s}(x-h\tilde{\nu}(x)),\quad x\in \tilde{\Gamma}.
\end{equation*}
It is clear that $\partial \tilde{v}^s / \partial \tilde{\nu}(x)= - \partial v^s_{-}/\partial \nu (\tilde{x})$ and $\tilde{v}^s(x)=v_{-}^s(\tilde{x})$ for $x\in \tilde{\Gamma}$.
This, together with (\ref{eq61}), \eqref{eq19}, the boundary condition $\partial \tilde{v}^s/\partial \tilde{\nu}-i\eta \tilde{v}^{s} = 0$ on $\Gamma$ as well as Theorem \ref{thm10}, implies that $\psi \in C^{0,\lambda}(\Gamma)$.
Thus, by Theorems \ref{thm11} and \ref{thm12}, $\tilde{v}^{s}$ satisfies the condition (iv) of the impedance problem (IP) in \cite[Section 2]{ZC2003}. Then by combining the statements (\romannumeral1), (\romannumeral3) and (\romannumeral5) in Theorem \ref{thm7} and the statements (\romannumeral1), (\romannumeral3) and (\romannumeral4) in Theorem \ref{thm27}, $\tilde{v}^{s}$ satisfies the conditions (\romannumeral1), (\romannumeral3) and (\romannumeral5) of the impedance problem (IP) in \cite{ZC2003}.
Hence, based on the above discussions, $\tilde{v}^{s}$ satisfies the impedance problem (IP) in \cite{ZC2003} with $D=\tilde{D}$, $\Gamma=\tilde{\Gamma}$,  $\beta = \eta/k_{-}$ and with the wave number $k=k_{-}$ and the boundary data $g =0$. Choose $\eta>0$ so that ${\rm Re} \beta >\epsilon>0$ for some $\epsilon$. Then by the uniqueness theorem \cite[Theorem 4.7]{ZC2003}, $\tilde{v}^{s} \equiv 0$ in $ \tilde{D}$, which implies that $v^s \equiv 0$ in $\mathbb{R}^2\backslash \overline{D}$. Then the jump relations in (\ref{eq26}) give that $\psi =0$. Therefore, the proof is complete.
\end{proof}

In the rest of this subsection, we assume that $\eta>0$ and $k_+>k_->0$. Now we utilize Theorem \ref{thm15} to prove the existence of the integral equation (\ref{eq35}). We use the notations defined in  Appendix \ref{section9}.
For some $c_1<0$ and $c_2>0$, we define $B(c_1,c_2)$  by
    \begin{equation*}
        B\left(c_{1}, c_{2}\right):=\left\{f \in C^{1,1}(\mathbb{R}) \,:\, f(s) \leq c_{1}, s \in \mathbb{R} \text { and }\|f\|_{C^{1,1}(\mathbb{R})} \leq c_{2}\right\}.
    \end{equation*}
Let $W_{dir} :=\{k_{f}\,:\,f \in B(c_1,c_2)\}$. By Theorem \ref{integral_equation_uniqueness}, $I-\mathscr{K}_{f}:BC(\mathbb{R})\to BC(\mathbb{R})$ is injective for all $f\in W_{dir}$, where $\mathscr{K}_{l}$ is defined by \eqref{eq65} and $I$ is the identity operator on $BC(\mathbb{R})$.
Then $T_a(W_{dir}) =W_{dir} $ for all $a \in \mathbb{R}$, where $T_{a}l(s,t)=l(s-a,t-a)$. By Lemma \ref{thm16} (\romannumeral1), $W_{dir}\subset BC(\mathbb{R},L^1(\mathbb{R}))\subset \mathbf{K}$, and for all $s\in \mathbb{R}$ satisfies
\begin{equation*}
    \sup_{k_{f} \in W_{dir}} \int_{\mathbb{R}}\left|k_{f}(s, t)-k_{f}\left(s', t\right)\right| \mathrm{d} t \rightarrow 0, \quad \text { as } s' \rightarrow s.
\end{equation*}
By the statements (\romannumeral1) and (\romannumeral2) in Lemma \ref{thm17}, $W_{dir}$ is $\sigma$-sequentially compact in $\mathbf{K}$.
Let $l\in W_{dir}$ and $f\in B(c_1,c_2)$ such that $l=k_{f}$. Choose a periodic function $f_{n}\in B(c_1,c_2)$ satisfying $f_{n}(x_{1})=f(x_{1})$ for $x_1\in[-n, n]$ and let $l_{n}:=k_{f_{n}}\in W_{dir}$. Then it yields that
\begin{equation*}
    f_{n} \stackrel{s}{\longrightarrow} f ,\quad f_{n}'\stackrel{s}{\longrightarrow} f'.
\end{equation*}
This, together with Lemma \ref{thm17} (\romannumeral2), implies that $l_{n} \stackrel{\sigma}{\longrightarrow} l$.
Since $T_{a_n} l_n=l_n$, where $a_n>0$ is the period of $f_n$, and $l_n \in B C\left(\mathbb{R}, L^1(\mathbb{R})\right)$, it follows from Theorem 2.10 in \cite{CZ1997} that \eqref{eq34} holds.
By the above discussions, $W_{dir}$ satisfies all the conditions of Theorem \ref{thm15} and thus we obtain the following results.

\begin{theorem}
Let $\eta>0$ and $k_{+}>k_{-}>0$. Then, for all $f\in B(c_{1},c_{2})$ the integral operator $I-K_{f}:BC(\mathbb{R})\to BC(\mathbb{R})$ is bijective (and so
boundedly invertible) with
\[\sup_{f\in B(c_1,c_2)}\|(I-K_{f})^{-1}\|<\infty.\]
Thus the integral equations \eqref{eq17} and \eqref{eq35} have exactly one solution for every $f\in B(c_1,c_2)$ and $g\in BC(\Gamma)$, with
\begin{equation*}
\|\psi\|_{\infty,\Gamma}=\|\tilde{\psi}\|_{\infty}\leq C\|\tilde{g}\|_{\infty}=\|g\|_{\infty,\Gamma},
\end{equation*}
where $C$ is a positive constant  depending only on $k_{\pm}$ and $B(c_1,c_2)$.
\label{thm20}
\end{theorem}

By combining Theorems \ref{thm9}, \ref{thm8}, \ref{integral_equation_uniqueness}, \ref{thm20}, \ref{thm7} (iii) and \ref{thm27} (iii), we arrive at the following theorem on the well-posedness of the problem (DBVP).

\begin{theorem} Assume $f\in B(c_1,c_2)$ and $k_{+}>k_{-}>0$. Then for every $\eta>0$ and $g\in BC(\Gamma)$, the problem (DBVP) has exactly one solution in the form
    \begin{equation*}
        u^s(x)=\int_{\Gamma}\bigg(\frac{\partial G(x,y)}{\partial \nu(y)}+i\eta G(x,y)\bigg)\psi(y)ds(y),\quad x\in D.
    \end{equation*}
    Here, the density function $\psi \in BC(\Gamma) $ is the unique solution of the integral equation
    % \begin{equation*}
    % A_{d}\psi:=\bigg(-\frac{1}{2}I+K+i\eta S\bigg)\psi = g,
    % \end{equation*}
    \begin{equation*}
    A_d\psi (x):= -\frac{1}{2}\psi(x) +\int_{\Gamma}\bigg(\frac{\partial G(x,y)}{\partial \nu (y)}+i\eta G(x,y)\bigg)\psi(y) ds(y)=g(x),\quad x\in \Gamma,
    \end{equation*}
    where $A_{d}$ is bijective (and thus boundedly invertible) in $BC(\Gamma)$. Moreover, for some constant $C>0$ depending only on $B(c_1,c_2)$ and $k_{\pm}$,
    \begin{equation*}
        |u^{s}(x)|\leq C \big|x_{2}+|f_{-}|+1\big|^{1/2}\|g\|_{\infty,\Gamma},\quad x\in D,
    \end{equation*}
    for all $f\in B(c_1,c_2)$ and $g\in BC(\Gamma)$.
    \label{thm31}
\end{theorem}

\subsection{The Existence Result of the Problem (IBVP)}
\label{section11}
In this subsection, we seek a solution in the form of the single-layer potential
\begin{equation}
    u^s(x)=\int_{\Gamma}G(x,y)\psi(y)ds(y),\quad x\in D,
    \label{eq27}
\end{equation}
for some $\psi\in BC(\Gamma)$. Using the statements (\romannumeral1) and (\romannumeral3) in Theorem \ref{thm27}, we obtain $u^s$ satisfies the conditions (\romannumeral1), (\romannumeral2) and (\romannumeral4) of the problem (IBVP) with $\alpha=-1/2$. With the aid of Theorem \ref{thm12}, we have that $u^s$ satisfies the condition (\romannumeral5) of the problem (IBVP) for any $\theta \in (0,1)$. Thus, by Theorem \ref{thm27} (\romannumeral2), the single-layer potential (\ref{eq27}) is a solution of the problem (IBVP) provided $\psi$ satisfies the following integral equation
\begin{equation}
    \psi(x)+2\int_{\Gamma}\left(\frac{\partial G(x,y)}{\partial \nu(x)}-ik_{-}\beta(x)G(x,y)\right)\psi(y)ds(y)=2g(x),\quad x\in \Gamma.
    \label{eq30}
\end{equation}
Hence, we obtain the following theorem.
\begin{theorem}
    The single-layer potential (\ref{eq27}) satisfies the problem (IBVP) for $\alpha=-1/2$ and for any $\theta \in (0,1)$, provided $\psi\in BC(\Gamma)$ satisfies the boundary integral equation (\ref{eq30}).
    \label{thm29}
\end{theorem}

Let $\tilde{\psi}$ and $\tilde{g}$ be given as in \eqref{eq168} and let $\tilde{\beta} \in BC(\mathbb{R})$ be given by
\begin{equation}
   \tilde{\beta}(s):=\beta(s,f(s)), \quad s\in \mathbb{R}. \label{eq158}
\end{equation}
By parameterizing the integral in (\ref{eq30}), we have the following integral equation problem: find $\tilde{\psi}\in BC(\mathbb{R})$ such that
\begin{equation}
    \tilde{\psi}(s)+2\int_{\mathbb{R}}\left(\frac{\partial G(x,y)}{\partial \nu(x)}-ik_{-}\tilde{\beta}(s)G(x,y)\right)\sqrt{1+|f'(t)|^2}\tilde{\psi}(t)dt=2\tilde{g}(s),\quad s\in \mathbb{R},
    \label{eq37}
\end{equation}
where $x=(s,f(s))$ and $y=(t,f(t))$. Define the kernel $\kappa_{\tilde{\beta},f}$ by
\begin{equation}
    \kappa_{\tilde{\beta},f}(s,t) := -2\left(\frac{\partial G(x,y)}{\partial \nu(x)}-ik_{-}\tilde{\beta}(s)G(x,y)\right)\sqrt{1+|f'(t)|^2},\quad s,t\in \mathbb{R},\quad s\neq t,\label{eq163}
\end{equation}
where $x=(s,f(s))$ and $y = (t, f(t))$. Using this kernel, define the integral operator $K_{\tilde{\beta},f}$ by
\begin{equation*}
    (K_{\tilde{\beta},f}\phi)(s):=\int_{\mathbb{R}}\kappa_{\tilde{\beta},f}(s,t)\phi(t)dt,\quad s\in \mathbb{R},
    % \label{eq159}
\end{equation*}
for $\phi\in BC(\mathbb{R})$. Then the equation (\ref{eq30}) can be written as
\begin{equation*}
    (I-K_{\tilde{\beta},f})\tilde{\psi}=2\tilde{g},
\end{equation*}
where $I$ denotes the identity operator on $BC(\mathbb{R})$.
Here, we use the subscript to indicate the dependence of the kernel $\kappa_{\tilde{\beta},f}$ and the operator $K_{\tilde{\beta},f}$ on the functions $\tilde{\beta}$ and $f$.

Using Theorem \ref{thm26} for the uniqueness of the impedance problem (IBVP), we can establish the following uniqueness result for the integral equation (\ref{eq30}).

\begin{theorem}
Suppose that $k_{\pm}>0$ and $d>0$. If $\beta\in BC(\Gamma)$ with $\mathrm{Re}(\beta(x))\geq d$ on $x\in \Gamma$,
then the boundary integral equation (\ref{eq30}) has at most one solution in $BC(\Gamma)$.
\label{thm13}
\end{theorem}

\begin{proof}
    Suppose $\tilde{\psi}\in BC(\mathbb{R})$ satisfies
    \begin{equation*}
        (I-K_{\tilde{\beta},f})\tilde{\psi}=0.
        % \label{eq32}
    \end{equation*}
    We only need to prove that
    $\tilde{\psi}=0$.

    Define $\psi \in BC(\Gamma)$ by $\psi(t,f(t)):=\tilde{\psi}(t), t\in \mathbb{R}$, and let $v^{s}$ in $\mathbb{R}^2\backslash \Gamma$ be the single-layer potential with the density $\psi$, that is,
    \begin{equation*}
        v^{s}(x):=\int_{\Gamma}G(x,y)\psi(y)ds,\quad x\in \mathbb{R}^2\backslash\Gamma.
    \end{equation*}
    Then $\psi$ satisfies (\ref{eq30}) with $g=0$, so that, by Theorem \ref{thm26}, $v^{s}\equiv 0$ in $D$. Furthermore, by Theorem \ref{thm27} (\romannumeral2),
    \begin{equation}
        v^{s}_{-}(x)=v^{s}_{+}(x),  \quad \partial_{\nu} v^{s}_{-}(x)- \partial_{\nu} v^{s}_{+}(x) = -\psi(x),\quad x\in \Gamma \label{eq33},
    \end{equation}
    where $v^{s}_{\pm}$ and $\partial_{\nu}v^{s}_{\pm}$ are defined as in the equations (\ref{eq21}) and (\ref{eq20}), respectively.
    Thus, $v^{s}_{-}=0$ on $\Gamma$.
    This implies that  $\tilde{v}^{s}=v^{s}_{-}=0$ on $\tilde{\Gamma}$, where $\tilde{\Gamma}$ is given as in \eqref{eq195} and
    $\tilde{v}^s$ is defined in the same way as in \eqref{eq197}.
    Moreover, let $\tilde{D}$ be given as in \eqref{eq196}. By the statements (\romannumeral1)--(\romannumeral4) of Theorem \ref{thm27}, $\tilde{v}^{s}$ satisfies the problem (P) in \cite[Section 2]{CZ1998} with
    $D=\tilde{D}$ and $\Gamma=\tilde{\Gamma}$ and with the boundary data $g=0$ on $\tilde{\Gamma}$. Hence, by Theorem 3.4 in \cite{CZ1998}, it follows that $\tilde{v}^s\equiv 0$ in $\tilde{D}$, which implies that $v^{s}\equiv 0$ in $\mathbb{R}^2\backslash \overline{D}$. Therefore, by (\ref{eq33}) we obtain $\psi=0$. The proof is now completed.
\end{proof}

Now we are going to prove the existence of the integral equation (\ref{eq30}). We will use Theorem \ref{thm15} and use the notations in Appendix \ref{section9}.
For some $d_1\geq 0$, $d_2>0$ and some function $\omega:[0,\infty)\to[0,\infty)$ such that $\omega(s)\to 0$ as $s\to 0$, let $E(d_1,d_2,\omega)$ be defined by
\begin{align*}
    E(d_1,d_2,\omega):=&\Big\{\tilde{\beta}(s)\in BC(\mathbb{R})\,:\,\mathrm{Re}(\tilde{\beta}(s))\geq d_1,s\in\mathbb{R},\|\tilde{\beta}\|_{\infty}\leq d_2, \text{ and } \\
    &\big|\tilde{\beta}(s)-\tilde{\beta}(t)\big|\leq \omega(|s-t|),s,t\in\mathbb{R}\Big\}.
    % \label{eq70}
\end{align*}
Note that $E(d_1,d_2,\omega)\subset BUC(\mathbb{R})$. Conversely, given $\tilde{\beta}\in BUC(\mathbb{R})$, it holds that $\tilde{\beta}\in E(d_1,d_2,\omega)$ provided $d_1\leq \inf_{s\in \mathbb{R}} \mathrm{Re}(\tilde{\beta}(s)), d_2\geq \|\tilde{\beta}\|_{\infty}$ and $\omega(h) \geq \sup_{s\in\mathbb{R},|t|\leq h}|\tilde{\beta}(s+t)-\tilde{\beta}(s)|$ for all $h\geq 0$.
We have the following existence result for the integral equation (\ref{eq30}).

\begin{theorem}
       Suppose that, for some $d>0$, $\mathrm{Re}\beta(x)\geq d$ for all $x\in \Gamma$. Then the integral equations  \eqref{eq30} and \eqref{eq37} have exactly one solution for every $f\in B(c_1,c_2)$, $g\in BC(\Gamma)$ and $\beta \in BUC(\Gamma)$. Moreover, if $d_1>0$, then there exists some constant $C>0$ depending only on $B(c_1,c_2)$, $E(d_1,d_2,\omega)$ and $k_{\pm}$ such that
       \begin{align*}
\|\psi\|_{\infty,\Gamma}=\|\tilde{\psi}\|_{\infty} \leq C \|\tilde{g}\|_{\infty} =\|g\|_{\infty,\Gamma}
       \end{align*}
       for all $f\in B(c_1,c_2)$, $g\in BC(\Gamma)$ and $\tilde{\beta}\in E(d_1,d_2,\omega)$ with $\beta$ defined in terms of $\tilde{\beta}$ by \eqref{eq158}.
    \label{thm14}
\end{theorem}

\begin{proof}
    Let $W_{imp}:=\{\kappa_{\tilde{\beta},f}\,:\,f\in B(c_1,c_2),\tilde{\beta} \in E(d_1,d_2,\omega)\}$.
    It follows from Theorem \ref{thm13} that $I-\mathscr{K}_{l}:BC(\mathbb{R})\to BC(\mathbb{R})$ is injective for all $l\in W_{imp}$, where $\mathscr{K}_{l}$ is defined by \eqref{eq65} and $I$ is the identity operator.
    Moreover, $T_{a}(W_{imp})=W_{imp}$ for all $a\in \mathbb{R}$. By Lemma \ref{thm16} (ii), $W_{imp}\subset BC(\mathbb{R},L^{1}(\mathbb{R}))\subset \mathbf{K}$ and $W_{imp}$ satisfies (\ref{eq29}).
    From the statement (i) in Lemma \ref{thm17}  and the statements (i) and (ii) in Lemma \ref{thm18}, $W_{imp}$ is $\sigma$-sequentially compact in $\mathbf{K}$.
    Let $l\in W_{imp}$, $f\in B(c_1,c_2)$ and $\tilde{\beta} \in E(d_1,d_2,\omega)$ such that $l=\kappa_{\tilde{\beta},f}$.
    For each $n\in \mathbb{N}^{+}$, choose $f_{n}\in B(c_1,c_2)$ and $\tilde{\beta}_n \in E(d_1,d_2,\omega)$ so that $f_{n}$ and $\tilde{\beta}_{n}$ are periodic with the same period and $f_{n}(x_{1})=f(x_1)$, $\tilde{\beta}_{n}(x_1)=\tilde{\beta}(x_1)$ for $x_1\in[-n, n]$.
    Then $l_{n}:=\kappa_{\tilde{\beta}_{n},f_{n}}\in W_{imp}$ and $f_{n}\xrightarrow{s}f$, $f_{n}'\xrightarrow{s}f'$, $\tilde{\beta}_{n}\xrightarrow{s}\tilde{\beta}$, so that, by Lemma \ref{thm18} (ii), $l_{n}\xrightarrow{\sigma}l$.
Since $T_{a_{n}}l_{n}=l_{n}$, where $a_{n}>0$ is the period of $f_{n}$ and $\tilde{\beta}_{n}$ and where $l_{n}\in BC(\mathbb{R},L^{1}(\mathbb{R}))$, it follows from \cite[Theorem 2.10]{CZ1997} that the condition (\ref{eq34}) holds.
Thus $W_{imp}$ satisfies all the conditions in Theorem \ref{thm15}. Therefore, the statement of this theorem follows from Theorem \ref{thm15}.
\end{proof}
By combining Theorems \ref{thm26}, \ref{thm29}, \ref{thm13}, \ref{thm14},  \ref{thm7} (iii) and \ref{thm27} (iii), we obtain the following result on the well-posedness of the problem (IBVP).

\begin{theorem}
    Assume $f\in B(c_1,c_2)$ and $k_{+},k_{-}>0$ with $k_{+}\neq k_{-}$. Let $d> 0$ and suppose that $\beta\in BUC(\Gamma)$ satisfies $\mathrm{Re}(\beta(x))\geq d$ for all $x\in \Gamma$. Then for every $g\in BC(\Gamma)$, the problem (IBVP) has exactly one solution in the form
    \begin{equation*}
        u^s(x)=\int_{\Gamma}G(x,y)\psi(y)ds(y),\quad x\in D. \quad
    \end{equation*}
    Here, the density function $\psi \in BC(\Gamma) $ is the unique solution of the integral equation
%     \begin{equation*}
% A_{i}\psi:=\bigg(\frac{1}{2}I+K'-ik_{-}\beta S\bigg)\psi = g,
%     \end{equation*}
    \begin{equation*}
    A_{i}\psi(x):=\frac{1}{2}\psi(x)+\int_{\Gamma}\bigg(\frac{\partial G(x,y)}{\partial \nu(x)}-ik_{-}\beta(x)G(x,y)\bigg)\psi(y)ds(x)=g(x),\quad x\in \Gamma,
    \end{equation*}
    where $A_{i}$ is bijective (and thus boundedly invertible) in $BC(\Gamma)$. Moreover, if $d_1>0$, then for some constant $C>0$ depending only on $B(c_1,c_2)$, $E(d_1,d_2,\omega)$ and $k_{\pm}$,
    \begin{equation*}
        |u^{s}(x)|\leq C \big|x_{2}+|f_{-}|+1\big|^{1/2}\|g\|_{\infty,\Gamma},\quad x\in D,
    \end{equation*}
    for all $f\in B(c_1,c_2)$, $g\in BC(\Gamma)$ and $\tilde{\beta}\in E(d_1,d_2,\omega)$ with $\beta$ defined in terms of $\tilde{\beta}$ by $\beta(s,f(s))=\tilde{\beta}(s)$, $s\in \mathbb{R}$.
    \label{thm32}
\end{theorem}

\section{The Nystr\"{o}m Method for the Problems (DBVP) and (IBVP)}
\label{section6}
In this section, motivated by \cite{MCK2000}, we present the Nystr\"{o}m method for numerically solving the problems (DBVP) and (IBVP), based on the integral equations \eqref{eq17} and \eqref{eq30}.
Nystr\"{o}m methods have been extensively studied for computing solutions of integral equations on bounded curves (see, e.g., \cite{CK2019}). Moreover, this kind of methods was extended in \cite{MCK2000} to solve  integral equations on unbounded domains.

For $\psi \in BC(\mathbb{R})$, define the boundary integral operators
\begin{align*}
    &(S\psi)(x):= \int_{\Gamma}G(x,y)\psi(y)ds(y),\quad x\in \Gamma,\\
    &(K\psi)(x):=\int_{\Gamma}\frac{\partial G(x,y)}{\partial \nu(y)}\psi(y)ds(y),\quad x\in \Gamma,\\
    &(K'\psi)(x):= \int_{\Gamma}\frac{\partial G(x,y)}{\partial \nu(x)}\psi(y)ds(y),\quad x\in \Gamma.
\end{align*}
For $x,y\in\Gamma$, we write $x=(s,f(s))$ and $y=(t,f(t))$ for $s,t\in \mathbb{R}$. For the functions $\psi(x), g(x), \beta(x) \in BC(\Gamma) $, we define the parameterized functions $\tilde{\psi},\tilde{g},\tilde{\beta} \in BC(\mathbb{R})$ in terms of \eqref{eq168} and \eqref{eq158}.
Then the above three integral operators can be parameterized as
\begin{align*}
    &(\tilde{S}\tilde{\psi})(s) :=\int_{\mathbb{R}}\kappa_1(s,t)\tilde{\psi}(t)dt,\\
    &(\tilde{K}\tilde{\psi})(s) :=\int_{\mathbb{R}}\kappa_2(s,t)\tilde{\psi}(t)dt,\\
    &(\tilde{K}'\tilde{\psi})(s):=\int_{\mathbb{R}}\kappa_{3}(s,t)\tilde{\psi}(t)dt
\end{align*}
for $s\in \mathbb{R}$, respectively, where the kernels $\kappa_{1}$, $\kappa_{2}$ and $\kappa_{3}$ are given by
\begin{align*}
    &\kappa_{1}(s,t):=G(x,y)\sqrt{1+|f'(t)|^2}, \\
    &\kappa_{2}(s,t):=\frac{\partial G(x,y)}{\partial \nu(y)}\sqrt{1+|f'(t)|^2},\\
    &\kappa_{3}(s,t):=\frac{\partial G(x,y)}{\partial \nu(x)}\sqrt{1+|f'(t)|^2}.
\end{align*}

Next, we rewrite the kernels $\kappa_{1}$, $\kappa_{2}$, and $\kappa_{3}$. Let $H^{(1)}_{j}(t)$
denote the Hankel function of the first kind of order $j$ and let $J_{j}(t)$ and $Y_{j}(t)$ denote the Bessel function and the Neumann function, respectively, of order $j$ (see \cite{CK2019}).
According to (\ref{eq63}),  the expansion (3.98) in \cite{CK2019} for the Neumann functions, and the properties of Hankel functions $({H}_{0}^{(1)}(t))'=-{H}_{1}^{(1)}(t)$ and $H_{j}^{(1)}(t)=J_{j}(t)+iY_{j}(t)$ for $j=1,2$, the kernels $\kappa_{1}, \kappa_{2}$ and $\kappa_{3}$ have the representations
\begin{align}
    &\kappa_{1}(s,t)=a_{1}(s,t)\ln|s-t|+b_{1}(s,t),\label{eq74}\\
    &\kappa_{2}(s,t)=a_{2}(s,t)\ln|s-t|+b_{2}(s,t),\label{eq75}\\
    &\kappa_{3}(s,t)=a_{3}(s,t)\ln|s-t|+b_{3}(s,t),\label{eq76}
\end{align}
where $a_{i}$ and $b_{i}$ ($i=1,2,3$) are given by
\begin{align}
    &a_{1}(s,t):=-\frac{1}{2\pi}J_{0}(k_{-}|x-y|)\sqrt{1+|f'(t)|^2},\label{eq80}\\
    &b_{1}(s,t):=\frac{i}{4}H_{0}^{(1)}(k_{-}|x-y|)\sqrt{1+|f'(t)|^2}-a_{1}(s,t)\ln|s-t|+G_{\mathcal{R}}(x,y)\sqrt{1+|f'(t)|^2},\label{eq81}\\
    &a_{2}(s,t):=\frac{k_{-}}{2\pi}(y-x)\cdot \nu(y)\frac{J_{1}(k_{-}|x-y|)}{|x-y|}\sqrt{1+|f'(t)|^2},\label{eq82}\\
    &b_{2}(s,t):=\frac{ik_{-}}{4}H_{1}^{(1)}(k_{-}|x-y|)\frac{x-y}{|x-y|} \cdot {\nu}(y)\sqrt{1+|f'(t)|^2}-a_{2}(s,t)\ln|s-t|+c(s,t)\sqrt{1+|f'(t)|^2},\label{eq83}\\
    & a_{3}(s,t):=\frac{k_{-}}{2\pi}(x-y)\cdot \nu(x)\frac{J_{1}(k_{-}|x-y|)}{|x-y|}\sqrt{1+|f'(t)|^2},\label{eq84}\\
    & b_{3}(s,t):= \frac{ik_{-}}{4}H_{1}^{(1)}(k_{-}|x-y|)\frac{y-x}{|x-y|} \cdot {\nu}(x)\sqrt{1+|f'(t)|^2}-a_{3}(s,t)\ln|s-t|+d(s,t)\sqrt{1+|f'(t)|^2},\label{eq85}
 \end{align}
where $\nu(x)=(f'(s),-1)/\sqrt{1+|f'(s)|^2}$, $\nu(y)=(f'(t),-1)/\sqrt{1+|f'(t)|^2}$, $c(s,t)={\partial G_{\mathcal{R}}(x,y)}/{\partial \nu(y)}$ and $d(s,t)={\partial G_{\mathcal{R}}(x,y)}/{\partial \nu(x)}$.
We note from \eqref{hw-eq9} that $c(s,s)=d(s,s)$ for $s\in\mathbb{R}$.
Then using the formulas (3.97) and (3.98) in \cite{CK2019}, we can deduce that the diagonal terms $a_{1}(s,s)=a_{2}(s,s)=a_{3}(s,s)=0$ for $s\in \mathbb{R}$, and
\begin{align*}
    &b_{1}(s,s)=\left(\frac{i}{4}-\frac{\gamma}{2\pi}-\frac{1}{2\pi}\ln\left(\frac{k_{-}}{2}\sqrt{1+|f'(s)|^2}\right)\right)\sqrt{1+|f'(s)|^2}+G_{\mathcal{R}}(x,x)\sqrt{1+|f'(s)|^2},\\
    &b_{2}(s,s)=b_{3}(s,s)=-\frac{1}{4\pi}\frac{1}{1+|f'(s)|^2}f''(s)+c(s,s)\sqrt{1+|f'(s)|^2}
\end{align*}
for $s\in \mathbb{R}$, where $\gamma$ denotes the Euler constant.

Let $\chi \in C_{0}^{\infty}(\mathbb{R})$ denote the cut-off function satisfying $0\leq\chi (s)\leq 1 $ for $s\in \mathbb{R}$ and satisfying that $\chi(s)=0$ for $|s|\geq \pi$, $\chi(s)=1$ for $|s|\leq 1$ and $\chi(-s)=\chi(s)$ for $s\in \mathbb{R}$. Then $\kappa_{1}$, $\kappa_{2}$ and $\kappa_{3}$ can be written as
\begin{equation*}
    \kappa_{i}(s,t)=\frac{1}{2\pi}A_{i}(s,t)\ln \left( 4\sin^2\left(\frac{s-t}{2}\right)\right)+B_{i}(s,t), \quad s,t\in \mathbb{R},\, s\neq t,\, i=1,2,3,
\end{equation*}
where $A_{i}$ and $B_{i}$ are given by
\begin{align*}
    &A_{i}(s,t):=\pi a_{i}(s,t)\chi (s-t),\\
    &B_{i}(s,t):=a_{i}(s,t)\left(\ln |s-t|(1-\chi(s-t))-\chi(s-t)\ln\left(\frac{\sin ((s-t)/2)}{(s-t)/2}\right)\right)+b_{i}(s,t)
\end{align*}
for $i=1,2,3$. In particular, we set $B_{i}(s,s):=b_{i}(s,s)$ for all $s\in \mathbb{R}$ and $i=1,2,3.$

In the following two subsections, we will give the convergence analysis and the numerical implementation of our Nystr\"{o}m method.

\subsection{Convergence Analysis}
\label{section12}
Set the step length $h:=\pi/N$ for $N\in \mathbb{N}^{+}$ and set $t_{j}=jh$ for $j\in \mathbb{Z}$. It follows from \cite{MCK2000} that we can approximate the integral operators $\tilde{S},\tilde{K}$ and $\tilde{K}'$ by $\tilde{S}_{N},\tilde{K}_{N}$ and  $\tilde{K}_{N}'$, respectively, which are given by
\begin{equation}
    (W\psi)(s):=\sum_{j\in \mathbb{Z}}\alpha^{N,i}_{j}(s)\psi(t_{j}),\quad s\in \mathbb{R},
    \label{eq79}
\end{equation}
for $(W,i)=(\tilde{S}_{N},1),(\tilde{K}_{N},2),(\tilde{K}_{N}',3)$. Here, $\alpha^{N,i}_{j}$ is given by
\begin{align*}
    &\alpha^{N,i}_{j}(s):=R_{j}^{N}(s)A_{i}(s,t_{j})+\frac{\pi}{N}B_{i}(s,t_{j}),\quad s\in \mathbb{R},
\end{align*}
with
\begin{align*}
    &R_{j}^{N}(s):=-\frac{1}{N}\Bigg\{ \sum_{m=1}^{N-1}\frac{1}{m}\cos ( m(s-t_{j}))+\frac{1}{2N}\cos(N(s-t_{j})) \Bigg\},\quad s\in \mathbb{R},
\end{align*}
for $N\in \mathbb{N}^{+},j\in \mathbb{Z}$ and $i=1,2,3.$

Using the discretization operators $\tilde{S}_{N},\tilde{K}_{N}$ and  $\tilde{K}_{N}'$, we approximate the integral equations (\ref{eq35}) and (\ref{eq37}) by
\begin{equation}
    \tilde{\psi}^{D}_{N}(s)-2[(\tilde{K}_{N}+i\eta \tilde{S}_{N})\tilde{\psi}^{D}_{N}](s)=-2\tilde{g}(s),\quad s\in \mathbb{R}, \label{eq72}
\end{equation}
and
\begin{equation}
    \tilde{\psi}_{N}^{I}(s)+2[(\tilde{K}_{N}'-ik_{-}\tilde{\beta}\tilde{S}_{N})\tilde{\psi}_{N}^{I}](s)=2\tilde{g}(s),\quad s\in \mathbb{R},\label{eq73}
\end{equation}
respectively. Here, the functions $\tilde{\psi}_{N}^{D}, \tilde{\psi}_{N}^{I}\in BC(\mathbb{R})$ denote the solutions of the approximate systems \eqref{eq72} and \eqref{eq73}, respectively.

Let $\tilde{\psi}^D$ and $\tilde{\psi}^I$ be the solutions of the equations \eqref{eq35} and \eqref{eq37}, respectively.
In the rest of this subsection, we estimate the error terms $\|\tilde{\psi}^D_{N}-\tilde{\psi}^D\|_{\infty}$ and $\|\tilde{\psi}^I_{N}-\tilde{\psi}^I\|_{\infty}$ by applying Theorem 3.13 in \cite{MCK2000}. To this end, we define the condition $\mathbf{C_{n}}$ for any kernel function $\kappa(s,t)$ with $s,t\in\mathbb{R}$ and $s\neq t$.

 $\mathbf{Condition}$ $\mathbf{C_{n}}.$ For $n\in \mathbb{N}_{0}:=\mathbb{N}^{+}\cup\{0\}$, we say that $\kappa$ satisfies $\mathbf{C_{n}}$ if
\begin{equation*}
    % \label{eq77}
\kappa(s, t)=a^{*}(s, t) \ln |s-t|+b^{*}(s, t), \quad s, t \in \mathbb{R},~s \neq t,
\end{equation*}
where $a^*$, $b^*$ $\in C^n(\mathbb{R}^2)$, and there exist constants $C>0$ and $p>1$ such that for all $j,l\in \mathbb{N}_{0}$ with $j+l\leq n$, we have
\begin{equation}
    \left|\frac{\partial^{j+l} a^{*}(s, t)}{\partial s^{j} \partial t^{l}}\right| \leq C, \quad\left|\frac{\partial^{j+l} b^{*}(s, t)}{\partial s^{j} \partial t^{l}}\right| \leq C, \quad s, t \in \mathbb{R},~|s-t| \leq \pi, \label{eq43}
\end{equation}
and
\begin{equation}
    \left|\frac{\partial^{j+l} \kappa(s, t)}{\partial s^{j} \partial t^{l}}\right| \leq C(1+|s-t|)^{-p}, \quad s, t \in \mathbb{R},~|s-t| \geq \pi.\label{eq89}
\end{equation}

For $m,n\in \mathbb{N}^{+}$, we denote by $B C^n\left(\mathbb{R}^m\right)$ the Banach space of all functions whose derivatives up to order $n$ are bounded and continuous on $\mathbb{R}^m$ with the norm defined by
\begin{equation*}
    \|\psi\|_{BC^n(\mathbb{R}^m)}:=\max_{l=0,1,\cdots,n}\max\limits_{\substack{0\leqslant \alpha_{i}\leqslant l\\ \sum_{i=1}^{m}\alpha_{i}=l}}\|\partial_{1}^{\alpha_{1}}\partial_{2}^{\alpha_{2}}\cdots\partial_{m}^{\alpha_{m}} \psi \|_{\infty,\mathbb{R}^{m}},
\end{equation*}
where $\partial_{i}^{\alpha_i}\psi({x})={\partial^{\alpha_{i}}}\psi({x})/{\partial x_{i}^{\alpha_{i}}}$.

To give the convergence of Nystr\"{o}m method, we should introduce some function spaces for the functions $f$ and $\tilde{\beta}$. For $c_1<0$, $c_2>0$ and $n\in \mathbb{N}_{0}$, we define the function space
\begin{equation*}
   B_{n}(c_1,c_2):=\left\{f \in B C^{n+2}(\mathbb{R})\,:\, \sup\nolimits_{x\in\mathbb{R}} f(x) \leq c_1 ,\|f\|_{B C^{n+2}(\mathbb{R})} \leq c_2\right\}.
\end{equation*}
For $d_1\geq 0$, $d_2>0$ and $n\in \mathbb{N}_{0}$, let
\begin{equation*}
    E_{n}(d_1,d_2):=\{\tilde{\beta}\in BC^{n}(\mathbb{R})\,:\,\mathrm{Re}(\tilde{\beta}(s))\geq d_1~\textrm{for}~s\in\mathbb{R},\|\tilde{\beta}\|_{BC^{n}(\mathbb{R})}\leq d_2, \tilde{\beta}\in BUC(\mathbb{R})\}.
\end{equation*}
Note that $B_{n}(c_1,c_2)$ and $E_{n}(d_1,d_2)$ are different from $B(c_1,c_2)$ and $E(d_1,d_2, \omega)$. It can be seen that $B_{n}(c_1,c_2)\subset B(c_1,c_2)$ and $E_{n}(d_1,d_2)\subset E(d_1,d_2,\omega)$ for $n\in \mathbb{N}_{0}$ if $w(h)=\sup_{s\in\mathbb{R},|t|\leq h}|\tilde{\beta}(s+t)-\tilde{\beta}(s)|$, $h\geq 0$, for some $\tilde{\beta}\in BC(\mathbb{R})$.

The following theorem presents the properties of $\kappa_{f}$ and $\kappa_{\tilde{\beta},f}$ given in \eqref{eq162} and \eqref{eq163}.

\begin{theorem} Suppose that $k_{\pm}>0$, $k_{+}\neq k_{-}$, $c_1<0$, $c_2>0$, $d_1\geq 0$, $d_2>0$, $\eta>0$ and $n\in \mathbb{N}_{0}$,  then $\kappa_{f}$ and $\kappa_{\tilde{\beta},f}$ given in \eqref{eq162} and \eqref{eq163} satisfy the condition $\mathbf{C_{n}}$ with the same constant for all $f\in B_{n}(c_1,c_2)$ and $\tilde{\beta}\in E_{n}(d_1,d_2)$.
    \label{thm19}
\end{theorem}

\begin{proof}
Let $f\in B_{n}(c_1,c_2)$ and $\tilde{\beta}\in E_{n}(d_1,d_2)$. The proof is divided into two parts.

\textbf{Part 1}: we consider the kernel $\kappa_{f}$. In view of the formulas (\ref{eq74})--(\ref{eq76}),
    let
    \begin{equation*}
        a^{*}(s,t):=2a_2(s,t)+2i\eta a_1(s,t)
    \end{equation*}
    and
    \begin{equation*}
        b^{*}(s,t):=2b_2(s,t)+2i\eta b_1(s,t)=k_{f}(s,t)-a^{*}(s,t)\ln(|s-t|)
    \end{equation*}
    for $s,t\in \mathbb{R}$, $s\neq t$.

First, we  establish the estimates of $a^{*}$ in (\ref{eq43}). Let $x(s)=(s,f(s)),y(t)=(t,f(t))$. Then it can be seen that $\sqrt{1+|f'(t)|^2}\in BC^{n+1}(\mathbb{R})$.
By utilizing the power series expansions of Bessel functions of the first kind (see \cite[equation (10.2.2)]{OLBC2010}), we have $J_{i}(k_{-}|x(s)-y(t)|)\in BC^{n}(\mathbb{R}^2)$ for $i=0,1$.
Note that $h_0(s,t):=\nu(y(t))\cdot (x(s)-y(t))/|x(s)-y(t)|^2\in C^{n}(\mathbb{R}^2)$ (see \cite[statement (7.1.36) in Section 7.1.3]{A1997}) and $\|h_0(\cdot,\cdot)\|_{BC^n(\mathbb{R}^2)}$ is uniformly bounded for all $f$ with $\|f\|_{BC^{n+2}(\mathbb{R})}\leq c_2$ (see \cite[formulas (7.1.32) and (7.1.33) in Section 7.1.3]{A1997}).
Consequently, from the definition of $a^{*}$ and the formulas \eqref{eq80} and \eqref{eq82}, it follows that $a^{*}(s,t)\in C^{n}(\mathbb{R}^2)$ with
\begin{equation}
\left|\frac{\partial^{j+l}a^{*}(s,t)}{\partial s^j \partial t^l}\right|\leq C,\quad s,t\in \mathbb{R},~|s-t|\leq \pi,\label{eq164}
\end{equation}
for all $j,l\in \mathbb{N}_{0}$ with $j+l\leq n$, where the constant $C>0$ depends only on $c_1,c_2,k_{\pm},\eta,n$.

Second, we establish the estimates of $b^{*}$ in (\ref{eq43}).
From \eqref{eq81} and \eqref{eq83}, we write $b_{1}$, $b_{2}$ as
\begin{align*}
b_{1}(s,t)&=\left( \frac{i}{4}H_{0}^{(1)}(k_{-}|x(s)-y(t)|)+\frac{1}{2\pi}J_{0}(k_{-}|x(s)-y(t)|)\ln|s-t|\right)\sqrt{1+|f'(t)|^2}\\
&\quad +G_{\mathcal{R}}(x(s),y(t))\sqrt{1+|f'(t)|^2}
\\
&:=b_{1,p}(s,t)\sqrt{1+|f'(t)|^2}+G_{\mathcal{R}}(x(s),y(t))\sqrt{1+|f'(t)|^2}
\end{align*}
and
\begin{align*}
b_{2}(s,t)&=\left(\frac{ik_{-}}{4}H_{1}^{(1)}(k_{-}|x(s)-y(t)|)+\frac{k_{-}}{2\pi}J_{1}(k_{-}|x(s)-y(t)|)\ln|s-t|\right)|x(s)-y(t)|\\
&\quad \cdot\frac{x(s)-y(t)}{|x(s)-y(t)|^2} \cdot {\nu}(y(t))\sqrt{1+|f'(t)|^2}+\frac{\partial G_{\mathcal{R}}(x(s),y(t))}{\partial \nu(y)}\sqrt{1+|f'(t)|^2}
\\
&:=b_{2,p}(s,t)\frac{x(s)-y(t)}{|x(s)-y(t)|^2} \cdot {\nu}(y(t))\sqrt{1+|f'(t)|^2}+\frac{\partial G_{\mathcal{R}}(x(s),y(t))}{\partial \nu(y(t))}\sqrt{1+|f'(t)|^2}.
\end{align*}
By using $H_{n}^{(1)}=J_{n}+iY_{n}$ as well as the power series expansions of Bessel functions \cite[equations (10.8.1) and (10.8.2)]{OLBC2010}, $b_{1,p}$ and $b_{2,p}$ can be rewritten as
\begin{align*}
&b_{1,p}(s,t)=J_{0}(k_{-}|x(s)-y(t)|)\left(\frac{i}{4}-\frac{1}{2\pi}\ln\left(k_{-}\sqrt{1+\left|\frac{f(s)-f(t)}{s-t}\right|^2}\right)-\frac{1}{2\pi}(-\ln(2)+\gamma)\right)\\
&\quad -\frac{1}{2\pi}\sum_{r=1}^{+\infty}(-1)^{r-1}\bigg(\sum_{j=1}^{r}\frac{1}{j}\bigg)\frac{(\frac{1}{4}k_{-}^2|x(s)-y(t)|^2)^{r}}{r!},\\
&b_{2,p}(s,t)=J_{1}(k_{-}|x(s)-y(t)|)|x(s)-y(t)|\left(\frac{ik_{-}}{4}-\frac{k_{-}}{2\pi}\ln\left(k_{-}\sqrt{1+\left|\frac{f(s)-f(t)}{s-t}\right|^2}\right) +\frac{k_{-}}{2\pi}\ln(2)\right)\\
&\quad +\frac{1}{2\pi}\sum_{r=0}^{n-1}\frac{(n-r-1)!}{r!}(\frac{1}{2}k_{-}^2|x(s)-y(t)|^2)^{r}\\
&\quad+ \frac{k_{-}}{4\pi}(\frac{1}{2}k_{-}|x(s)-y(t)|)^{n}\sum_{r=0}^{+\infty}(q(r+1)-q(n+r+1))\frac{(-\frac{1}{2}k_{-}^2|x(s)-y(t)|^2)^{r}}{r!(n+r)!}|x(s)-y(t)|,
\end{align*}
where $q(x):=\Gamma'(x)/\Gamma(x)$.
Here, $\gamma$ is the Euler constant and $\Gamma(x)$ denotes the Gamma function (see \cite[(5.2.1)]{OLBC2010}).
Consequently, from the fact that $(f(s)-f(t))^2/(s-t)^2\in BC^{n}(\mathbb{R}^2)$ and the
analyticity of the power series involved in the above two formulas for $b_{1,p},b_{2,p}$, we obtain that
$b_{i,p}(s,t)\in BC^{n}(\mathbb{R}^2)$ for $i=1,2$.
By utilizing the expression \eqref{hw-eq9} of $G_{\mathcal{R}}$, it follows that $G_{\mathcal{R}}(x,y)\in C^{\infty}(\mathbb{R}^{2}_{-}\times\mathbb{R}^{2}_{-})$ and
$G_{\mathcal{R}}(x,y)$ has the form
$G_{\mathcal{R}}(x,y)=R_0(x_1-y_1,x_2+y_2)$ for some function $R_0(z_1,z_2)\in C^{\infty}(\mathbb{R}\times \mathbb{R}_{-})$ with $\mathbb{R}_-:=(-\infty,0)$.
Thus for any $m\in \mathbb{N}_{0}$ and $\alpha_{1},\alpha_{2}\in N_{0}$ with $\alpha_{1}+\alpha_{2}=m$
and for any  $h_{0},h_{1},h_{2}>0$ with $h_1<h_2$,
$
\partial_{z_1}^{\alpha_1}\partial_{z_2}^{\alpha_2}R_0(z_1,z_2)$ is bounded in $\{(z_1,z_2):|z_1|\leq h_{0} ,-h_{2} \leq z_2 \leq -h_1\}$.
By choosing $h_{0}=\pi,h_{1}=2|f_{+}|,h_{2}=2|f_{-}|$, we obtain that for any $\alpha_{1},\ldots,\alpha_{4}\in \mathbb{N}_{0}$ with $\alpha_{1}+\alpha_{2}+\alpha_{3}+\alpha_{4}=n$,
\begin{equation*}
|\partial_{x_1}^{\alpha_1}\partial_{x_2}^{\alpha_2}\partial_{y_1}^{\alpha_3}\partial_{y_2}^{\alpha_4} G_{\mathcal{R}}(x,y)|=|\partial_{x_1}^{\alpha_1}\partial_{x_2}^{\alpha_2}\partial_{y_1}^{\alpha_3}\partial_{y_2}^{\alpha_4} R_0(x_1-y_1,x_2+y_2)|\leq C
\end{equation*}
for any $x,y\in \mathbb{R}^{2}_{-}$ satisfying $|x_1-y_1|\leq \pi$ and $ 2|f_{+}|\leq |x_2+y_2|\leq 2|f_{-}|$,
where $C$ is a constant depending only on $c_1,c_2,k_{\pm},n$.
Hence, combining the above analysis and the definition of $b^{*}$, we deduce that $b^{*}\in C^{n}(\mathbb{R}^2)$ and
\begin{equation}
\left|\frac{\partial^{j+l}b^{*}(s,t)}{\partial s^j \partial t^l}\right|\leq C,\quad s,t\in \mathbb{R},~|s-t|\leq \pi,\label{eq190}
\end{equation}
for all $j,l\in \mathbb{N}_{0}$ with $j+l\leq n$, where the constant $C>0$ depends only on $c_1,c_2,k_{\pm},\eta,n$.

Third, we show that $\kappa_{f}$ satisfies (\ref{eq89}).  In fact, it is clear from Theorem \ref{thm38} {\rm (i)} that for $x,y\in  \overline{\mathbb{R}_{-}^{2}}$,
    \begin{equation*}
        \left|G(x,y)\right| \leq C (1+|x_2|)(1+|y_2|) \left( |x-y|^{-\frac{3}{2}}+ |x-y'|^{-\frac{3}{2}}\right)
             \quad \text{ for } x\neq y~\textrm{and}~x\neq y'.
    \end{equation*}
    This, together with the regularity estimates for solutions to elliptic partial differential equations (see \cite[Theorem 3.9]{GT1983}) and the symmetry property  $G(x,y)=G(y,x)$ for $x,y\in \mathbb{R}^2\backslash \Gamma_0$ with $x\neq y$ (see \cite[(2.28)]{P2016}), implies that for any $\delta>0$, $m\in \mathbb{N}^{+}$ and $\alpha_1,\alpha_2,\alpha_3,\alpha_4\in \mathbb{N}^{+}\cup \{0\}$ with $\alpha_1+\alpha_2+\alpha_3+\alpha_4=m$,
    \begin{equation*}
|\partial_{x_1}^{\alpha_1}\partial_{x_2}^{\alpha_2}\partial_{y_1}^{\alpha_3}\partial_{y_2}^{\alpha_4} G(x,y)| \leq C (1+|x_2|)(1+|y_2|) \left( |x-y|^{-\frac{3}{2}}+ |x-y'|^{-\frac{3}{2}}\right)
    \end{equation*}
    for all $x,y\in \Gamma$ satisfying $|x-y|>\delta$ and $|x-y'|>\delta$, where the constant $C$ depends only on $\delta,m,c_1$. Furthermore, since $f\in BC^{n+2}(\mathbb{R})$, we have that $\sqrt{1+f'(s)}$ and $\nu((s,f(s)))$ belong to $BC^{n}(\mathbb{R})$. Thus it follows from the definition of the kernel $\kappa_{f}$ in \eqref{eq162} that
\begin{equation}
\left|\frac{\partial^{j+l} \kappa_{f}(s, t)}{\partial s^{j} \partial t^{l}}\right| \leq C(1+|s-t|)^{-\frac{3}{2}}, \quad s, t \in \mathbb{R},~|s-t| \geq \pi,\label{eq165}
\end{equation}
for any $j,l\in\mathbb{N}_0$ satisfying $j+l\leq n$, where the constant $C>0$ depends only on $c_1,c_2,k_{\pm},\eta,n$. 

\textbf{Part 2}: we consider the kernel $\kappa_{\tilde{\beta},f}$. With a slight abuse of notations, define
\begin{equation*}
    a^{*}(s,t):=-2a_3(s,t)+2ik_{-}\tilde{\beta}a_1(s,t)
\end{equation*}
and
\begin{equation*}
    b^{*}(s,t):=-2b_3(s,t)+2ik_{-}\tilde{\beta}b_1(s,t)=\kappa_{\tilde{\beta},f}(s,t)-a^{*}(s,t)\ln(|s-t|)
\end{equation*}
for $s,t\in \mathbb{R},s\neq t$. For $a^{*}$ and $b^{*}$ given in this part, since $\tilde{\beta}\in BC^{n}(\mathbb{R})$, we can use similar arguments as in Part 1 to obtain that $a^{*},b^{*}\in C^{n}(\mathbb{R}^2)$ and that the estimates \eqref{eq164} and \eqref{eq190} also hold for $a^{*}$ and $b^{*}$, where the constant $C$ depends on $c_1,c_2,d_1,d_2,k_{\pm},n$. Moreover, by the definition of $k_{\tilde{\beta},f}$ in \eqref{eq163}, the estimates in \eqref{eq165} also hold for $k_{\tilde{\beta},f}$, where the constant $C$ depends only on $c_1,c_2,d_1,d_2,k_{\pm},n$.
Therefore, the proof is complete.
\end{proof}

As a direct consequence of Theorems \ref{thm31}, \ref{thm32} and \ref{thm19}, we can apply \cite[Theorems 2.1 and 3.13]{MCK2000} to obtain the following two theorems on the convergence of our Nystr\"{o}m method.
\begin{theorem}\label{hthm1}
    Suppose $n\in \mathbb{N}^{+}$, $k_{+}>k_{-}>0$ and $c_1<0,c_2>0$. Let $\eta>0$. Then there exist $\tilde{N}\in \mathbb{N}^{+}$ and a constant $C>0$ such that for any integer $N\geq \tilde{N}$ and $f\in B_{n}(c_1,c_2)$, the equation {\rm (\ref{eq72})} has a unique solution $\tilde{\psi}^{D}_{N}$ and we have
    \begin{equation*}\big\Vert\tilde{\psi}^D-\tilde{\psi}^D_N\big\Vert_{\infty}\leq CN^{-n}\big\Vert \tilde{g}\big\Vert_{BC^n(\mathbb{R})}, \end{equation*}
    where $\tilde{\psi}^D$ is  the unique solution of the equation \eqref{eq35}.
\end{theorem}

\begin{theorem}\label{hthm2}
    Suppose $n\in \mathbb{N}^{+}$, $k_{\pm}>0$ and $c_1<0,c_2>0,d_1> 0,d_2>0$. Then there exist $\tilde{N}\in \mathbb{N}^{+}$ and a constant $C>0$ such that for any integer $N\geq \tilde{N}$ and for $f\in B_{n}(c_1,c_2)$ and $\tilde{\beta}\in E_{n}(d_1,d_2)$, the equation {\rm (\ref{eq73})} has a unique solution $\tilde{\psi}^{I}_{N}$ and we have
    \begin{equation*}\big\Vert\tilde{\psi}^I-\tilde{\psi}^I_N\big\Vert_{\infty}\leq CN^{-n}\big\Vert \tilde{g}\big\Vert_{BC^n(\mathbb{R})}, \end{equation*}
    where $\tilde{\psi}^I$ is the unique solution of the equation \eqref{eq37}.
\end{theorem}

\subsection{Numerical implementation}
\label{section15}
Now we describe the numerical implementation of our Nystr\"{o}m method.
With the benefit of the convergence results given in Theorems \ref{hthm1} and \ref{hthm2}, the main part of our method is to numerically solve the discretized equations (\ref{eq72}) and (\ref{eq73}) instead of solving the equations (\ref{eq35}) and (\ref{eq37}). For the integrals arising in (\ref{eq72}) and (\ref{eq73}), we truncate the infinite interval $(-\infty,+\infty)$
into a finite interval $[-T, T]$ with $T$ satisfying $T/h\in \mathbb{N}^{+}$, where $h$ is given as in Section \ref{section12}. That is, the integral operators
$\tilde{S}_{N}$, $\tilde{K}_{N}$ and $\tilde{K}_{N}'$
defined in (\ref{eq79}) are approximated by
\begin{equation*}
    (W\psi)(s)\approx \sum_{j=-T/h}^{j=T/h}\alpha^{N,i}_{j}(s)\psi(t_{j}),\quad s\in \mathbb{R},
\end{equation*}
for $(W,i)=(\tilde{S}_{N},1)$, $(\tilde{K}_{N},2)$ and $(\tilde{K}_{N}',3)$, respectively.
Then by using these approximations and choosing $s= t_{j}$ for $j=-T/h, -T/h+1, \ldots,T/h$ in (\ref{eq72}) and (\ref{eq73}), the equations (\ref{eq72}) and (\ref{eq73}) are reduced to two finite linear systems
which can be solved to obtain the approximate values of the density functions $\tilde{\psi}^D$ and $\tilde{\psi}^I$ at the points $s= t_{j}$ ($j=-T/h, -T/h+1, \ldots,T/h$).
Finally, by using the relations in
\eqref{eq168} and \eqref{eq158} and the formulas (\ref{eq16}) and (\ref{eq27}), we can apply the trapezoidal rule to calculate the approximate values of the scattered wave $u^s$ at any points in the domain $D$.
It is worth mentioning that in our method, we use the approach given in \cite[Section 2.3.5]{P2016} to compute the two-layered Green function $G(x,y)$ with high accuracy.
Moreover, to accelerate our algorithm, we divide the matrices involved in the above linear systems into several matrices with smaller sizes and compute these small matrices in parallel.

In the rest of this subsection, numerical experiments will be carried out to demonstrate the feasibility of our Nystr\"{o}m method for the problems (DBVP) and (IBVP).
For the problem (DBVP), the parameter $\eta$ involved in relevant
integral equations is set to be $\eta=\sqrt{k_+ k_-}$.

To investigate the performance of our method,
we will choose a sampling set $S$ in $D$ with finite points and define the following relative error:
for any wave field $w$,
\begin{equation*}
    E_{rel}(w):=\frac{\|w^{num}-w\|_{\infty,S}}{\|w\|_{\infty,S}},
    % \label{eq157}
\end{equation*}
where $w^{num}$ is a vector consisting of the approximate values of $w$ at the points in $S$ by using our Nystr\"{o}m method
and $\|\cdot \|_{\infty,S}$ denotes the infinity norm for any function defined in $S$.

\begin{figure}[!htbp]
    \subfigure[]{
      \begin{minipage}[b]{0.45\textwidth}
        \centering
        \includegraphics[width=\textwidth]{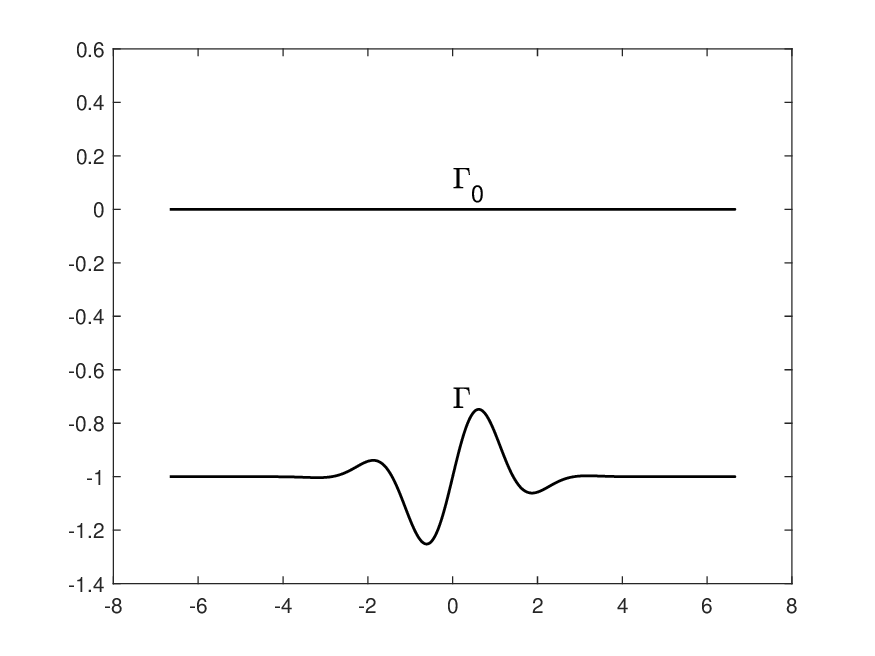}
      \end{minipage}}%
    \hfill
    \subfigure[]{
      \begin{minipage}[b]{0.45\textwidth}
        \centering
        \includegraphics[width=\textwidth]{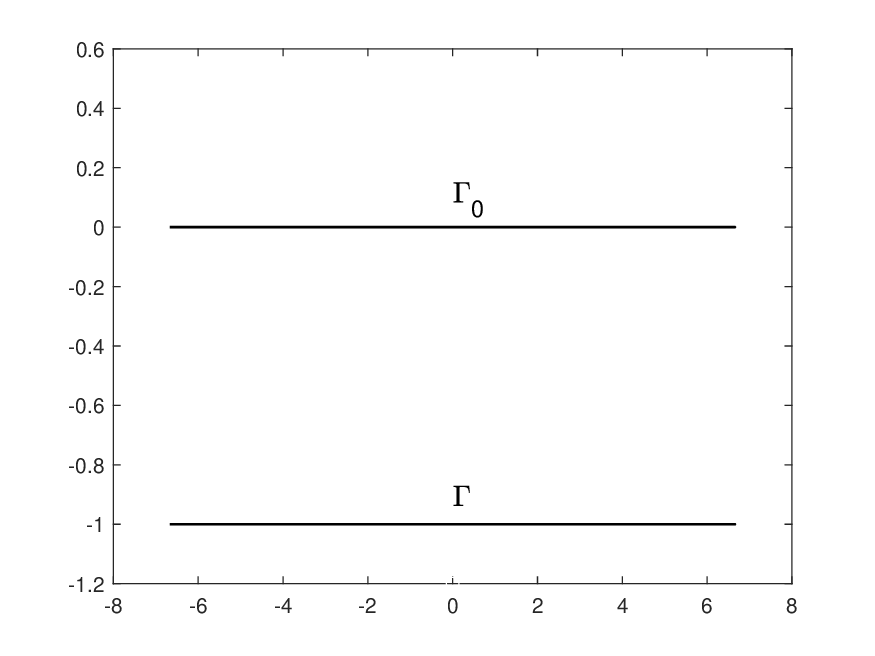}
      \end{minipage}}
    \subfigure[]{
      \begin{minipage}[b]{0.45\textwidth}
        \centering
        \includegraphics[width=\textwidth]{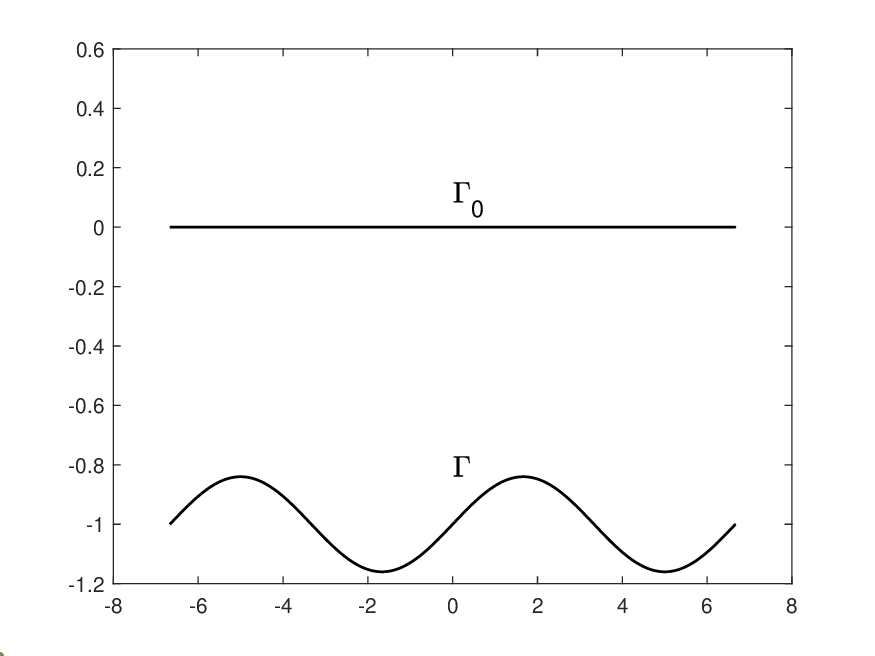}
      \end{minipage}}%
\hfill
    \subfigure[]{
      \begin{minipage}[b]{0.45\textwidth}
        \centering
\includegraphics[width=\textwidth]{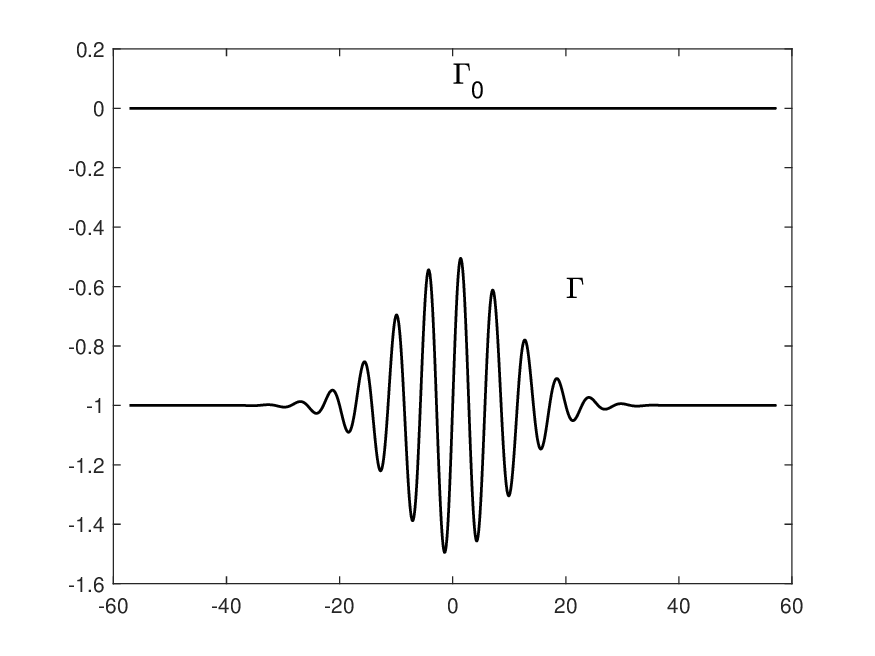}
      \end{minipage}}
\caption{(a)--(d) show the geometries of Examples 1--4, respectively.}
  \label{fig2}
  \end{figure}

\textbf{Example 1.} Consider the rough surface $\Gamma$ with (see Figure \ref{fig2} (a))
\begin{equation*}
f(t)=-1+0.3\sin(0.7\pi t)e^{-0.4t^2}.
\end{equation*}
We choose the wave numbers $k_{+}=3.5$, $k_{-}=2.7$ and choose $S$ to be the set of $100$ points uniformly distributed on the line segment $\{(x_1,0.56):|x_1|\leq 30\}$.
In the first case, we consider the problem (DBVP).
Let $u^s_{dir}$ be the solution of the problem (DBVP) with the boundary data $g = v|_\Gamma$, where
$v(x) := G(x,y_0)$ denotes the two-layered Green function at the source point $y_0=(1, -1.3)\in \mathbb{R}^2\backslash \overline{D}$. It is easily verified that $u^s_{dir}=v$ in $D$ and thus the exact values of $u^s_{dir}$ can be obtained by using the approach given in \cite[Section 2.3.5]{P2016}.
The second and third columns of Table \ref{t1} (a)   present the relative errors $E_{rel}(u^s_{dir})$ of our method for $T=20\pi$ and $T=40\pi$, respectively, with $N=8, 16, 32, 64$.
In the second case, we consider the problem (IBVP).
We choose $\beta \equiv 1$.
Let $u^s_{imp}$ be the solution of the problem (IBVP) with the boundary data $g = (\partial v/\partial \nu - ik_{-}\beta v)|_\Gamma$, where $v$ is given as above.
It is also easily verified that $u^s_{imp}=v$ in $D$ and thus the exact values of $u^s_{imp}$ can also be obtained as in the first case.
The second and third columns of Table \ref{t1} (b)  present the relative errors $E_{rel}(u^s_{imp})$ of our method for $T=20\pi$ and $T=40\pi$, respectively, with $N=8, 16, 32, 64$.
It can be seen from Table \ref{t1} that the relative errors are smaller as $N$ becomes larger.
Moreover, it can be observed from Table \ref{t1} that for each problem and for sufficiently large $N$, the relative error for the case $T=40\pi$ is smaller than that for the case $T=20\pi$.

\begin{table}
    \centering
    \begin{minipage}[t]{0.45\textwidth}
    \begin{tabular}{|lr|l|l|}
        \hline
        && {$T=20\pi$} & $T=40\pi$ \\ \hline
        & N   &  {$E_{rel}(u^s_{dir})$} &  $E_{rel}(u^s_{dir})$ \\
        $k_{+}=3.5$ & 8   & {0.0015} &    0.0015 \\
        $k_{-}=2.7$ & 16  & {2.7362e-06} & 3.3423e-06 \\
    & 32  & {9.3889e-07} & 8.8875e-07\\
    & 64  & {5.7465e-07} & 9.8288e-08 \\
    \hline
        \end{tabular}
        \caption*{(a) The problem (DBVP)}
    \end{minipage}
    \hfill
    \begin{minipage}[t]{0.45\textwidth}
    \begin{tabular}{|lr|l|l|}
     \hline
     & & {$T=20\pi$} & $T=40\pi$ \\ \hline
     & N   & {$E_{rel}(u^s_{imp})$} & $E_{rel}(u^s_{imp})$ \\
    $k_{+}=3.5$ & 8   & {0.0043} &    0.0043 \\
    $k_{-}=2.7$ & 16  & {8.2946e-06} & 8.6324e-06 \\
& 32  & {7.9876e-07} & 5.2844e-07 \\
& 64  & {4.7788e-07} & 1.8623e-07 \\ \hline
    \end{tabular}
    \caption*{(b) The problem (IBVP)}
    \end{minipage}
    \caption{Relative errors against $N$ for Example 1.}
 \label{t1}
\end{table}

\textbf{Example 2.}
Consider the rough surface $\Gamma$ as the flat plane $x_2=-1$ (see Figure \ref{fig2} (b)). We choose the wave numbers $k_{+}=3.5$, $k_{-}=2.7$ and choose $S$ to be the set of 100 points uniformly distributed on the line segment $\{(x_1,-0.2):|x_1|\leq 30\}$.
We set $\theta_{d}=\frac{4}{3}\pi$.
For such $\theta_{d}$,
let
$d$, $d_{r}$, $d_{t}$, the plane wave $u_{pl}^{i}(x)$ and the reference wave $u^{0}_{pl}(x)$ be given as in Section \ref{section2}.
Here, note that since $|(k_-/k_+)^{-1}\cos(\theta_d)|<1$,
then $d^t=(\cos(\theta^t_d),\sin(\theta^t_d))$ with $\theta^t_d\in[\pi,2\pi]$ satisfying $\cos(\theta^t_d)=(k_-/k_+)^{-1}\cos(\theta_d)$.
Further, let $d_n=(\cos(\theta^t_d),-\sin(\theta^t_d))$ be the reflection of $d_t$ about the $x_1$-axis.
In the first case, we consider the problem (DBVP).
Let $u^{s}_{dir}$ be the solution of the problem (DBVP) with the boundary data
$g = -u^{0}_{pl}|_\Gamma$.
Then it can be seen from Section \ref{section2} that $u^{tot}_{dir}(x):=u^{0}_{pl}(x)+u^s_{dir}(x)$ is
the total field of the scattering problem (\ref{eq2})--(\ref{eq3})
with the sound-soft boundary $\Gamma$
and with
the incident wave $u^i(x)$ given by $u_{pl}^{i}(x)$.
Moreover,
since the rough surface $\Gamma$ is a plane, the total field $u^{tot}_{dir}$ has the analytical expression (see \cite[Section 2.1.3]{C1999})
\begin{equation}
    u^{tot}_{dir}(x) = \begin{cases}
        e^{ik_{+}x\cdot d} + \lambda_{d_r} e^{ik_{+}x\cdot d_{r}},\quad x_2>0,\\
        \lambda_{d_t}e^{ik_{-}x\cdot d_t} + \lambda_{d_n} e^{ik_{-}x\cdot d_{n}},\quad -1\leq x_2\leq0,
    \end{cases}
    \label{eq58}
\end{equation}
with $\lambda_{d_r},\lambda_{d_t},\lambda_{d_n}\in\mathbb{C}$ satisfying the system of linear equations
\begin{equation*}
 \begin{bmatrix}
-1 & 1 & 1 \\
k_{+}\sin (\theta_{d}) & k_{-}\sin (\theta_{d}^{t}) & -k_{-}\sin (\theta_{d}^{t}) \\
0 & \exp(-ik_{-}\sin (\theta_{d}^{t})) & \exp(ik_{-}\sin (\theta_{d}^{t}))
\end{bmatrix}
\begin{bmatrix}\lambda_{d_r}\\
\lambda_{d_t}\\
\lambda_{d_n}
\end{bmatrix}=\begin{bmatrix}
1\\
k_{-}\sin (\theta_{d})\\
0
\end{bmatrix},
\end{equation*}
which is due to the transmission condition \eqref{eq172} of $u^{tot}_{dir}$ on $\Gamma_{0}$  and the Dirichlet boundary condition $u^{tot}_{dir}=0$ on $\Gamma$.
Thus we can obtain the exact values of $u^{tot}_{dir}(x)$ by the above formulas.
The second and third
columns of Table \ref{t2} (a) present the relative errors $E_{rel}(u^{tot}_{dir})$ for $T=20\pi$ and $T=40\pi$, respectively, with $N=8,16,32,64$,
where the approximate values of $u^{tot}_{dir}(x)$ are obtained by applying our method to the numerical computations of
$u^{s}_{dir}(x)$.
In the second case, we consider the problem (IBVP).
We choose $\beta \equiv 1$ on $\Gamma$.
Let $u^{s}_{imp}$ be the solution to the problem (IBVP) with $g = -\partial u^{0}_{pl}/\partial\nu|_\Gamma + ik_{-}\beta u^{0}_{pl}|_\Gamma$.
Then it can be seen from Section \ref{section2} that $u^{tot}_{imp}(x):=u^0_{pl}(x)+u^s_{imp}(x)$ is the total field of the scattering problem (\ref{eq2})--(\ref{eq4}) with the impedance boundary $\Gamma$ and with
the incident wave $u^i(x)$ given by  $u_{pl}^{i}(x)$.
Similarly to the first case, $u^{tot}_{imp}$ has the analytical expression \eqref{eq58}
with $\lambda_{d_r},\lambda_{d_t},\lambda_{d_n}\in\mathbb{C}$ satisfying the system of linear equations
\begin{equation*}
\begin{bmatrix}
-1 & 1 & 1 \\
k_{+}\sin (\theta_{d}) & k_{-}\sin (\theta_{d}^{t}) & -k_{-}\sin (\theta_{d}^{t}) \\
0 & -\sin(\theta_{d}^{t})-\beta & (\sin (\theta_{d}^{t})-\beta)\exp(2ik_{-}\sin (\theta_{d}^{t}))
\end{bmatrix}\begin{bmatrix}\lambda_{d_r}\\
\lambda_{d_t}\\
\lambda_{d_n}
\end{bmatrix}=\begin{bmatrix}
1\\
k_{-}\sin (\theta_{d})\\
0
\end{bmatrix},
\end{equation*}
which is due to the transmission condition \eqref{eq172} of $u^{tot}_{imp}$ on $\Gamma_{0}$  and the impedance boundary condition $\partial u^{tot}_{imp} / \partial \nu - ik_{-}\beta u^{tot}_{imp}=0$ on $\Gamma$.
Thus we can also obtain the exact values of $u^{tot}_{imp}(x)$.
The second and third columns of Table \ref{t2} (b) show the relative errors  $E_{rel}(u^{tot}_{imp})$ for  $T=20\pi$ and $T=40\pi$, respectively, with $N=8,16,32,64$,
where the approximate values of $u^{tot}_{imp}(x)$ are obtained by applying our method to the numerical computations of
$u^{s}_{imp}(x)$.
It is shown in Table \ref{t2} that the relative errors for the case $T=40\pi$ are smaller than those for the case $T=20\pi$.

\begin{table}
    \centering
    \begin{minipage}[t]{0.45\textwidth}
    \begin{tabular}{|lr|l|l|}
    \hline
    & & {$T=20\pi$} & $T=40\pi$ \\ \hline
& N &  {$E_{rel} (u^{tot}_{dir})$}   &    $E_{rel}(u^{tot}_{dir})$ \\
$k_{+}=3.5$ & 8 & {6.6499e-04}
 &   2.0781e-04 \\
$k_{-}=2.7$ & 16 &  {4.0849e-04}
 &  1.0260e-04\\
 & 32 & {2.4464e-04}&  6.0844e-05
 \\
 & 64 &  {1.5477e-04} &  3.7765e-05
  \\ \hline
    \end{tabular}
    \caption*{(a) The problem (DBVP)}
    \end{minipage}
    \hfill
    \begin{minipage}[t]{0.45\textwidth}
    \begin{tabular}{|lr|l|l|}
    \hline
    & & {$T=20\pi$} & $T=40\pi$ \\  \hline
& N &  {$E_{rel}(u^{tot}_{imp})$} & $E_{rel}(u^{tot}_{imp})$      \\
$k_{+}=3.5$ & 8 & {1.6147e-04} & 4.4780e-05 \\
$k_{-}=2.7$ & 16 &  {9.0964e-05}& 2.4418e-05 \\
 & 32 & {5.3144e-05}&  1.4069e-05 \\
 & 64 &  {3.2544e-05}&  8.4220e-06  \\ \hline
    \end{tabular}
    \caption*{(b) The problem (IBVP)}
    \end{minipage}
    \caption{Relative errors against $N$ for Example 2.}
    \label{t2}
\end{table}

\textbf{Example 3.}
Consider the rough surface $\Gamma$ to be a periodic curve with (see Figure \ref{fig2} (c))
\begin{equation*}
    f(t) = -1+0.16\sin(0.3\pi t).
\end{equation*}
We choose the wave numbers $k_{+}=4,k_{-}=3$ and choose
$S$ to be the set of $100$ points uniformly  distributed on the line segment $\{(x_1,0.3):|x_1|\leq 30\}$.
Moreover, we set $T=40\pi$.
Let the plane wave $u_{pl}^{i}(x)$ and the reference wave $u^{0}_{pl}(x)$ be given as in Section \ref{section2}, where $\theta_d$ is chosen to be $17/12\pi$.
Let $u_{dir}^{s}$ be the solution of the problem (DBVP) with the boundary data $g=-u^{0}_{pl}|_\Gamma$ and let $u_{imp}^{s}$ be the solution of the problem (IBVP) with $g = -\partial u^{0}_{pl}/\partial\nu|_\Gamma + ik_{-}\beta u^{0}_{pl}|_\Gamma$, where we choose $\beta \equiv 1$ on $\Gamma$.
Then similarly to Example 2, $u_{dir}^{tot}(x):=u^{0}_{pl}(x)+u_{dir}^{s}(x)$ (resp. $u^{tot}_{imp}(x):=u^{0}_{pl}(x)+u^{s}_{imp}(x)$) is the total field of the scattering problem (\ref{eq2})--(\ref{eq3}) with the sound-soft boundary $\Gamma$ (resp. the scattering problem (\ref{eq2})--(\ref{eq4}) with the impedance boundary $\Gamma$), where the incident wave $u^i(x)$ is given by $u^{i}_{pl}(x)$. In this example, we compute the approximate values of $u_{dir}^{tot}(x)$ and $u^{tot}_{imp}(x)$ by our method, which are obtained in a same way as in Example 2.

Note that the rough surface $\Gamma$ and $u^i_{pl}(x)e^{-ik_{+}\cos(\theta_{d})x_1}$ for $x\in\Gamma$ have the same period $L_{p} = 2/0.3$ in $x_1$-direction.
Thus, to test the performance of our method,
we model the considered scattering problems as the quasi-periodic scattering problems (see, e.g., \cite{CW2003}) and then use the finite element method with the technique of perfectly
matched layer (PML) to compute the PML solutions $u^{tot}_{dir,PML}$ and $u^{tot}_{imp,PML}$, which are the approximations of $u_{dir}^{tot}$ and $u^{tot}_{imp}$, respectively. To be more specific, we use the PML technique given in \cite{CW2003} with the following settings.
The problem is solved in the domain $\Omega_{PML}:=\{(x_1,x_2)\,:\,x_1\in (0,L_{p}), f(x_1) < x_2< 3\}$.
The PML layer is chosen to be $\{(x_1,x_2):x_1\in(0,L_p), 3<x_2<3+\delta\}$ with the thickness of the PML layer $\delta=2.243995$ as suggested in \cite[Section 6]{CW2003}.
The number of nodal points is chosen to be $335497$ with using uniform mesh refinement.
The finite element method with the PML technique is implemented by the open-source software freeFEM++ (see \cite{H2012}).
By the above approach, we can compute the approximate values of
$u^{tot}_{dir,PML}$ and $u^{tot}_{imp,PML}$ in the domain $\Omega_{PML}$. Moreover, the approximate values of $u^{tot}_{dir,PML}$ and $u^{tot}_{imp,PML}$ in the domain $\Omega_{L}:=\{(x_1,x_2)\,:\,x_1\in (-\infty,+\infty), f(x_1) < x_2< 3\}$ can be obtained by using the quasi-periodic properties of $u^{tot}_{dir,PML}$ and $u^{tot}_{imp,PML}$ (see \cite{CW2003}), i.e.,
  \begin{align*}
  &u^{tot}_{dir,PML}(x)e^{-ik_{+}\cos(\theta_{d})x_1}=u^{tot}_{dir,PML}(x+(L_p,0))e^{-ik_{+}\cos(\theta_{d})(x_1+L_p)},\\
  &u^{tot}_{imp,PML}(x)e^{-ik_{+}\cos(\theta_{d})x_1}=u^{tot}_{imp,PML}(x+(L_p,0))e^{-ik_{+}\cos(\theta_{d})(x_1+L_p)}
  \end{align*}
for any $x=(x_1,x_2)$ in $\Omega_{L}$.

Table \ref{t3} shows the approximate values of  $u^{tot}_{dir}(x)$ and $u^{tot}_{imp}(x)$ with $N=8,16,32,64$ by using our Nystr\"{o}m method as well as the approximate values of $u^{tot}_{dir,PML}(x)$ and $u^{tot}_{imp,PML}(x)$ by using the finite element method with the PML technique, where $x$ is chosen to be the point $(10,0.3)$.  Figure \ref{fig3} presents the real parts of the approximate values of $u^{tot}_{dir}(x)$ and $u^{tot}_{imp}(x)$
by using our Nystr\"{o}m method with $N=64$ (blue circles) as well as the real parts of the approximate values of $u^{tot}_{dir,PML}(x)$ and $u^{tot}_{imp,PML}(x)$ by using the finite element method with the PML technique (orange dots), where we choose $x$ to be
some discrete points on the segment line $\{(x_1,0.3):0\leq x_1\leq 10\}$.

\begin{table}[!hbtp]
    \centering
    \begin{tabular}{|lr|ll|ll|}
    \hline
              &   &\multicolumn{2}{|c}{The problem (DBVP)}                  & \multicolumn{2}{|c|}{The problem (IBVP)}                \\ \hline
              & N & $\textrm{Re}(u^{tot}_{dir})$    &  $\textrm{Im}(u^{tot}_{dir})$    & $\textrm{Re}(u^{tot}_{imp})$    &  $\textrm{Im}(u^{tot}_{imp})$    \\
    $k_{+}=4$ & 8 & 1.62243876
    & 0.57170715 &    0.35175005      & 0.84347716
    \\
    $k_{-}=3$ & 16 & 1.62236304
    &  0.57170811 &  0.35174104      & 0.84346688
    \\
              & 32 & 1.62228626
              &  0.57168368 & 0.35173401                & 0.84346096
              \\
              & 64 &  1.62224382  & 0.57167564
              & 0.35172998        & 0.84345808
              \\\hline
    && $\textrm{Re}(u^{tot}_{dir,PML})$    & $\textrm{Im}(u^{tot}_{dir,PML})$ &$\textrm{Re}(u^{tot}_{imp,PML})$&$\textrm{Im}(u^{tot}_{imp,PML})$\\
    &PML solution        &1.62194914     & 0.56924364
    &  0.35164894      & 0.84255734
    \\ \hline
    \end{tabular}
    \caption{The approximate values of  $u^{tot}_{dir}(x)$ and $u^{tot}_{imp}(x)$ in Example 3 by using our Nystr\"{o}m method as well as the approximate values of $u^{tot}_{dir,PML}(x)$ and $u^{tot}_{imp,PML}(x)$ in Example 3 by using the finite element method with the PML technique, where $x$ is chosen to be the point $(10,0.3)$.}
    \label{t3}
\end{table}

\begin{figure}[!htbp]
    \subfigure[$\mathrm{Re}(u^{tot}_{dir})$ and $\mathrm{Re}(u^{tot}_{dir,PML}$)]{
      \begin{minipage}[b]{0.45\textwidth}
        \centering
        \includegraphics[width=\textwidth]{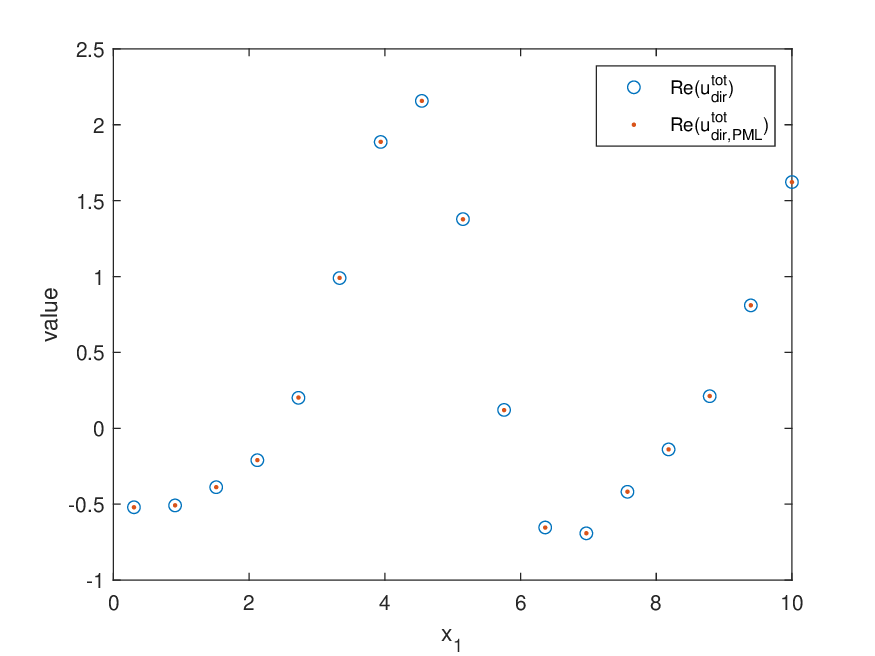}
      \end{minipage}}%
    \hfill
    \subfigure[$\mathrm{Re}(u^{tot}_{imp})$ and $\mathrm{Re}(u^{tot}_{imp,PML}$)]{
      \begin{minipage}[b]{0.45\textwidth}
        \centering
        \includegraphics[width=\textwidth]{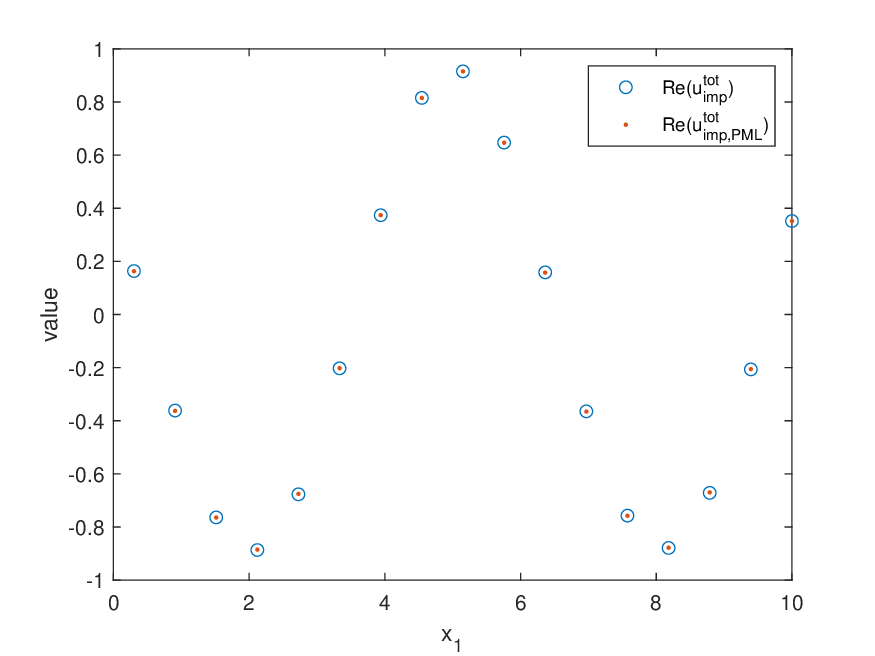}
      \end{minipage}}
    % \subfigure[subtitle]{
    %   \begin{minipage}[b]{0.45\textwidth}
    %     \centering
    %     \includegraphics[width=\textwidth]{e3.eps}
    %   \end{minipage}}%
    % \hfill
    % \subfigure[subtitle]{
    %   \begin{minipage}[b]{0.45\textwidth}
    %     \centering
    %     \includegraphics[width=\textwidth]{example/example1/Reconstructed_curve_k_13.eps}
    %   \end{minipage}}%
\caption{The real parts of the approximate values of $u^{tot}_{dir}(x)$ and $u^{tot}_{imp}(x)$ in Example 3
by using our Nystr\"{o}m method with $N=64$ as well as the real parts of the approximate values of $u^{tot}_{dir,PML}(x)$ and $u^{tot}_{imp,PML}(x)$ in Example 3 by using the finite element method with the PML technique, where we choose $x$ to be
some discrete points on the segment line $\{(x_1,0.3):0\leq x_1\leq 10\}$. Here, the wave numbers $k_{+}=4$ and $k_{-}=3$.}
  \label{fig3}
  \end{figure}

\textbf{Example 4.} Consider the rough surface $\Gamma$ (see Figure \ref{fig2} (d)) with \begin{equation*}
      f(t)=-1+0.5\sin(0.35\pi t)\exp(-0.005t^2).
\end{equation*}
We choose the wave numbers $k_{+}=3.5$, $k_{-}=2.7$ and choose $S$ to be the set of $100$ points uniformly distributed on the line segment $\{(x_1,0.5)\,:\,|x_1|\leq 30\}$.
Let $d:=(\cos(\theta_d),\sin(\theta_d))$ with $\theta_d= 5\pi/12$ be the incident direction. Similarly to the reference wave defined in Section \ref{section2}, we introduce the reference wave $v_{l}$ that is generated by the incident plane wave $e^{ik_{-}x\cdot d}$ propagating in the lower half space $\mathbb{R}^2_{-}$ and that satisfies the Helmholtz equations as well as the transmission boundary condition in \eqref{hw-eq16}. Similarly to the Fresnel formulas given in \eqref{hw-eq17}, $v_l$ is given by
\begin{equation}
    v_{l}(x) := \begin{cases} \tilde{\mathcal{T}}(\pi-\theta_d)e^{ik_{+}x\cdot d^{t}},\quad  x\in \mathbb{R}^2_+, \\e^{ik_{-}x\cdot d}+\tilde{\mathcal{R}}(\pi-\theta_d)e^{ik_{-}x\cdot d^{r}}, \quad x\in \mathbb{R}^2_-,
    \end{cases}\label{eq173}
\end{equation}
where $d^r:=(\cos(\theta_d),-\sin(\theta_d))$ and $d^t:=(\cos(\theta_{d}^{t}),\sin(\theta_{d}^{t}))$ with $\theta_d^t\in (0,\pi)$ satisfying $\cos(\theta_d^t)=n\cos(\theta_d)$ and where the coefficient functions $\tilde{\mathcal{T}}$ and $\tilde{\mathcal{R}}$ are defined by
 \begin{equation*}
    \tilde{\mathcal{R}}(\theta):=\frac{i \sin \theta+\mathcal{S}(\cos \theta, 1/n)}{i \sin \theta-\mathcal{S}(\cos \theta, 1/n)}, \quad \tilde{\mathcal{T}}(\theta):=\tilde{\mathcal{R}}(\theta)+1 \quad \text{for } \theta \in \mathbb{R}.
 \end{equation*}
In the first case, we consider the problem (DBVP). Let $u^{s}_{dir}$ be the solution of the problem (DBVP) with the boundary data $g=v_{l}|_{\Gamma}$. Then it is easily verified that $u^{s}_{dir}=v_{l}$ in $D$ and thus the exact values of $u^{s}_{dir}$ can be computed by \eqref{eq173}.
The second and third columns of Table \ref{t4} (a) present the relative errors $E_{rel}(u^{s}_{dir})$ of our method for $T=20\pi$ and $T=40\pi$, respectively, with $N=8,16,32,64$.
In the second case, we consider the problem (IBVP). Let $u^{s}_{imp}$ be the solution of the problem (IBVP) with the boundary data $g=(\partial v_{l}/\partial \nu-ik_{-}\beta v_{l})|_{\Gamma}$, where we choose  $\beta \equiv 0.5-0.5i$ on $\Gamma$. Then it is also easily verified that $u^{s}_{imp}=v_{l}$ in $D$ and thus the exact values of $u^{s}_{imp}$ can be obtained as in the first case.
The second and third columns of Table \ref{t4} (b) present the relative errors $E_{rel}(u^{s}_{imp})$ of our method for $T=20\pi$ and $T=40\pi$, respectively, with $N=8,16,32,64$.
From Table \ref{t4}, it can be observed that the relative errors for the case $T=40\pi$ are smaller  than those for the case $T=20\pi$.

\begin{table}
    \centering
    \begin{minipage}[t]{0.45\textwidth}
    \begin{tabular}{|lr|l|l|}
    \hline
    & &{$T=20\pi$} & $T=40\pi$ \\ \hline
& N &  {$E_{rel}(u^{s}_{dir})$}   &    $E_{rel}(u^{s}_{dir})$   \\
$k_{+}=3.5$ & 8 &  {0.0013} &   4.4373e-04
\\
$k_{-}=2.7$ & 16 &  {8.2514e-04} &   2.2806e-04
 \\
 & 32 &  {5.5794e-04} &  1.5310e-04
 \\
 & 64 & {4.0977e-04} &  1.1097e-04
  \\ \hline
    \end{tabular}
    \caption*{(a) The problem (DBVP)}
    \end{minipage}
    \hfill
    \begin{minipage}[t]{0.45\textwidth}
    \begin{tabular}{|lr|l|l|}
    \hline
&  & {$T=20\pi$}    &   $T=40\pi$    \\ \hline
& N & {$E_{rel}(u^{s}_{imp})$} & $E_{rel}(u^{s}_{imp})$ \\
$k_{+}=3.5$ & 8 &  {0.0032} & 9.0459e-04
\\
$k_{-}=2.7$ & 16 &  {0.0020} &   5.4954e-04 \\
& 32 &  {0.0014} &  3.7795e-04 \\
 & 64 & {0.0011} &  2.9280e-04  \\ \hline
    \end{tabular}
    \caption*{(b) The problem (IBVP)}
    \end{minipage}
    \caption{Relative errors against $N$ for Example 4.}
    \label{t4}
\end{table}

\section{Conclusions}\label{section13}

In this paper, we investigated the problems of scattering of time-harmonic acoustic waves by a
two-layered medium with a rough boundary.
We have formulated the considered scattering problems as the boundary value problems and  proved the uniqueness and existence results of each boundary value problem by utilizing the integral
equation method associated with the two-layered Green function.
Moreover, we have developed the
Nystr\"{o}m method for numerically solving the considered boundary value problems and established the convergence results of our Nystr\"{o}m method.
It is worth noting that in establishing the well-posedness of the considered boundary value
problems as well as the convergence results of our Nystr\"{o}m method, an essential role has been played
by the investigation of the asymptotic properties of the two-layered Green function for small
and large arguments.
We should mention that the numerical results presented in Tables \ref{t1}, \ref{t2} and \ref{t4} do not fully support the convergence
results established in Theorems \ref{hthm1} and \ref{hthm2}.
We think the main reason is the presence of truncation errors in the numerical implementation of our Nystr\"{o}m method since in order to numerically compute the discretized equation \eqref{eq72} and \eqref{eq73},
we need to truncate the infinite interval $(-\infty,+\infty)$ into a finite interval $[-T,T]$.
Actually, it is shown in Tables \ref{t1}, \ref{t2} and \ref{t4} that for sufficiently large $N$, the relative errors of wave fields computed by our Nystr\"{o}m method for the case $T=40\pi$ are smaller than those for the case $T=20\pi$.
In the future, we hope to study the convergence rates for the numerical solutions of the truncated forms
of \eqref{eq72} and \eqref{eq73} as well as their dependence on the truncation parameter $T$.
Furthermore, it is interesting to study uniqueness and numerical algorithms of the inverse problems for the considered scattering models, which will be our future work.

\section*{Acknowledgments}

The work of H. Liu and J. Yang is partially supported by National Key R\&D Program of China (2022YFA1005102) and
the National Science Foundation of China (11961141007). The work of H. Zhang is partially supported by Beijing Natural
Science Foundation Z210001, the NNSF of China grant 12271515, and the Youth Innovation Promotion Association CAS. The work of B. Zhang is partially supported by the NNSF of China grant 12431016.

\begin{appendices}
\renewcommand{\theequation}{\Alph{section}.\arabic{equation}}

\section{Proofs of Lemma {\ref{thm34}} and Theorem \ref{thm28}}\label{section14}

\begin{proof}[Proof of Lemma {\ref{thm34}}] Let $x_0,y_0\in \mathbb{R}^2$ be arbitrarily fixed. Our proof is divided into three parts.

\textbf{Part 1}: we  prove that $R(x,y)$ is continuous at $(x_0,y_0)$. Denote the ball centered at $z\in \mathbb{R}^2$ with radius $r$ by $B_{r}(z)$. Choose the cutoff function $\chi\in C_{c}^{\infty}(\mathbb{R}^2)$ such that
\begin{equation*}
\chi(x) = \begin{cases}1,\quad x\in  B_{2\epsilon}(y_{0}),\\
0,\quad x\in \mathbb{R}^2\backslash \overline{B_{4\epsilon}(y_0)},
\end{cases}
\end{equation*}
with $\epsilon>0$ being a fixed number.
Let $P(x,y):=G(x,y)-\chi(x)G_{0}(x,y)$. Then we have
\begin{equation*}
\Delta_{x}P(x,y)+k^2(x)P(x,y)=-k^2(x)\chi(x) G_{0}(x,y)-\Delta_{x}((\chi(x)-1)G_{0}(x,y))=:f_0(x,y).
\end{equation*}
For any $y\in B_{\epsilon}(y_0)$ and $p\in(1,+\infty)$, we can easily verify that $\|f_{0}(\cdot,y)\|_{L^{p}(\mathbb{R}^2)}\leq C_{p,\epsilon}$ for some constant $C_{p,\epsilon}>0$.
Furthermore, for any $y_1,y_2\in B_{\epsilon}(y_0)$ and $p\in(1,+\infty)$,  we can easily prove that \begin{equation}
\|f_{0}(\cdot,y_1)-f_{0}(\cdot,y_2)\|_{L^{p}(\mathbb{R}^2)}\to 0 \text{ as } y_1\to y_2.\label{eq179}
\end{equation}
Let $K\subset\mathbb{R}^2$ be a bounded domain with $C^\infty$-boundary $\partial K$.
By the well-posedness of the scattering problem in a two-layered medium (see \cite{BHY2018}) and the interior regularity of the elliptic equation (see \cite[Sections 6.2 and 6.3]{E2010}), it follows that for $y\in B_{\epsilon}(y_0)$,
\begin{equation}
\|P(\cdot,y)\|_{H^2(K)}\leq C\|f_{0}(\cdot,y)\|_{L^2(\mathbb{R}^2)}\leq C_{2,\epsilon,K}\label{eq174}
\end{equation}
for some constants $C,C_{2,\epsilon,K}>0$,
which implies that
\begin{equation}
P(\cdot,y)\in C(\overline{K}),\quad\|P(\cdot,y)\|_{C(\overline{K})}\leq C_{2,\epsilon,K}.\label{eq177}
\end{equation}
Similarly to the above derivations, we can also obtain that for $y_1,y_2\in B_{\epsilon}(y_0)$,
\begin{equation}
\|P(\cdot,y_1)-P(\cdot,y_2)\|_{C(\overline{K})}\leq C\|P(\cdot,y_1)-P(\cdot,y_2)\|_{H^{2}(K)}\leq C\|f_{0}(\cdot,y_1)-f_{0}(\cdot,y_2)\|_{L^2(K)}.\label{eq178}
\end{equation}
Thus by using \eqref{eq177}, \eqref{eq178} and \eqref{eq179}, we can deduce that $P(x,y)$ is continuous at $(x_0,y_0)$. Hence $R(x,y)$ is continuous at $(x_0,y_0)$ due to the fact that $R(x,y)=(\chi(x) -1)G_0(x,y)+P(x,y)$.

\textbf{Part 2}: we prove that $\nabla_{x}R(x,y)$ is continuous at $(x_0,y_0)$. To do this, we utilize the $L^p$ estimates of the elliptic equation. Choose the cutoff function $\eta\in C_{c}^{\infty}(\mathbb{R}^2)$ such that
\begin{equation*}
\eta(x):=\begin{cases}1,\quad x\in B_{2\epsilon}(x_0),\\
0,\quad x\in\mathbb{R}^2\backslash \overline{B_{4\epsilon}(x_0)},
\end{cases}
\end{equation*}
with a fixed number $\epsilon>0$.
Let $P_{\eta}(x,y):=\eta(x)P(x,y)$, where $P(x,y)$ is given as in Part 1. Then we have
\begin{align*}
\Delta_{x}P_{\eta}(x,y)&=\Delta_{x} \eta(x) P(x,y)+2\nabla_{x} \eta(x) \cdot \nabla_{x} P(x,y)+\eta(x) \Delta_{x} P(x,y)\\
&=[(\Delta_{x} \eta(x)-k^2\eta(x)) P(x,y)+2\nabla_{x}\eta(x)\cdot \nabla_{x} P(x,y)]+\eta(x) f_{0}(x,y)\\
&=:f_1(x,y)+\eta(x) f_0(x,y)\\
&=:f_2(x,y).
\end{align*}
By the Gagliardo--Nirenberg--Sobolev inequality (see \cite[Theorem 1 in Section 5.6.1]{E2010}), it follows that
for any $p>2$ and any $y\in B_{\epsilon}(y_0)$,
\begin{align*}
\|f_1(\cdot,y)\|_{L^{p}(\mathbb{R}^2)}&\leq C\|\nabla_{x}f(\cdot,y)\|_{L^{p^*}(\mathbb{R}^2)}\\
&\leq C \|\nabla_{x}f(\cdot,y)\|_{L^{p^*}(B_{4\epsilon}(x_0))}\\
&\leq C \|f_1(\cdot,y)\|_{H^1(B_{4\epsilon}(x_0))}\\
& \leq C \|P(\cdot,y)\|_{H^2(B_{4\epsilon}(x_0))}
\end{align*}
for some constant $C>0$,
where $p^*$ satisfies $1/p=1/p^*-1/2$ (it is clear that $p^*\in (1,2)$).
This, together with \eqref{eq174}, implies that for any $p>2$ and any $y\in B_{\epsilon}(y_0)$,
\begin{equation}
\|f_1(\cdot,y)\|_{L^{p}(\mathbb{R}^2)}\leq C\|P(\cdot,y)\|_{H^{2}(B_{4\epsilon}(x_0))}\leq C\|f_{0}(\cdot,y)\|_{L^2(\mathbb{R}^2)}.\label{eq175}
\end{equation}
Similarly, it follows that for any $p>2$ and $y_1,y_2\in B_{\epsilon}(y_0)$,
\begin{equation}
\|f_1(\cdot,y_1)-f_1(\cdot,y_2)\|_{L^{p}(\mathbb{R}^2)}\leq C\|f_0(\cdot,y_1)-f_{0}(\cdot,y_2)\|_{L^{2}(\mathbb{R}^2)}.\label{eq176}
\end{equation}
It is easy to see that $P_{\eta}(\cdot,y)\in H^2(B_{4\epsilon}(x_0))$ with $y\in B_{\epsilon}(y_0)$ is the solution of the following problem
\begin{equation*}
\begin{cases}\Delta w(x)=f_2(x,y)\quad \textrm{in}~ B_{4\epsilon}(y_0),\\
w(x)=0\quad \textrm{on}~ \partial B_{4\epsilon}(x_0).
\end{cases}
\end{equation*}
Then it can be deduced from \cite[Theorem 9.15]{GT1983} that
\begin{equation*}
P_{\eta}(\cdot,y)\in W^{2,p}(B_{4\epsilon}(x_0))
\end{equation*}
for any $p\in (2,\infty)$ and $y\in B_{\epsilon}(y_0)$.
Furthermore, applying the Sobolev inequality given in \cite[Theorem 6 in Section 5.6.3]{E2010}, Lemma 9.17 in \cite{GT1983} and the inequality \eqref{eq175}, we obtain that for any $p>2$ and $y\in B_{\epsilon}(y_0)$,
$P_{\eta}(\cdot,y)\in C^{1}\big(\overline{B_{4\epsilon}(x_0)}\big)$ with
\begin{align*}
\|P_{\eta}(\cdot,y)\|_{C^{1}\big(\overline{B_{4\epsilon}(x_0)}\big)}&\leq C\|P_{\eta}(\cdot,y)\|_{W^{2,p}(B_{4\epsilon}(x_0))}\notag\\
&\leq C \|f_{0}(\cdot,y)\|_{L^{2}(\mathbb{R}^2)}+C\|f_{0}(\cdot,y)\|_{L^{p}(\mathbb{R}^2)}\notag\\
&\leq C\|f_0(\cdot,y)\|_{L^{p}(B_{4\epsilon(y_0)})}\leq C_{p,\epsilon}
% \label{eq180}
\end{align*}
for some constants $C,C_{p,\epsilon}$.
Similarly, we can apply \cite[Lemma 9.17]{GT1983} and the inequality \eqref{eq176} to obtain that for any $p>2$ and $y_1,y_2\in B_{\epsilon}(y_0)$,
\begin{align}
\|P_{\eta}(\cdot,y_1)-P_{\eta}(\cdot,y_2)\|_{C^1(\overline{B_{4\epsilon}(x_0)})}&\leq C\|P_{\eta}(\cdot,y_1)-P_{\eta}(\cdot,y_2)\|_{W^{2,p}(B_{4\epsilon}(x_0))}\nonumber\\
&\leq C  \|f_{0}(\cdot,y_1)-f_{0}(\cdot,y_2)\|_{L^{p}(B_{4\epsilon}(x_0))}\label{eq181}
\end{align}
for some constant $C>0$.
Hence, by using the formulas \eqref{eq181} and  \eqref{eq179}
and the fact that $P_{\eta}(\cdot,y)\in C^{1}\big(\overline{B_{4\epsilon}(x_0)}\big)$, we have that $\nabla_{x}P_{\eta}(x,y)$ is continuous at $(x_0,y_0)$. This, together with the definitions of $P$ and $R$, implies that $\nabla_{x}R(x,y)$ is continuous at $(x_0,y_0)$.

\textbf{Part 3}: we prove that $\nabla_{y} R(x,y)$ is continuous at $(x_0,y_0)$.
It is known from \cite[(2.28)]{P2016} that $G(x,y)=G(y,x)$ for $x,y\in \mathbb{R}^2\backslash \Gamma_0$ with $x\neq y$. It was also proved in \cite[Remark 3.5]{LYZZ2022} that $G(\cdot,y)\in C^1(\mathbb{R}^2\backslash \{y\})$ for any $y\in \mathbb{R}^2$ and $G(x,\cdot)\in C^1(\mathbb{R}^2\backslash \{x\})$ for any $x\in \mathbb{R}^2$. Thus it follows that $G(x,y)=G(y,x)$ for any $x,y\in \mathbb{R}^2$ with $x\neq y$. This, together with the facts that $G_0(x,y)=G_0(y,x)$ for any $x,y\in\mathbb{R}^2$ with $x\neq y$ and $R(x,y)\in C(\mathbb{R}^2\times \mathbb{R}^2)$ (see the result in Part 1),
implies that $R(x,y)=R(y,x)$ for any $x,y\in \mathbb{R}^2$.
Hence, we can apply the result in Part 2 to obtain that $\nabla_{y} R(x,y)$ is continuous at $(x_0,y_0)$.

Therefore, the proof is complete due to the arbitrariness of $x_0,y_0$.
\end{proof}

We now prove Theorem \ref{thm28}. To do this, we need some notations and lemmas.
Define the angle $\theta_c:=\mathrm{arccos}(n)$ if $0<n<1$, where $n= k_-/k_+$ is given as in Section \ref{section2}.
For any $R_{0}>0$, define $B_{R_0}:=\{y\in \mathbb{R}^2\,:\,|y|<R_0\}$ and $B_{R_{0}}^{\pm}:=\{y\in \mathbb{R}_{\pm}^{2}:|y|<R_{0}\}$.
The following lemma gives the asymptotic properties of $G_{\mathcal{D},\kappa}$.

\begin{lemma}\label{thm41} Assume $\kappa>0$ and let $R_{0} > 0 $ be  an arbitrary fixed number. Suppose that $y= (y_1,y_2)$ and $y'=(y_1,-y_2)$ for $y \in \mathbb{R}^2$ and suppose that $ x = \hat{x}|x| = |x|(\cos \theta_{\hat{x}},\sin\theta_{\hat{x}})$ with $\theta_{\hat{x}}\in [0,2\pi)$ for $x\in \mathbb{R}^2$ with $|x|\neq 0$. Then we have the asymptotic behaviors
    \begin{equation*}
        \begin{aligned}
        &G_{\mathcal{D},\kappa}(x,y) = \frac{e^{i\kappa|x|}}{\sqrt{|x|}}\frac{e^{i\frac{\pi}{4}}}{\sqrt{8\pi \kappa}}\Big( e^{-i\kappa\hat{x}\cdot y}-e^{-i\kappa\hat{x}\cdot y'}\Big) +G_{\mathcal{D},\kappa,Res,1}(x,y),  \\
        &\nabla_{y} G_{\mathcal{D},\kappa}(x,y) = \frac{e^{i\kappa|x|}}{\sqrt{|x|}}e^{-i\frac{\pi}{4}}\sqrt{\frac{\kappa}{8\pi}}\bigg(e^{-i\kappa\hat{x}\cdot y}\bigg(\genfrac{}{}{0pt}{0}{\cos\theta_{\hat{x}}}{\sin\theta_{\hat{x}}}\bigg)-e^{-i\kappa\hat{x}y'}\bigg(\genfrac{}{}{0pt}{0}{\cos\theta_{\hat{x}}}{-\sin\theta_{\hat{x}}}\bigg) \bigg)+G_{\mathcal{D},\kappa,Res,2}(x,y),
        \end{aligned}
    \end{equation*}
    where $G_{\mathcal{D},\kappa,Res,1}$, $G_{\mathcal{D},\kappa,Res,2}$ satisfy
    \begin{equation*}
|G_{\mathcal{D},\kappa,Res,1}(x,y)|,~|G_{\mathcal{D},\kappa,Res,2}(x,y)|\leq C_{R_{0}}|x|^{-\frac{3}{2}},\quad |x|\to \infty,
    \end{equation*}
    uniformly for all $\theta_{\hat{x}}\in [0,2\pi)$ and $y\in B_{R_0}$. Here, the constant $C_{R_{0}}$ is independent of $x$ and $y$ but dependent of $R_{0}$.
\end{lemma}
\begin{proof} The statement of this lemma is a direct consequence of the following asymptotic behaviors of the Hankel function $H_{0}^{(1)}$ (see (3.105) in \cite{CK2019})
    \begin{equation*}
        \begin{aligned}
            &H_{0}^{(1)}(t) = \sqrt{\frac{2}{\pi t}}e^{i(t-\frac{\pi}{4})}\bigg\{1+\mathcal{O}(\frac{1}{t})\bigg\},\quad t\to \infty,\\
            &\frac{d}{dt}H_{0}^{(1)}(t) = \sqrt{\frac{2}{\pi t}}e^{i(t+\frac{\pi}{4})}\bigg\{1+\mathcal{O}(\frac{1}{t})\bigg\},\quad t\to \infty.
        \end{aligned}
    \end{equation*}
\end{proof}

The following lemma provides the uniform far-field asymptotics of some functions relevant to the two-layered Green function $G$, which are mainly based on the work \cite{LYZZ2022}.

\begin{lemma}
    \label{thm30}Assume $k_{+}>k_{-}>0$ and let $R_{0}>0$ be an arbitrary fixed number. Suppose that $y = (y_1,y_2)$ and $y'=(y_1,-y_2)$ for $y\in \mathbb{R}^2$ and suppose that $x=(x_1,x_2) = \hat{x}|x| = |x|(\cos\theta_{\hat{x}},\sin\theta_{\hat{x}})$ with $\theta_{\hat{x}}\in (0, \pi)\cup (\pi,2\pi)$ for $x\in \mathbb{R}^{2}_+\cup\mathbb{R}^{2}_-$.
    Define
    \begin{equation*}
        H(x,y) :=\begin{cases} G(x,y) - G_{\mathcal{D},k_{+}}(x,y), &x\in \mathbb{R}^2_{+},y \in \overline{\mathbb{R}^{2}_{+}},\\
                    G(x,y), &x\in \mathbb{R}^{2}_{+}, y\in \mathbb{R}^2_{-} ,
                \end{cases}
    \end{equation*}
    and
    \begin{equation*}
        I(x,y) := \begin{cases} G(x,y) - G_{\mathcal{D},k_{-}}(x,y), &x\in \mathbb{R}^2_{-},y\in \overline{\mathbb{R}_{-}^{2}},\\
            G(x,y), &x\in \mathbb{R}_{-}^2,y\in \mathbb{R}_{+}^{2}.
        \end{cases}
    \end{equation*}
    Then we have the following statements.

    {\rm (i)} For $\theta_{\hat{x}}\in (0,\pi)$, we have the asymptotic behaviors
    \begin{align*}
        H(x,y)=\frac{e^{i k_{+}|x|}}{\sqrt{|x|}} H_{1}^{\infty}(\hat{x},y)+H_{1,Res}(x, y),\\
        \nabla_{y} H(x,y)=\frac{e^{i k_{+}|x|}}{\sqrt{|x|}} H_{2}^{\infty}(\hat{x},y)+H_{2,Res}(x, y),
    \end{align*}
    where $H_{1}^{\infty}$ and $H_{2}^{\infty}$ are given by
    \begin{equation*}
        H_{1}^{\infty}(\hat{x},y)  :=\frac{e^{i\frac{\pi}{4}}}{\sqrt{8\pi k_{+}}}\begin{cases}\mathcal{T}(\theta_{\hat{x}})e^{-ik_{+}\hat{x}\cdot y'},& \hat{x}\in \mathbb{S}^{1}_{+},y\in \overline{\mathbb{R}_{+}^{2}}, \\
             \mathcal{T}(\theta_{\hat{x}})e^{-ik_{+}(y_1\cos\theta_{\hat{x}}+iy_2\mathcal{S}(\cos\theta_{\hat{x}},n))},& \hat{x}\in \mathbb{S}^{1}_{+},y\in \mathbb{R}_{-}^{2},
        \end{cases}
    \end{equation*}
    \begin{equation*}
    H_{2}^{\infty}(\hat{x},y):=e^{-i\frac{\pi}{4}}\sqrt{\frac{k_{+}}{8\pi}}
    \begin{cases}\mathcal{T}(\theta_{\hat{x}})e^{-ik_{+}\hat{x}\cdot y'}\begin{pmatrix}
\cos \theta_{\hat{x}} \\
-\sin \theta_{\hat{x}}
\end{pmatrix}^{T},& \hat{x}\in \mathbb{S}^{1}_{+},y\in \overline{\mathbb{R}_{+}^{2}},\\
\mathcal{T}(\theta_{\hat{x}})e^{-ik_{+}(y_1\cos \theta_{\hat{x}}+iy_2 \mathcal{S}(\cos \theta_{\hat{x}},n))} \begin{pmatrix}\cos \theta_{\hat{x}},\\
i\mathcal{S}(\cos \theta_{\hat{x}},n)
\end{pmatrix}^{T},& \hat{x}\in \mathbb{S}^{1}_{+},y\in \mathbb{R}_{-}^{2},
    \end{cases}
    \end{equation*}
    and where $H_{1,Res}$ and $H_{2,Res}$ satisfy the estimates
    \begin{equation*}
        |H_{1,Res}(x,y)|, ~|H_{2,Res}(x,y)|\leq C_{R_0} |x|^{-3/4},\quad |x|\to\infty,
    \end{equation*}
    uniformly for all $\theta_{\hat{x}}\in (0,\pi)$ and $y\in B_{R_0}$,
    \begin{equation*}
        |H_{1,Res}(x,y)|, ~|H_{2,Res}(x,y)|\leq C_{R_0}|\theta_{c}-\theta_{\hat{x}}|^{-\frac{3}{2}}|x|^{-\frac{3}{2}},\quad |x|\to \infty,
    \end{equation*}
    uniformly for all $\theta_{\hat{x}}\in (0,\theta_c)\cup (\theta_c,\pi/2)$ and $y\in B_{R_0}$, and
    \begin{equation*}
       |H_{1,Res}(x,y)|, ~|H_{2,Res}(x,y)|\leq C_{R_0}|\pi-\theta_c-\theta_{\hat{x}}|^{-\frac{3}{2}}|x|^{-\frac{3}{2}},\quad |x|\to \infty,
    \end{equation*}
    uniformly for all $\theta_{\hat{x}}\in [\pi/2,\pi-\theta_c)\cup (\pi-\theta_c,\pi)$ and $y\in B_{R_0}$. Here, the constant $C_{R_0}$ is independent of $x$ and $y$ but dependent of $R_{0}$.

{\rm (ii)} For $\theta_{\hat{x}}\in (\pi,2\pi)$, we have the asymptotic behaviors
\begin{align*}
    I(x,y)=\frac{e^{ik_{-}|x|}}{\sqrt{|x|}}I_{1}^{\infty}(\hat{x},y)+I_{1,Res}(x,y),\\
    \nabla_{y}I(x,y)=\frac{e^{ik_{-}|x|}}{\sqrt{|x|}}I_{2}^{\infty}(\hat{x},y)+I_{2,Res}(x,y),
\end{align*}
where $I_{1}^{\infty}$ and $I_{2}^{\infty}$ are given by
\begin{equation*}
    I_{1}^{\infty}(\hat{x},y) := \frac{e^{i\frac{\pi}{4}}}{\sqrt{8\pi k_{-}}}\begin{cases}\frac{2i\sin\theta_{\hat{x}}}{i\sin\theta_{\hat{x}}+\mathcal{S}(\cos\theta_{\hat{x}},1/n)}e^{-ik_{-}(y_1\cos \theta_{\hat{x}}-iy_2\mathcal{S}(\cos \theta_{\hat{x}},1/n))},& \hat{x}\in \mathbb{S}^{1}_{-}, y\in \mathbb{R}_{+}^{2},\\
\frac{2i\sin\theta_{\hat{x}}}{i\sin \theta_{\hat{x}}+\mathcal{S}(\cos\theta_{\hat{x}},1/n)}e^{-ik_{-}\hat{x}\cdot y'},& \hat{x}\in \mathbb{S}^{1}_{-}, y\in \overline{\mathbb{R}_{-}^{2}},
    \end{cases}
\end{equation*}
\begin{equation*}
I^{\infty}_{2}(\hat{x},y):=e^{-i\frac{\pi}{4}}\sqrt{\frac{k_{-}}{8\pi}}\left\{
\begin{array}{ll}
\frac{2i\sin\theta_{\hat{x}}}{i\sin \theta_{\hat{x}}+\mathcal{S}(\cos\theta_{\hat{x}},1/n)} e^{-ik_{-}(y_1 \cos \theta_{\hat{x}}-iy_2 \mathcal{S}(\cos \theta_{\hat{x}},1/n))}
&\begin{pmatrix}
\cos \theta_{\hat{x}}\\
-i\mathcal{S}(\cos \theta_{\hat{x}},1/n)
\end{pmatrix}, \\  &\hat{x}\in \mathbb{S}^{1}_{-}, y\in \mathbb{R}_{+}^{2}, \\ \frac{2i\sin\theta_{\hat{x}}}{i\sin \theta_{\hat{x}}+\mathcal{S}(\cos\theta_{\hat{x}},1/n)} e^{-ik_{-}\hat{x}\cdot y'} \begin{pmatrix}\cos \theta_{\hat{x}}\\
-\sin \theta_{\hat{x}}
\end{pmatrix}^{T},
\quad & \hat{x}\in \mathbb{S}^{1}_{-}, y\in \overline{\mathbb{R}_{-}^{2}},
\end{array}\right.
\end{equation*}
and where $I_{1,Res}$ and $I_{2,Res}$ satisfy the estimates
\begin{equation*}
    |I_{1,Res}(x,y)|,~|I_{2,Res}(x,y)|\leq C_{R_{0}}|x|^{-\frac{3}{2}},\quad |x|\to +\infty,
\end{equation*}
uniformly for all $\theta_{\hat{x}}\in (\pi,2\pi)$ and $y\in B_{R_0}$. Here, the constant $C_{R_0}$ is independent of $x$ and $y$ but dependent of $R_{0}$.
\end{lemma}

\begin{proof} The statement of this lemma is a direct consequence of Lemma \ref{thm41} and \cite[Theorems 2.14 and 3.2 and Remark 3.5]{LYZZ2022}.
\end{proof}

\begin{remark}\label{hw-re1}
By \eqref{eq63}, \eqref{eq156} and Lemma \ref{thm34}, $H$ and $I$ can be rewritten as follows:
\begin{equation*}
        H(x,y) =\begin{cases}
            \frac{1}{2\pi}\int_{-\infty}^{+\infty}\frac{e^{-ik_{+}(z y_1-i\mathcal{S}(z,1)y_2)}}{\mathcal{S}(z,1)+\mathcal{S}(z,n)}e^{ik_{+}(z x_1+i\mathcal{S}(z,1) x_2)}dz, & x\in \mathbb{R}^{2}_{+},y \in \overline{\mathbb{R}^{2}_{+}},\\
            \frac{1}{2 \pi} \int_{-\infty}^{\infty} \frac{e^{-ik_{+}(z y_1+i\mathcal{S}(z,n)y_2)}}{\mathcal{S}\left(z, 1\right)+\mathcal{S}\left(z, n\right)} e^{ik_{+}(zx_1+i\mathcal{S}(z,1)x_2)} d z,& x\in \mathbb{R}^{2}_{+}, y\in \mathbb{R}^2_{-},
        \end{cases}
        % \label{eq170}
\end{equation*}
and
\begin{equation*}
        I(x,y)=
        \begin{cases}
            \frac{1}{2 \pi} \int_{-\infty}^{\infty} \frac{e^{-ik_{+}(zy_1-i\mathcal{S}(z,n)y_2)}}{\mathcal{S}\left(z, 1\right)+\mathcal{S}\left(z, n\right)} e^{ik_{+}(zx_1-i\mathcal{S}(z,1)x_2)} d z,& x\in \mathbb{R}^{2}_{-}, y\in \mathbb{R}^2_{+}, \\
            \frac{1}{2\pi}\int_{-\infty}^{+\infty}\frac{e^{-ik_{+}(zy_1+i\mathcal{S}(z,n)y_2)}}{\mathcal{S}(z,1)+\mathcal{S}(z,n)}e^{ik_{+}(zx_1-i\mathcal{S}(z,n)x_2)}dz,& x\in \mathbb{R}_{-}^{2},y \in \overline{\mathbb{R}^{2}_{-}}.
        \end{cases}
        % \label{eq171}
\end{equation*}
\end{remark}

We are now ready to prove Theorem \ref{thm28}.

\begin{proof}[Proof of Theorem {\ref{thm28}}]
We only give the derivations on the estimates of $G_{\mathcal{P}}(x,y)$ and $G_{\mathcal{Q}}(x,y)$ by using the asymptotic behaviors of $H(x,y)$ and $I(x,y)$ given in Lemma \ref{thm30}
and the continuity of $R(x,y)$ given in  Lemma \ref{thm34}.
We omit the proof on the estimates of
$\nabla_y G_{\mathcal{P}}(x,y)$ and
$\nabla_y G_{\mathcal{Q}}(x,y)$, since these estimates can be similarly deduced by using the asymptotic behaviors of $\nabla_y H(x,y)$ and $\nabla_y I(x,y)$ given in Lemma \ref{thm30} as well as the continuity of $\nabla_y R(x,y)$ given in  Lemma \ref{thm34}.
Our proof is divided into the following three parts.

\textbf{Part 1:} we establish the estimates for $G_{\mathcal{P}}$ when $k_{+}>k_{-}$.
In this part, we consider three steps.

{\bf Step 1.1:} we prove that there exists some $\delta>0$ such that
\begin{equation}
    |G_{\mathcal{P}}(x,y)|\leq C(1+|x_2|+|y_2|)|x-y'|^{-\frac{3}{2}} \label{eq148}
\end{equation}
for all $x,y\in \mathbb{R}_{+}^{2}$ with $|x-y'|\geq \delta$, where $C$ is a constant depending only on $k_{\pm}$.

By taking the substitution $\xi = k_{+}z$ in (\ref{eq126}), $G_{\mathcal{P}}$ can be rewritten as
\begin{equation*}
G_{\mathcal{P}}(x,y)=\frac{1}{2\pi}\int_{-\infty}^{+\infty}\frac{1}{\mathcal{S}(z,1)+\mathcal{S}(z,n)}e^{ik_{+}\left(z(x_1-y_1)+i\mathcal{S}(z,1)(x_2+y_2)\right)}dz
\end{equation*}
for $x,y\in \mathbb{R}^2_{+}$.
This, together with Remark \ref{hw-re1}, implies that for $x, y\in \mathbb{R}_{+}^{2}$,
\begin{equation*}
G_{\mathcal{P}}(x,y) = H(x-y',(0,0)),
\end{equation*}
where $x-y'=(x_1-y_1,x_2+y_2)=|x-y'|(\cos(\theta_{\widehat{x-y'}}), \sin(\theta_{\widehat{x-y'}}))$ with $\theta_{\widehat{x-y'}}\in (0,\pi)$.
Then it follows from Lemma \ref{thm30} that
\begin{equation}
    G_{\mathcal{P}}(x,y) = \frac{e^{i k_{+}|x-y'|}}{\sqrt{|x-y'|}} \frac{e^{i\frac{\pi}{4}}}{\sqrt{8\pi k_{+}}}\frac{2i\sin \theta_{\widehat{x-y'}}}{i\sin \theta_{\widehat{x-y'}} -\mathcal{S}(\cos \theta_{\widehat{x-y'}}, n)}+G_{Res,a}(x,y)\label{eq135}
\end{equation}
for $x, y\in \mathbb{R}_{+}^{2}$, where $G_{Res,a}$ satisfies
\begin{equation}
    |G_{Res,a}(x,y)|\leq C |x-y'|^{-\frac{3}{4}},\quad |x-y'|\to\infty, \label{eq131}
\end{equation}
uniformly for all $\theta_{\widehat{x-y'}}\in (0,\pi)$,
\begin{equation}
    |G_{Res,a}(x,y)|\leq C|\theta_{c}-\theta_{\widehat{x-y'}}|^{-\frac{3}{2}}|x-y'|^{-\frac{3}{2}},\quad |x-y'|\to \infty,\label{eq132}
\end{equation}
uniformly for all $\theta_{\widehat{x-y'}}\in (0,\theta_c)\cup (\theta_c,\pi/2)$, and
\begin{equation}
    |G_{Res,a}(x,y)|\leq C |\pi-\theta_c-\theta_{\widehat{x-y'}}|^{-\frac{3}{2}}|x-y'|^{-\frac{3}{2}},\quad |x-y'|\to \infty,\label{eq133}
\end{equation}
uniformly for all $\theta_{\widehat{x-y'}}\in [\pi/2,\pi-\theta_c)\cup (\pi-\theta_c,\pi)$, where $C$ is a constant depending only on $k_{\pm}$.
If $\theta_{\widehat{x-y'}}\in (0,\theta_{c}/2)\cup (\pi-\theta_{c}/2,\pi)$, then we can apply (\ref{eq132}) and (\ref{eq133}) to obtain that there exists $\delta_1>0$ such that
\begin{equation}
    |G_{Res,a}(x,y)|\leq C\Big|\frac{\theta_c}{2}\Big|^{-\frac{3}{2}}|x-y'|^{-\frac{3}{2}}\quad \text{for } |x-y'|\geq \delta_1.\label{hw-eq2}
\end{equation}
Moreover, if
$\theta_{\widehat{x-y'}} \in [\theta_{c}/2, \pi - \theta_{c}/2]$,  then we can apply (\ref{eq131}) and the fact that $|x_1-y_1|\tan(\theta_{c}/2)\leq |x_2+y_2|$ to deduce that there exists $\delta_2>0$ such that
\begin{equation}
    |G_{Res,a}(x,y)|\leq C |x-y'|^{-\frac{3}{4}} \leq C'\frac{x_2+y_2}{|x-y'|^{\frac{3}{2}}} \quad \text{for }|x-y'|\geq \delta_2.\label{hw-eq3}
\end{equation}
Here, the constants $C$, $C'$ in \eqref{hw-eq2} and \eqref{hw-eq3} depend only on $k_{\pm}$.
Hence, combining the estimates \eqref{hw-eq2} and \eqref{hw-eq3}, we have that there exists $\delta:=\max(\delta_1,\delta_2)$ such that
\begin{equation}
|G_{Res,a}(x,y)|\leq C (1+x_2+y_2)|x-y'|^{-\frac{3}{2}}\label{eq136}
\end{equation}
for all $x,y\in \mathbb{R}^2_{+}$ with $|x-y'|\geq \delta$, where the constant $C$ depends only on $k_{\pm}$.

On the other hand, since $\sin \theta_{\widehat{x-y'}} =  (x_2+y_2)/|x-y'|$, it follows from (\ref{eq135}) that
\begin{equation}
    |G_{\mathcal{P}}(x,y)-G_{Res,a}(x,y)|\leq C(x_2+y_2)|x-y'|^{-\frac{3}{2}}\label{eq137}
\end{equation}
for all $x,y\in \mathbb{R}_{+}^{2}$, where $C$ is a constant depending only on $k_{\pm}$.

Utilizing (\ref{eq136}) and (\ref{eq137}), we have that \eqref{eq148}
holds for all $x,y\in \mathbb{R}_{+}^{2}$ with $|x-y'|\geq \delta$, where $\delta$ is given as above.

{\bf Step 1.2:} we prove that there exists some $\delta>0$ such that $G_{\mathcal{P}}$ satisfies (\ref{eq148}) for all $x,y\in \mathbb{R}^2_{-}$  with  $|x-y'|\geq \delta$, where the constant $C$ depends only on $k_{\pm}$.

For $x,y\in \mathbb{R}_{-}^{2}$, we can write $G_{\mathcal{P}}(x,y)$ as
\begin{equation*}
    G_{\mathcal{P}}(x,y) = \frac{1}{2\pi}\int_{-\infty}^{+\infty}\frac{e^{ik_{+}(z(x_1-y_1)-i\mathcal{S}(z,n)(x_2+y_2))}}{\mathcal{S}(z,1)+\mathcal{S}(z,n)}dz.
\end{equation*}
Then we obtain from Remark \ref{hw-re1} that
\[G_{\mathcal{P}}(x,y) = I(x-y',(0,0)).\]
Hence it follows from Lemma \ref{thm30} that
\begin{equation}
    G_{\mathcal{P}}(x,y)=\frac{e^{ik_{-}|x-y'|}}{\sqrt{|x-y'|}}\frac{e^{i\frac{\pi}{4}}}{\sqrt{8\pi k_{-}}}\frac{2i\sin\theta_{\widehat{x-y'}}}{i\sin\theta_{\widehat{x-y'}}+\mathcal{S}(\cos\theta_{\widehat{x-y'}},1/n)} + G_{Res,b}(x,y)\label{hw-eq4}
\end{equation}
for $x,y\in \mathbb{R}_{-}^{2}$, where $G_{Res,b}$ satisfies
\begin{equation*}
    |G_{Res,b}(x,y)| \leq C|x-y'|^{-\frac{3}{2}},\quad |x-y'|\to \infty,
\end{equation*}
uniformly for $\theta_{\widehat{x-y'}}\in (\pi,2\pi)$, where $C$ is a constant depending only on $k_{\pm}$.
This implies that
there exists $\delta>0$ such that
\begin{equation}\label{hw-eq5}
    |G_{Res,b}(x,y)|\leq C|x-y'|^{-\frac{3}{2}}
\end{equation}
for all $x,y\in \mathbb{R}_{-}^{2}$ with $|x-y'|\geq \delta $, where the constant $C$ depends only on $k_{\pm}$.

On the other hand, similarly to Step 1.1, it follows from \eqref{hw-eq4} that
\begin{equation}\label{hw-eq6}
    |G_{\mathcal{P}}(x,y)-G_{Res,b}(x,y)|\leq C(|x_2|+|y_2|)|x-y'|^{-\frac{3}{2}}
\end{equation}
for all $x,y\in \mathbb{R}^{2}_{-}$, where $C$ is a constant depending only on $k_{\pm}$.

By \eqref{hw-eq5} and \eqref{hw-eq6}, we have that $G_{\mathcal{P}}$ satisfies (\ref{eq148}) for all $x,y\in \mathbb{R}^2_{-}$  with  $|x-y'|\geq \delta$, where $\delta$ is given as above.

{\textbf{Step 1.3:}} we prove that for any $\delta_0>0$, there exists a constant $C>0$ depending on $\delta_0$ such that \eqref{eq148} holds for all $x,y$ satisfying $x_2\cdot y_2> 0$ and $|x-y'|\leq \delta_0$.

Recall that $G_{\mathcal{R}}(x,y)=-(i/4) H_{0}^{(1)}(k_{+}|x-y'|)+G_{\mathcal{P}}(x,y)$
   for $x,y\in \mathbb{R}^{2}_{+}$ and $G_{\mathcal{R}}(x,y)=-(i/4) H_{0}^{(1)}(k_{-}|x-y'|)+G_{\mathcal{P}}(x,y)$
   for $x,y\in \mathbb{R}^{2}_{-}$. Then from the equation \eqref{hw-eq9} and Remark \ref{thm35}, it follows that for some function $P$ defined in $\mathbb{R}^2_+\cup\mathbb{R}^2_-$, we can write $G_{\mathcal{R}}$ as $G_{\mathcal{R}}(x,y)=P(x-y')$ for $x,y\in \mathbb{R}^2_+$ and for $x,y\in \mathbb{R}^2_-$.  Using the continuity property of $G_{\mathcal{R}}$ given in Remark \ref{thm35}, it is clear that $P(z)$ can be extended as a function in $C(\overline{\mathbb{R}^2_{+}})\cup C(\overline{\mathbb{R}^2_{-}})$.
    Thus we have that for any $\delta_{0}>0$, there exists a constant $C$ depending only on $\delta_{0}$ and $k_{\pm}$ such that $|G_{\mathcal{R}}(x,y)|\leq C$ for all $x,y$ satisfying $x_2\cdot y_2 > 0$ and $|x-y'|\leq \delta_0$.
   This, together with the asymptotic properties of the Hankel function $H^{(1)}_0$ for small arguments (see \cite{CK2019} for the expression of $H^{(1)}_0$), implies that there exists a constant $C>0$ such that \eqref{eq148} holds for $x,y$ satisfying $x_2\cdot y_2> 0$ and $|x-y'|\leq \delta_0$.

From the discussions in Steps 1.1, 1.2 and 1.3, we obtain that $G_{\mathcal{P}}(x,y)$ satisfies (\ref{eq148}) for all $x,y$ satisfying $x_2\cdot y_2> 0$, where the constant $C$ depends only on $k_{\pm}$.

\textbf{Part II:} we establish the estimates for $G_{\mathcal{Q}}$ when $k_{+}>k_{-}$.
In this part, we consider three steps.

{\textbf{Step 2.1:}} we prove that there exists some $\delta>0$ such that
     \begin{equation}
        |G_{\mathcal{Q}}(x,y)|\leq C\frac{1+x_2}{|\widetilde{x-y}|^{\frac{3}{2}}} \label{eq149}
     \end{equation}
   for all $x,y$ satisfying $x_2>0$, $-h\leq y_2<0$ and $|\widetilde{x-y}| \geq \delta$, where $C$ is a constant depending only on $k_{\pm}$ and $h$.

Suppose that $x,y$ satisfy $x_2>0$ and $-h\leq y_2<0$.
    Let $\widetilde{x-y}=(x_1-y_1,x_2)=|\widetilde{x-y}|(\cos\theta_{\widetilde{x-y}},\sin\theta_{\widetilde{x-y}})$ with $\theta_{\widetilde{x-y}}\in(0,\pi)$ and $\tilde{y}:=(0,y_2)$.
    By the change of variable  $\xi=k_{+}z$,  $G_{\mathcal{Q}}$ can be written as
    \begin{equation*}
        G_{\mathcal{Q}}(x,y)=\frac{1}{2\pi}\int_{-\infty}^{+\infty}\frac{e^{-ik_{+}i\mathcal{S}(z,n)y_2}}{\mathcal{S}(z,1)+\mathcal{S}(z,n)}e^{ik_{+}(z(x_1-y_1)+i\mathcal{S}(z,1)x_2)}dz.
    \end{equation*}
    Then it follows from Remark \ref{hw-re1} that
    \begin{equation*}
        G_{\mathcal{Q}}(x,y) = H(\widetilde{x-y},\tilde{y}).
    \end{equation*}
    Hence using Lemma \ref{thm30}, we obtain that
    \begin{equation}
        G_{\mathcal{Q}}(x,y)=\frac{e^{ik_{+}|\widetilde{x-y}|}}{\sqrt{|\widetilde{x-y}|}}\frac{e^{i\frac{\pi}{4}}}{\sqrt{8\pi k_{+}}}\frac{2i\sin\theta_{\widetilde{x-y}}}{i\sin\theta_{\widetilde{x-y}}-\mathcal{S}(\cos\theta_{\widetilde{x-y}},n)}e^{-ik_{+}iy_2\mathcal{S}(\cos\theta_{\widetilde{x-y}},n)}+G_{Res,c}(x,y)\label{eq138}
    \end{equation}
for $x \in \mathbb{R}_{+}^{2}$, $y\in \mathbb{R}_{-}^{2}$, where $G_{Res,c}$ satisfies
    \begin{equation}
        |G_{Res,c}(x,y)|\leq C|\widetilde{x-y}|^{-\frac{3}{4}},\quad |\widetilde{x-y}|\to +\infty,\label{eq122}
    \end{equation}
     uniformly for all $\theta_{\widetilde{x-y}}\in (0,\pi)$ and $-h\leq y_2\leq 0$,
    \begin{equation}
        |G_{Res,c}(x,y)|\leq C|\theta_{c}-\theta_{\widetilde{x-y}}|^{-\frac{3}{2}}|\widetilde{x-y}|^{-\frac{3}{2}},\quad |\widetilde{x-y}|\to +\infty, \label{eq123}
    \end{equation}
     uniformly for all $\theta_{\widehat{x-y}}\in (0,\theta_{c})\cup (\theta_{c},\pi/2)$ and $-h\leq y_2\leq 0$, and
     \begin{equation}
        |G_{Res,c}(x,y)|\leq C |\pi - \theta_{c} - \theta_{\widetilde{x-y}}|^{-\frac{3}{2}}|\widetilde{x-y}|^{-\frac{3}{2}},\quad |\widetilde{x-y}|\to +\infty, \label{eq124}
     \end{equation}
     uniformly for all $\theta_{\widetilde{x-y}}\in [\pi/2,\pi-\theta_{c})\cup (\pi-\theta_{c},\pi)$ and $-h\leq y_2\leq 0$, where $C$ is a constant depending only on $k_{\pm}$ and $h$.

     If $\theta_{\widetilde{x-y}}\in (\theta_{c}/2, \pi-\theta_{c}/2)$, then we can apply (\ref{eq122}) and the fact that $|x_1-y_1|\tan(\theta_{c}/2)\leq |x_2|$ to obtain that there exists $\delta_1>0$ such that
\begin{equation}\label{hw-eq7}
        |G_{Res,c}(x,y)|\leq C|\widetilde{x-y}|^{-3/4}\leq  C' \frac{x_2}{|\widetilde{x-y}|^{\frac{3}{2}}}
     \end{equation}
     for $|\widetilde{x-y}|\geq \delta_1$.
Moreover, if $\theta_{\widetilde{x-y}}\in (0,\theta_{c}/2]\cup [\pi -\theta_c/2,\pi )$, then we can apply (\ref{eq123}) and (\ref{eq124}) to deduce that there exists $\delta_2>0$ such that
\begin{equation}\label{hw-eq8}
        |G_{Res,c}(x,y)|\leq C \Big|\frac{\theta_c}{2}\Big|^{-\frac{3}{2}}|\widetilde{x-y}|^{-\frac{3}{2}}\leq C'' |x-y|^{-\frac{3}{2}} \quad \text{for }|\widetilde{x-y}|\geq \delta_2.
     \end{equation}
     Here, the constants $C$, $C'$, $C''$ in \eqref{hw-eq7} and \eqref{hw-eq8} depend only on $k_{\pm}$ and $h$.
     Combining \eqref{hw-eq7} and \eqref{hw-eq8}, we have that there exists $\delta:=\max(\delta_1,\delta_2)$ such that
     \begin{equation}
        |G_{Res,c}(x,y)|\leq C\frac{1+x_2}{|x-y|^{\frac{3}{2}}}\label{eq140}
     \end{equation}
     for all $x,y$ satisfying $x_2>0$, $-h\leq y_2<0$ and $|\widetilde{x-y}|\geq \delta$, where $C$ is a constant depending only on $k_{\pm}$ and $h$.

     On the other hand, since $\sin\theta_{\widetilde{x-y}}=x_2/|\widetilde{x-y}|$ and $-h \leq y_2 <0$, we obtain from  (\ref{eq138}) that
     \begin{equation}
        |G_{\mathcal{Q}}(x,y)-G_{Res,c}(x,y)|\leq C \frac{x_2}{|\widetilde{x-y}|^{\frac{3}{2}}}\label{eq141}
     \end{equation}
     for all $x,y$ satisfying $x_2>0$ and $-h\leq y_2<0$, where $C$ is a constant depending only on $k_{\pm}$ and $h$.

     Hence, (\ref{eq140}) and (\ref{eq141}) give that \eqref{eq149} holds for all $x,y$ satisfying $x_2>0$, $-h\leq y_2<0$ and $|\widetilde{x-y}| \geq \delta$, where $C$ is a constant depending only on $k_{\pm}$ and $h$.

{\bf Step 2.2:} we prove that there exists some $\delta>0$ such that \eqref{eq149} holds for all $x,y$ satisfying $x_2< 0$, $0< y_2\leq h$ and $|\widetilde{x-y}|\geq \delta$, where $C$ is a constant depending only on $k_{\pm}$ and $h$.

Suppose $x,y$ satisfy $x_2< 0$ and $0< y_2\leq h$.
By \eqref{hw-eq11} we can write $G_{\mathcal{Q}}$ as
    \begin{equation*}
        G_{\mathcal{Q}}(x,y) = \frac{1}{2\pi}\int_{-\infty}^{+\infty}\frac{e^{ik_{+}y_2 i\mathcal{S}(z,n)}}{\mathcal{S}(z,1)+\mathcal{S}(z,n)}e^{ik_{+}(z(x_1-y_1)-i\mathcal{S}(z,n)x_2)}dz.
    \end{equation*}
    This, together with Remark \ref{hw-re1}, implies that
    \begin{equation*}
        G_{\mathcal{Q}}(x,y) = I(\widetilde{x-y},\tilde{y}),
    \end{equation*}
    where $\tilde{y}=(0,y_2)$ and $\widetilde{x-y}=(x_1-y_1,x_2)$.
   Then it follows from Lemma \ref{thm30} that for $x\in \mathbb{R}^2_{-}$ and $y\in \mathbb{R}^2_{+}$,
   \begin{equation}\label{hw-eq10}
    G_{\mathcal{Q}}(x,y)  =  \frac{e^{ik_{+}|\widetilde{x-y}|}}{\sqrt{|\widetilde{x-y}|}}\frac{e^{i\frac{\pi}{4}}}{\sqrt{8\pi k_{-} }}\frac{2i\sin\theta_{\widetilde{x-y}}}{i\sin \theta_{\widetilde{x-y}}+\mathcal{S}(\cos\theta_{\widetilde{x-y}},1/n)}e^{-k_{-}y_2\mathcal{S}(\cos \theta_{\hat{x}},1/n)}+G_{Res,d}(x,y),
   \end{equation}
   where $G_{Res,d}$ satisfies
   \begin{equation*}
    |G_{Res,d}(x,y)|\leq C|\widetilde{x-y}|^{-\frac{3}{2}},\quad |\widetilde{x-y}|\to +\infty,
   \end{equation*}
   uniformly for all $\theta_{\widetilde{x-y}}\in (\pi,2\pi)$ and $\tilde{y}$ with  $0 < y_2\leq h$ and where $C$ is a constant depending only on $k_{\pm}$ and $h$.
   Thus there exists $\delta>0$ such that
   \begin{equation}
    |G_{Res,d}(x,y)|\leq C|\widetilde{x-y}|^{-\frac{3}{2}}\label{eq146}
   \end{equation}
   for $|\widetilde{x-y}|\geq \delta$, where $C$ is a constant depending only on $k_{\pm}$ and $h$.

On the other hand, similarly to Step 2.1, by \eqref{hw-eq10} we have
    \begin{equation}
        |G_{\mathcal{Q}}(x,y)-G_{Res,d}(x,y)| \leq C\frac{|x_2|}{|\widetilde{x-y}|^{\frac{3}{2}}}\label{eq147}
    \end{equation}
    for $x,y$ satisfying $x_2< 0$ and $0< y_2\leq h$, where $C$ is a constant depending only on $k_{\pm}$ and $h$.

    Hence, it follows from (\ref{eq146}) and (\ref{eq147}) that $G_{\mathcal{Q}}(x,y)$ satisfies (\ref{eq149})
    for all $x,y$ satisfying $x_2< 0$, $0< y_2\leq h$ and $|\widetilde{x-y|}\geq \delta$, where the constant $C$ depends only on $k_{\pm}$ and $h$.

{\textbf{Step 2.3:}} we show that for any $\delta_0,h>0$, there exists a constant $C>0$ depending on $\delta_0$ and $h$ such that \eqref{eq149} holds for $x,y$ satisfying $x_2\cdot y_2<0$, $|y_2|\leq h$ and $|\widetilde{x-y}|\leq \delta_0$.

Recall that $G_{\mathcal{Q}}(x,y) =  (i/4)H_{0}^{(1)}(k_{+}|x-y|)+G_{\mathcal{S}}(x,y)$ for $x\in \mathbb{R}^{2}_{-},y\in \mathbb{R}^{2}_{+}$ and for $x\in  \mathbb{R}^{2}_{+},y\in \mathbb{R}^{2}_{-}$.
By \eqref{hw-eq13}, we can write $G_{\mathcal{S}}$ as $G_{\mathcal{S}}(x,y)=Q(\widetilde{x-y},y_2)$, where $Q(\cdot,\cdot)$ is a function defined on ${\mathbb{R}^2_{+}}\times {\mathbb{R}_{-}}$ and $ {\mathbb{R}^2_{-}}\times {\mathbb{R}_{+}} $ with $\mathbb{R}_{\pm}:=\{x\in \mathbb{R}\,:\,x\gtrless 0\}$.
Using the continuity property of $G_{\mathcal{S}}$ given in Remark \ref{thm35}, we obtain that $Q(\cdot,\cdot)$ can be extended as a function in $ C(\overline{\mathbb{R}^2_{+}}\times \overline{\mathbb{R}_{-}})\cup C(\overline{\mathbb{R}^2_{-}}\times \overline{\mathbb{R}_{+}})$.
Thus we have that for any $\delta_{0}>0$, there exists a constant $C$ depending only on $\delta_{0},h,k_{\pm}$ such that $|G_{\mathcal{S}}(x,y)|\leq C$ for $x,y$ satisfying $x_2\cdot y_2<0$, $|y_2|\leq h$ and $|\widetilde{x-y}|\leq \delta_{0}$.
This, together with the asymptotic properties of the Hankel function $H^{(1)}_0$ for small arguments, implies that there exists a constant $C>0$ such that \eqref{eq149} holds for $x,y$ satisfying $x_2\cdot y_2<0$, $|y_2|\leq h$ and $|\widetilde{x-y}|\leq \delta_0$.

Based on the analysis in Steps 2.1, 2.2 and 2.3, we obtain that $G_{\mathcal{Q}}(x,y)$ satisfies (\ref{eq149}) for all $x,y$ satisfying $x_2\cdot y_2 < 0$ and $|y_2|\leq h$, where the constant $C$ depends only on $k_{\pm}$ and $h$.

\textbf{Part III:} we establish the estimates for $G_\mathcal{P}$ and $G_\mathcal{Q}$ when $k_{+}<k_{-}$.

Define  \begin{equation*}
        G^{*}(x,y):= \begin{cases}
           G_{\mathcal{D},k_{-}}(x,y)+G^{*}_{\mathcal{P}}(x,y),\quad x,y\in \mathbb{R}_{+}^{2}, \\
           G^{*}_{\mathcal{Q}}(x,y),\quad x\in \mathbb{R}^{2}_{-},y\in \mathbb{R}^{2}_{+} \text{ or }x\in \mathbb{R}^{2}_{+},y\in \mathbb{R}^{2}_{-}, \\
           G_{\mathcal{D},k_{+}}(x,y)+G^{*}_{\mathcal{P}}(x,y),\quad x,y\in \mathbb{R}_{-}^{2},
        \end{cases}
        % \label{eq721}
    \end{equation*}
where $G_{\mathcal{P}}^{*}(x,y):=G_{\mathcal{P}}(x',y')$ and $G_{\mathcal{Q}}^{*}(x,y):=G_{\mathcal{Q}}(x',y')$.
Then by \eqref{hw-eq12}, it can be seen that for any $y\in\mathbb{R}^2_+\cup\mathbb{R}^2_-$, $G^{*}(x,y)$ is the two-layered Green function satisfying the scattering problem \eqref{eq166}--\eqref{eq22} with $k=k_{-}$ for $x\in \mathbb{R}^2_{+}$ and $k=k_{+}$ for $x\in \mathbb{R}^2_{-}$.
    Thus, by using the same analysis as in Parts I and II, we can directly obtain that
    \begin{equation}
        |G_{\mathcal{P}}(x',y')|= |G_{\mathcal{P}}^{*}(x,y)| \leq  C(1+|x_2|+|y_2|)|x-y'|^{-\frac{3}{2}}\label{eq182}
    \end{equation}
    for all $x,y$ satisfying $x_2\cdot y_2> 0$ and that
    \begin{equation}
        |G_\mathcal{Q}(x',y')|=|G_{\mathcal{Q}}^{*}(x,y)| \leq C (1+|x_2|)|\widetilde{x-y}|^{-\frac{3}{2}} \label{eq183}
    \end{equation}
    for all $x,y$ satisfying $x_2\cdot y_2< 0$ and $|y_2|\leq h$.
    Hence it follows from \eqref{eq182} that $G_{\mathcal{P}}(x,y)$ satisfies (\ref{eq148}) for all $x,y$ satisfying $x_2\cdot y_2> 0$, where the constant $C$ depends only on $k_{\pm}$. Moreover, it can be seen from \eqref{eq183} that $G_{\mathcal{Q}}(x,y)$ satisfies (\ref{eq149}) for all $x,y$ satisfying $x_2\cdot y_2 < 0$ and $|y_2|\leq h$, where the constant $C$ depends only on $k_{\pm},h$.

    Therefore, the proof is complete.
\end{proof}

\section{Potential Theory}\label{section7}

In this section, we give the properties of the single- and double-layer potentials associated with the two-layered Green function.
Similar properties for the layer potentials associated with the half-space Dirichlet Green function $G_{\mathcal{D},\kappa}$ with $\kappa>0$ have been  established in \cite[Appendix A]{ZC2003}. See also \cite[Appendix A]{CRZ1999} and \cite[Lemmas 4.1--4.3]{CR1996} for the properties of the layer potentials associated with the half-space impedance Green function.
We note that Theorems \ref{thm7}--\ref{thm12} below can be deduced in a very similar way as in \cite[Appendix A]{ZC2003}, due to the definition of the two-layered Green function (see \eqref{eq166}--\eqref{eq22}), the facts that
$G(x,y)-G_{\mathcal{D},k_{-}}(x,y)\in C^{\infty}(\mathbb{R}^2_{-}\times\mathbb{R}^{2}_{-})$ (see \eqref{hw-eq12} and \eqref{eq126}) and $G(x,y)\in C^{\infty}(\mathbb{R}^{2}_{+}\times\mathbb{R}^{2}_{-})$ (see \eqref{eq63} and \eqref{hw-eq11}) as well as Lemma \ref{thm34} and Theorem \ref{thm38}.
Thus, in what follows, we only present Theorems \ref{thm7} and \ref{thm27} with some necessary explanations in the proofs and only present Theorems \ref{thm10}--\ref{thm12} without proofs.
Throughout this section, we assume that $f$ belongs to $B(c_1,c_2)$ with $c_1<0$ and $c_2>0$ and let $\nu$ denote the unit normal on $\Gamma$ pointing to the exterior of $D$.

\begin{theorem}
    Let $W$ be the double-layer potential with the density $\psi\in BC(\Gamma)$, that is,
    \begin{equation}\label{hw-eq14}
        W(x):=\int_{\Gamma}\frac{\partial G(x,y)}{\partial \nu(y)}\psi(y)ds(y), \quad x\in \mathbb{R}^2\backslash \Gamma.
    \end{equation}
    Then the following results hold.

    {\rm (\romannumeral1)} $W$ satisfies $W \in C^2(\mathbb{R}^2\backslash (\Gamma_{0}\cup \Gamma))$, $W|_{\overline{U_{0}}}\in C^{1}(\overline{U_{0}})$,  $W|_{D\backslash U_{0}}\in C^{1}(D\backslash U_{0})$, and  satisfies the Helmholtz equations together with the transmission boundary condition on $\Gamma_{0}$, i.e.,
    \begin{equation*}
    \begin{cases}
    \Delta W+k_{+}^2W=0 \textrm{ in }  U_{0}, \\
    \Delta W+k_{-}^2W=0 \textrm{ in }  \mathbb{R}^2\backslash (\overline{U}_{0}\cap \Gamma), \\
     W|_{+}=W|_{-}, \partial_{2}W\left|_{+} = \partial_{2}W\right|_{-} \textrm{ on } \Gamma_{0}.
    \end{cases}
    \end{equation*}

    {\rm (\romannumeral2)} $W$ can be continuously extended from $D$ to $\overline{D}$ and from $\mathbb{R}^2\backslash \overline{D}$ to $\mathbb{R}^2\backslash D$ with the limiting values
    \begin{equation*}
        W_{\pm}(x)=\int_{\Gamma}\frac{\partial G(x,y)}{\partial \nu(y)}\psi(y)ds(y) \mp \frac{1}{2}\psi(x), \quad x\in \Gamma,
    \end{equation*}
where
\begin{equation}
    W_{\pm}(x) := \lim_{h \to 0+}W(x\mp h\nu(x)), \quad x\in \Gamma.
    \label{eq21}
\end{equation}
    The integral exists in the sense of improper integral.

    {\rm (\romannumeral3)} There exists some constant $C>0$ such that for all $f\in B(c_1,c_2)$ and $\psi\in BC(\Gamma)$,
    \begin{equation*} \sup_{x\in\mathbb{R}^2\backslash \Gamma}\left|(|x_{2}|+1)^{-\frac{1}{2}}W(x)\right|\leq C\|\psi\|_{\infty,\Gamma}.
    \end{equation*}

    {\rm (\romannumeral4)} There holds
        \begin{equation*}
            (\nabla W(x+h\nu(x))-\nabla W(x-h\nu(x)))\cdot \nu(x)\to 0
        \end{equation*}
        as $h\to0$, uniformly for $x$ in compact subsets of $\Gamma$.

    {\rm (\romannumeral5)} $W$ satisfies the upward propagating radiation condition \eqref{eq1} with the wave number $k_{+}$ in $U_{0}$ and the downward propagating radiation condition with the wave number $k_{-}$ in $\mathbb{R}^2\backslash \overline{U}_{f_{-}}$, that is, there exists some $h<f_{-}$ and $\phi\in L^{\infty}(\Gamma_{h})$ such that
    \begin{equation*}
        W(x) = -2\int_{\Gamma_{h}}\frac{\partial \Phi_{k_{-}}(x,y)}{\partial y_2}\phi(y)ds(y),\quad x\in \mathbb{R}^2\backslash \overline{U}_{h}.
    \end{equation*}
    \label{thm7}
\end{theorem}

\begin{proof}
We only prove $W|_{\overline{U_{0}}}\in C^{1}(\overline{U_{0}})$ and  $W|_{D\backslash U_{0}}\in C^{1}(D\backslash U_{0})$, since the other results in this theorem
can be deduced in a very similar way as in \cite[Appendix A]{ZC2003}.
In fact,
for any $x_{0}\in D$, it can be deduced from \eqref{eq189} that
 \begin{equation*}
\nabla W(x_{0})=\int_{\Gamma}\nabla_{x}\frac{\partial G(x_{0},y)}{\partial \nu(y)}\psi(y)ds(y).
 \end{equation*}
 Using \eqref{eq189}, the Lebesgue's dominated convergence theorem as well as the continuity properties of $G$ in Lemma \ref{thm36} {\rm (i)}, we have that for $x_0\in \overline{U_{0}}$,
 \begin{align*}
 \lim_{\substack{x\to x_{0}\\ x\in \overline{U_{0}}}}\nabla W(x)&=\lim_{\substack{x\to x_{0}\\ x\in \overline{U_{0}}}}\int_{\Gamma}\nabla_{x}\frac{\partial G(x,y)}{\partial \nu(y)}\psi(y)ds(y)\\
 &=\int_{\Gamma}\lim_{\substack{x\to x_{0}\\ x\in \overline{U_{0}}}}\nabla_{x}\frac{\partial G(x,y)}{\partial \nu (y)}\psi(y)ds(y)\\
&=\int_{\Gamma}\nabla_{x}\frac{\partial G(x_{0},y)}{\partial \nu(y)}\psi(y)ds(y)\\
&=\nabla W(x_0).
 \end{align*}
 This means that $\nabla W|_{\overline{U_{0}}}\in C(\overline{U_{0}})$. Similarly, we have that
$\nabla W|_{D\backslash U_{0}}\in C(D\backslash U_{0})$.
By similar arguments, we can use Lemma \ref{thm36} {\rm (i)} and Theorem \ref{thm38} to obtain that $W|_{\overline{U_{0}}}\in C(\overline{U_{0}})$ and $W|_{D\backslash U_{0}}\in C(D\backslash U_{0})$. Thus we obtain that $W|_{\overline{U_{0}}}\in C^{1}(\overline{U_{0}})$ and  $W|_{D\backslash U_{0}}\in C^{1}(D\backslash U_{0})$.
\end{proof}

\begin{theorem}
    Let $V$ be the single-layer potential with the density $\psi\in BC(\Gamma)$, that is,
    \begin{equation}\label{hw-eq15}
        V(x):=\int_{\Gamma}G(x,y)\psi(y)ds(y),\quad  x\in \mathbb{R}^2\backslash \Gamma.
    \end{equation}
    Then the following results hold.

    {\rm (\romannumeral1)} $V$ satisfies $V \in C^2(\mathbb{R}^2\backslash (\Gamma_{0}\cup \Gamma))$, $W|_{\overline{U_{0}}}\in C^{1}(\overline{U_{0}})$,  $W|_{D\backslash U_{0}}\in C^{1}(D\backslash U_{0})$ and  satisfies the Helmholtz equations together with the transmission boundary conditions on $\Gamma_{0}$, i.e.,
    \begin{equation*}
    \begin{cases}
    \Delta V+k_{+}^2V=0 \textrm{ in }  U_{0}, \\
    \Delta V+k_{-}^2V=0 \textrm{ in }  \mathbb{R}\backslash (\overline{U}_{0}\cup \Gamma), \\
     V|_{+}=V|_{-}, \partial_{2}V\left|_{+} = \partial_{2}V\right|_{-} \textrm{ on } \Gamma_{0}.
    \end{cases}
    \end{equation*}

    {\rm (\romannumeral2)} $V$ is continuous in $\mathbb{R}^2$ and
    \begin{align}
        V(x)&=\int_{\Gamma}G(x,y)\psi(y)ds(y),\quad x \in \Gamma, \label{eq24}\\
    \frac{\partial V_{\pm}}{\partial \nu}(x)&=\int_{\Gamma}\frac{\partial G(x,y)}{\partial \nu(x)}\psi(y)ds(y) \pm \frac{1}{2}\psi(y),\quad x \in \Gamma,\label{eq23}
    \end{align}
    where
    \begin{equation}
    \frac{\partial V_{\pm}}{\partial \nu}(x):=\lim_{h \to 0+}\nu(x) \cdot \nabla V(x \mp h\nu(x))
    \label{eq20}
    \end{equation}
    and the convergence in \eqref{eq20} is uniform on compact subsets of $\Gamma$. The integrals in (\ref{eq24}) and (\ref{eq23}) exist as improper integrals.

    {\rm (\romannumeral3)} There exists some constant $C>0$ such that for all $f\in B(c_1,c_2)$ and $\psi\in BC(\Gamma)$,
    \begin{equation*} \sup_{x\in\mathbb{R}^2\backslash \Gamma}\left|(|x_{2}|+1)^{-\frac{1}{2}}V(x)\right|< C\|\psi\|_{\infty,\Gamma}.
    \end{equation*}

    {\rm (\romannumeral4)} $V$ satisfies the upward propagating radiation condition \eqref{eq1} with the wave number $k_{+}$ in $U_{0}$ and the downward propagating radiation condition with the wave number $k_{-}$ in $\mathbb{R}^2\backslash \overline{U}_{f_{-}}$, that is, there exists some $h<f_{-}$ and $\phi\in L^{\infty}(\Gamma_{h})$ such that
    \begin{equation*}
        V(x) = -2\int_{\Gamma_{h}}\frac{\partial \Phi_{k_{-}}(x,y)}{\partial y_2}\phi(y)ds(y),\quad x\in \mathbb{R}^2\backslash \overline{U}_{h}.
    \end{equation*}

    \label{thm27}
\end{theorem}

\begin{proof}
Similarly to the proof of Theorem \ref{thm7}, we can use Lemma \ref{thm36} to deduce that $W|_{\overline{U_{0}}}\in C^{1}(\overline{U_{0}})$ and  $W|_{D\backslash U_{0}}\in C^{1}(D\backslash U_{0})$.
The other results in this theorem
can be deduced in a very similar way as in \cite[Appendix A]{ZC2003}.
\end{proof}

\begin{theorem}
    Let $\psi\in BC(\Gamma)$. The direct value of the double-layer potential is defined by
    \begin{equation*}
        W(x):=\int_{\Gamma}\frac{\partial G(x,y)}{\partial \nu(y)}\psi(y)ds(y),\quad x\in \Gamma,
    \end{equation*}
    and the direct value of the single-layer potential is defined by
    \begin{equation*}
  V(x):=\int_{\Gamma}G(x,y)\psi(y)ds(y),\quad x\in \Gamma.
    \end{equation*}
    Then, for any $\lambda\in (0,1)$, both $W(x)$ and $V(x)$ represent uniformly H\"{o}lder continuous functions on $\Gamma$ with the norms
    \begin{equation*}
         \|W\|_{C^{0,\lambda}(\Gamma)},~\|V\|_{C^{0,\lambda}(\Gamma)}\leq C\|\psi\|_{\infty,\Gamma}
    \end{equation*}
    for some constant $C>0$ depending only on $B(c_1,c_2)$ and $k_{\pm}$.
\label{thm10}
\end{theorem}

\begin{theorem}
Let $\psi\in C^{0,\lambda}(\Gamma)$ with $0<\lambda <1$ and let $W(x)$ be given as in \eqref{hw-eq14}. Then
\begin{equation*}
    |\nabla W(x)|\leq C|f(x_1)-x_2|^{\lambda-1}, \quad x\in U_{b_1}\backslash (\overline{U}_{b_2} \cup \Gamma),
\end{equation*}
where $C$ is a positive constant and $b_1=f_{-}-1$, $b_2=0$.
\label{thm11}
\end{theorem}

\begin{theorem}
Let $\psi\in BC(\Gamma)$ and let $V(x)$ be given as in \eqref{hw-eq15}. Then, for $0<\lambda <1$,
\begin{equation*}
|\nabla V(x)|\leq C|f(x_1)-x_2|^{\lambda-1},\quad x\in U_{b_1}\backslash (\overline{U}_{b_2} \cup \Gamma),
\end{equation*}
where $C$ is a positive constant and $b_1=f_{-}-1$, $b_2=0$.
\label{thm12}
\end{theorem}

\section{Integral Operators on the Real Line}\label{section9}

In this section, we introduce an integral equation theory on the real line, associated with the two-layered Green function.
We note that the results in this section are mainly based on the results in \cite[Appendix B]{ZC2003}.
Define the integral equation operator $\mathscr{K}_{l}$ with the kernel $l:\mathbb{R}^2\to \mathbb{C}$ given by
    \begin{equation}
        \mathscr{K}_{l}\psi(s):=\int_{\mathbb{R}}l(s,t)\psi(t)dt,\quad s \in \mathbb{R}.
        \label{eq65}
    \end{equation}
    It can be seen that the integral (\ref{eq65}) exists in a Lebesgue sense for every $\psi\in X:=L^{\infty}(\mathbb{R})$ and $s\in \mathbb{R}$ iff $l(s,\cdot)\in L^{1}(\mathbb{R}), s\in \mathbb{R}$, and that $\mathscr{K}_{l}:X\to Y:=BC(\mathbb{R})$ and is bounded iff $l(s,\cdot)\in L^{1}(\mathbb{R}), s\in \mathbb{R}$,
    \begin{equation}
        |\Vert l \Vert|:=\mathrm{ess} \sup_{s\in\mathbb{R}} \Vert l(s,\cdot) \Vert_{1} < \infty
        \label{eq66}
    \end{equation}
    and $\mathscr{K}_{l}\psi\in C(\mathbb{R})$ for every $\psi\in X$.
    Here, $\|\cdot\|_1$ denotes the $L^1$ norm.

    In the case that (\ref{eq66}) holds, it is convenient to identify $l:\mathbb{R}^2\to \mathbb{C}$ with the mapping $s\to l(s,\cdot)$ in $\mathbf{Z}:=L^{\infty}(\mathbb{R},L^1(\mathbb{R}))$, which mapping is essentially bounded with norm $|\Vert l \Vert|$.
    Let $\mathbf{K}$ denote the set of those functions $l\in \mathbf{Z}$ having the property that $\mathscr{K}_{l}\psi \in C(\mathbb{R})$ for every $\psi\in X$, where $\mathscr{K}_{l}$ is the integral operator (\ref{eq65}).
    Then, $\mathbf{Z}$ is a Banach space with the norm $|\Vert \cdot \Vert|$ and $\mathbf{K}$ is a closed subspace of $\mathbf{Z}$.
    Further, in terms of the above discussions, $\mathscr{K}_{l}:X\to Y$ and is bounded iff $l\in \mathbf{K}$.
    Let $BC(\mathbb{R},L^1(\mathbb{R}))$ denote the set of those functions $l\in \mathbf{Z}$ having the property that for all $s\in \mathbb{R}$,
    \begin{equation*}
        \Vert l(s,\cdot)-l(s',\cdot) \Vert_1 \to 0 \text{ as } s'\to s.
    \end{equation*}
    It is easy to see that $BC(\mathbb{R},L^1(\mathbb{R}))\subset \mathbf{K}$.

    For $(\phi_n)\subset Y$ and $\phi \in Y$,  we say that $(\phi_{n})$ converges strictly to $\phi$ and write $\phi_{n} \stackrel{s}{\rightarrow}\phi$  if $\sup _{n \in \mathbb{N}^{+}}$ $\left\|\phi_n\right\|_{\infty}<\infty$ and $\phi_n(t) \rightarrow \phi(t)$ uniformly on every compact subset of $\mathbb{R}$.
    For $(l_{n})\subset \mathbf{K}$ and $l\in \mathbf{K}$, we say that $(l_{n})$ is $\sigma$-convergence to $l$ and write $l_{n}\stackrel{\sigma}{\to}l$ if $\sup_{n\in \mathbb{N}^{+}} |\Vert l_{n} \Vert|<\infty$ and, for all $\psi\in X$,
    \begin{equation*}
        \int_{\mathbb{R}}l_{n}(s,t)\psi(t)dt\to \int_{\mathbb{R}}l(s,t)\psi(t)dt
    \end{equation*}
    as $n\to\infty$, uniformly on every compact subset of $\mathbb{R}$.

    For $a\in \mathbb{R}$, define the translation operator $T_a:\textbf{Z} \to \textbf{Z}$ by
    \begin{equation*}
        T_a l(s,t)=l(s-a,t-a),\quad s,t\in \mathbb{R}.
    \end{equation*}
    We say that a subset $W\subset \mathbf{K}$ is $\sigma$-sequentially compact in $\mathbf{K}$ if each sequence in $W$ has a $\sigma$-convergent subsequence with its limit in $W$.
    Let $B(Y)$ denote the Banach space of bounded linear operators on $Y$ and let $I$ denote the identity operator on $Y$.

    The following result on the invertibility of $I-\mathscr{K}_{l}$ has been proved in \cite{CZR2000}.

\begin{lemma}
Suppose that $W \subset \mathbf{K}$ is $\sigma$-sequentially compact and satisfies that, for all $s \in \mathbb{R}$,
\begin{equation}
\sup _{l \in W} \int_{\mathbb{R}}\left|l(s, t)-l\left(s', t\right)\right| \mathrm{d} t \rightarrow 0 \quad \text { as } s' \rightarrow s,
\label{eq29}
\end{equation}
that $T_{a}(W)=W$ for some $a \in \mathbb{R}$, and that $I-\mathscr{K}_{l}$ is injective for all $l \in W$. Then $\left(I-\mathscr{K}_{l}\right)^{-1}$ exists as an operator on the range space $\left(I-\mathscr{K}_{l}\right)(Y)$ for all $l \in W$ and
$$
\sup _{l \in W}\left\|\left(I-\mathscr{K}_{l}\right)^{-1}\right\|<\infty.
$$
If also, for every $l \in W$, there exists a sequence $\left(l_{n}\right) \subset W$ such that $l_{n} \stackrel{\sigma}{\rightarrow} l$ and, for each $n$, it holds that
\begin{equation}
I-\mathscr{K}_{l_{n}} \text { injective } \Rightarrow I-\mathscr{K}_{l_{n}} \text { surjective, }
\label{eq34}
\end{equation}
then also $I-\mathscr{K}_{l}$ is surjective for each $l \in W$ so that $\left(I-\mathscr{K}_{l}\right)^{-1} \in B(Y)$.
\label{thm15}
\end{lemma}

The following three lemmas give the properties of  $k_{f}$ and  $k_{\tilde{\beta},f}$, which are defined in \eqref{eq162} and \eqref{eq163}, respectively.
Due to the properties of the two-layered Green function given in Section \ref{section8},  we can deduce Lemmas \ref{thm16}, \ref{thm17} and \ref{thm18}
in a very similar manner as in  \cite[Appendix B]{ZC2003}. Thus we only present these lemmas without proofs.

\begin{lemma}
    Assume $c_1<0$, $c_2>0$, $d_1\geq 0$, $d_2>0$,
    and $\omega:[0,\infty)\to[0,\infty)$ is a function such that $\omega(s)\to 0$ as $s\to 0$.
    Let $\kappa\in L^1(\mathbb{R})$ be defined by
    \begin{equation*}
        \kappa(s):=
        \begin{cases}
            1-{\rm log}|s|,\quad 0<|s|\leq 1,\\
            |s|^{-3/2},\quad |s|>1.
        \end{cases}
    \end{equation*}

{\rm (\romannumeral1)} For all $f\in B(c_1,c_2)$,
    \begin{equation*}
        \left|\kappa_{f}(s,t)\right|\leq C\left|\kappa(s-t)\right|,\quad s,t \in \mathbb{R},\;s\neq t,
    \end{equation*}
    for some constant $C>0$ depending only on $c_1$, $c_2$, $\eta$, and $k_{\pm}$, and
    \begin{equation*}
        \sup_{|s_1-s_2|
        \leq h,f\in B(c_1,c_2)}\int_{\mathbb{R}}\left|\kappa_{f}(s_1,t)-\kappa_{f}(s_2,t)\right|dt\to 0
    \end{equation*}
    as $h\to0$.

{\rm (\romannumeral2)} For all $f \in B(c_1,c_2), \tilde{\beta} \in E(d_1,d_2,\omega)$,
\begin{equation*}
\left|\kappa_{\tilde{\beta}, f}(s, t)\right| \leq C\left|\kappa(s-t)\right|, \quad s, t \in \mathbb{R}, \quad s \neq t,
\end{equation*}
for some constant $C>0$ depending only on $c_{1}, c_{2}, d_{1}, d_{2}$, and $k_{\pm}$, and
\begin{equation*}
\sup\limits_{\substack{\left|s_{1}-s_{2}\right| \leq h, f \in B(c_1,c_2),\\ \tilde{\beta} \in E(d_1,d_2,\omega)}} \int_{\mathbb{R}}\left|\kappa_{\tilde{\beta}, f}\left(s_{1}, t\right)-\kappa_{\tilde{\beta}, f}\left(s_{2}, t\right)\right| \mathrm{d} t \rightarrow 0
\end{equation*}
as $h \rightarrow 0$.
\label{thm16}
\end{lemma}

\begin{lemma}Assume $c_1<0$, $c_2>0$. Then we have the following statements.

    {\rm (i)} Every sequence $\left(f_{n}\right) \subset B(c_1,c_2)$ has a subsequence $\left(f_{n_{m}}\right)$ such that $f_{n_{m}} \stackrel{s}{\rightarrow} f$, $f_{n_{m}}^{\prime} \stackrel{s}{\rightarrow} f^{\prime}$, with $f \in B(c_1,c_2)$.

    {\rm (ii)} Suppose that $\left(f_{n}\right) \subset B(c_1,c_2)$ and that $f_{n} \stackrel{s}{\rightarrow} f$, $f_{n}' \stackrel{s}{\rightarrow} f'$, with $f \in B(c_1,c_2)$. Then $\kappa_{f_{n}} \stackrel{\sigma}{\rightarrow} \kappa_{f}$.
    \label{thm17}
\end{lemma}

\begin{lemma} Assume $c_1<0$, $c_2>0$, $d_1\geq 0$, $d_2>0$, and $\omega:[0,\infty)\to[0,\infty)$ is a function such that $\omega(s)\to 0$ as $s\to 0$. Then we have the following statements.

{\rm (i)} Every sequence $\big(\tilde{\beta}_{n}\big) \subset E(d_1,d_2,\omega)$ has a subsequence $\big(\tilde{\beta}_{n_{m}}\big)$ such that $\tilde{\beta}_{n_{m}} \stackrel{s}{\rightarrow} \tilde{\beta}$ with $\tilde{\beta} \in E(d_1,d_2,\omega)$.

{\rm (ii)} If $\big(f_{n}\big) \subset B(c_1,c_2)$, $\big(\tilde{\beta}_{n}\big) \subset E(d_1,d_2,\omega)$ and $f_{n} \stackrel{s}{\rightarrow} f$, $f_{n}^{\prime} \stackrel{s}{\rightarrow} f^{\prime}$, $\tilde{\beta}_{n} \stackrel{s}{\rightarrow} \tilde{\beta}$, with $f \in B(c_1,c_2)$ and $\tilde{\beta} \in E(d_1,d_2,\omega)$, then $\kappa_{\tilde{\beta}_{n}, f_{n}} \stackrel{\sigma}{\rightarrow} \kappa_{\tilde{\beta}, f}$.
\label{thm18}
\end{lemma}

\end{appendices}

% \newpage

% \bibliographystyle{wileyNJD-AMS.bst}
% \bibliography{references.bib}

\end{document}